\documentclass[oneside,12pt,letterpaper]{book}

\usepackage{amssymb}
\usepackage{amsthm}
\usepackage{amsmath}
\usepackage{anysize}
\usepackage{hyperref}
\usepackage{bbm}
\usepackage{enumitem}
\usepackage{xcolor}
\usepackage[style=alphabetic,maxnames=99,maxalphanames=5, isbn=false, giveninits=true, doi=false, url=false]{biblatex}

 \AtBeginBibliography{\footnotesize}
 
\bibliography{Biblio}

\marginsize{2.6cm}{2.6cm}{2cm}{2cm}
\newtheorem{theorem}{Theorem}[chapter]
\newtheorem{lemma}[theorem]{Lemma}
\newtheorem{prop}[theorem]{Proposition}
\newtheorem{coro}[theorem]{Corollary}
\newtheorem{corollary}[theorem]{Corollary}

\theoremstyle{definition}
\newtheorem{definition}[theorem]{Definition}

\newtheorem{remark}[theorem]{Remark}

\newcommand{\PSH}{{\rm PSH}}
\newcommand{\vol}{{\rm Vol}}

\newcommand{\Ec}{\mathcal{E}}

\newcommand{\setdef}{\ \vert \ }
\newcommand{\vep}{\varepsilon}
\newcommand{\psh}{{\rm PSH}}

\newcommand{\capa}{{\rm Cap}}

\newcommand{\AM}{{\rm I}}
\newcommand{\AMO}{{\rm I}_{\phi}}

\newcommand{\Amp}{{\rm Amp}}
\newcommand{\id}{{\bf 1}}

\usepackage{hyperref}
\hypersetup{
    bookmarks=true,         
    unicode=false,          
    pdftoolbar=true,        
    pdfmenubar=true,        
    pdffitwindow=false,     
    pdfstartview={FitH},    
    pdftitle={Relative pluripotential theory on compact K\"ahler manifolds},    
    pdfauthor={T. Darvas, E. Di Nezza, C.H. Lu},     
    colorlinks=true,       
   linkcolor=black,          
    citecolor=black,        
    filecolor=black,      
    urlcolor=black}           

\numberwithin{equation}{chapter}
\numberwithin{theorem}{chapter}
\title{\textbf{\LARGE{Relative pluripotential theory on compact K\"ahler manifolds}}}
\author{\Large{Tam\'as Darvas (University of Maryland)} \vspace{0.1cm} \\ \Large{Eleonora Di Nezza (Sorbonne Université)} \vspace{0.1cm}\\ \Large{Chinh H. Lu (Université d'Angers)\vspace{0.5cm}}}
\date{\emph{\large{To the memory of Jean--Pierre Demailly (1957-2022)}}}

\begin{document}
\maketitle 

\chapter*{\centering \begin{normalsize}Abstract\end{normalsize}}
\thispagestyle{empty}
\begin{quotation}
\noindent 
Given a compact K\"ahler manifold, we survey the study of complex Monge--Amp\`ere type equations with prescribed singularity type, developed by the authors in a series of papers. In addition, we give a general answer to a question of Guedj--Zeriahi about the finite energy range of the complex Monge--Amp\`ere operator. 
\end{quotation}

\clearpage

\tableofcontents 
\setcounter{page}{1}

\chapter{Main results}

Let $(X,\omega)$ be a compact K\"ahler manifold of complex dimension $n$, and let $\theta$ be another smooth K\"ahler form. We consider the complex Monge--Amp\`ere equation
\begin{equation}\label{eq: CMAE_intr}
(\theta + dd^c u)^n = f \omega^n.
\end{equation}
When $f$ is smooth and positive, Yau \cite{Yau78}, proved that this equation admits a unique smooth solution, resolving the famous Calabi conjecture. This landmark result is one of the cornerstones for studying canonical K\"ahler metrics.

One can consider conical/edge behaviour along a divisor for solutions of \eqref{eq: CMAE_intr}. This is relevant in birational geometry and in the study of K-stability. It has been explored in many works, including \cite{Don11, JMR16, GP16}.  \\
\indent When $f \geq 0$, $f \in L^p(\omega^n), \ p > 1,$  and $u$ is allowed to be bounded $\theta$-plurisubharmo\-nic, a solution to \eqref{eq: CMAE_intr} exists by deep estimates of Ko{\l}odziej \cite{Kol98,Kol03IUMJ}. Furthermore, if one replaces $f \omega^n$ by a suitable Radon measure, it was shown in \cite{GZ07} that solutions to \eqref{eq: CMAE_intr}  exist in a suitable finite energy space, mirroring results of Cegrell \cite{Ceg98} in the local case.

When the cohomology class $\{\theta\}$ is allowed to be degenerate, solutions to 
 \eqref{eq: CMAE_intr} and related equations have been explored in \cite{EGZ09,  BEGZ10,BBGZ13}.

There have been numerous other extensions and generalizations in recent years. In this paper, we survey recent work done in \cite{DDL2,DDL4}, where we consider solutions of \eqref{eq: CMAE_intr} that have a \emph{prescribed singularity profile}. We also give a general answer to a question raised by Guedj--Zeriahi \cite{GZ07}.

To fix notation and terminology, we assume that $\theta$ is a closed smooth real $(1,1)$-form such that $\{\theta\}$ is big, and let  $\phi \in \textup{PSH}(X,\theta)$. Let $f \geq 0$, $f \in L^p(\omega^n), \ p > 1$.

We are looking for solutions $u \in \textup{PSH}(X,\theta)$ of \eqref{eq: CMAE_intr} satisfying
\begin{equation}\label{eq: min_sing_type}
\phi - C \leq u \leq \phi+C, \textup{ for some }C \in \mathbb R.
\end{equation}
If $\varphi$ and $\varphi'$ are two $\theta$-plurisubharmo\-nic functions on $X$, then $\varphi'$ is said to be \emph{less singular} than $\varphi$, i.e. $\varphi \preceq \varphi'$, if they satisfy $\varphi\le\varphi'+C$ for some $C\in \mathbb{R}$.
We say that $\varphi$ has the same singularity as $\varphi'$, i.e. $\varphi \simeq \varphi'$, if $\varphi \preceq \varphi'$ and $\varphi' \preceq \varphi$. The latter condition is easily seen to yield an equivalence relation, whose equivalence classes are denoted by $[\varphi]$, $\varphi \in \psh(X, \theta)$. Using this terminology \eqref{eq: min_sing_type} simply says that $[u]=[\phi]$.

Typically $\phi$ is unbounded, hence so is $u$, and the left-hand-side of \eqref{eq: CMAE_intr} has to be interpreted as the 
\emph{non-pluripolar complex Monge-Amp{\`e}re measure} of $u$, considered in \cite{BEGZ10} (see \eqref{eq: BEGZ10_def} below):
$$
\theta_u^n:=\langle(\theta+dd^c u)^n\rangle.
$$

As it turns out, this problem is well posed only for potentials $\phi$ having ``model'' singularity type $[\phi]$, that includes the case of analytic singularities, treated in \cite{PS14}.

We need to consider the following notion of envelope, only dependent on the singularity type $[\phi]$:
$$P_\theta[\phi] = \sup\{v \in \textup{PSH}(X,\theta) \textup{ s.t. } v \leq 0, \ [\phi] = [v]\}.$$
Since $\phi - \sup_X \phi \leq P_\theta[\phi]$, we have that $[\phi] \leq [P_\theta[\phi]]$ and typically equality does not happen. When $[\phi] = [P_\theta[\phi]]$, we say that $\phi$ has \emph{model singularity type}.

We state our first main result, initially proved in \cite{DDL4}:

\begin{theorem}\label{thm4} Suppose $\phi \in \textup{PSH}(X,\theta)$  and $[\phi] = [P_\theta[\phi]]$. Let  $f \in L^p(\omega^n)$, $p > 1$ such that $f \geq 0$  and $\int_X f \omega^n=\int_X \theta_{\phi}^n > 0$. Then the following hold:\\
(i) There exists  $u\in \textup{PSH}(X,\theta)$, unique up to a constant, such that $[u]=[\phi]$ and 
\begin{equation}\label{eq: main_thm_eq1}
\theta_u^n=f \omega^n.
\end{equation}
(ii) For any $\lambda >0$  there exists a unique $v\in \textup{PSH}(X,\theta)$,  such that $[v]=[\phi]$ and
\begin{equation}\label{eq: main_thm_eq2}
\theta_v^n=e^{\lambda v}f \omega^n.
\end{equation}
\end{theorem}

As it turns out, the condition $[\phi] = [P[\phi]]$ is very natural and necessary for well posedness in this context. Due to Theorem \ref{thm: naturality_of_model}, if \eqref{eq: main_thm_eq1} has a unique solution for any choice of $f \in L^\infty$, then the condition $[\phi] = [P[\phi]]$ must hold.

To prove this theorem, one needs to build up the variational approach of \cite{BEGZ10} \emph{relative to} the model singularity type $[\phi]$, motivating the title of this survey.

When $K_X>0$ and $\lambda =1$ (or $K_X$ trivial and $\lambda=0$) solutions of \eqref{eq: main_thm_eq2} can be interpreted as K\"ahler--Einstein metrics with prescribed singularity.

\begin{remark}\label{rem: examples}  Potentials $\phi$ with \emph{analytic singularity type} can be locally written as 
$$\phi = \frac{c}{2} \log \big(\sum_j |f_j|^2\big) + g,$$ 
where $f_j$ are holomorphic, $c \in \Bbb Q_+$ and $g$ is bounded. Such potentials are always of model type (\cite[Remark 4.6]{RWN14}, \cite{RS05}, see also Proposition \ref{prop: analytic example}). In particular,  discrete logarithmic singularity types are of model type, making connection with pluricomplex  Green currents \cite{CG09,PS14,RS05}. 

Our reader may wonder if there are other interesting enough potentials with model singularity type. We believe this to be the case:\vspace{0.1cm}

$\bullet$ By Theorem \ref{thm: ceiling coincide envelope non collapsing} below, $P_\theta[\psi]=P_\theta[P_\theta[\psi]]$ for any $\psi \in \textup{PSH}(X,\theta)$ with $\int_X \theta_\psi^n >0$. In particular, $P_\theta[\psi]$ is a model potential, giving an abundance of potentials with model type singularity.\vspace{0.1cm}

$\bullet$ By Proposition \ref{prop: Lp example} below, if $\psi \in \textup{PSH}(X,\theta)$ has small unbounded locus, and $\theta_\psi^n/ \omega^n  \in L^p(\omega^n), \ p > 1$ with $\int_X  \theta_\psi^n >0$, then $\psi$ has model type singularity.\vspace{0.1cm}

$\bullet$ Due to \cite{RWN14,DDL3}, potentials with model type singularity naturally arise as degenerations along geodesic rays and in particular along test configurations.
\end{remark}

As an application to Theorem \ref{thm4} we prove the log-concavity property for non-pluripolar products, initially conjectured in \cite{BEGZ10}, proved in \cite{DDL4}.

\begin{theorem}\label{thm5}
	Let $T_1,...,T_n$ be positive closed $(1,1)$-currents  on a compact K\"ahler manifold $X$.  Then 
	\[
	\int_X \langle T_1 \wedge \ldots \wedge T_n\rangle \geq \left(\int_X \langle T_1^n\rangle \right)^{\frac{1}{n}} \ldots \left(\int_X \langle T_n^n\rangle \right)^{\frac{1}{n}}.
	\] 
In particular, $T \to  \log \int_X \langle T^n \rangle$ is concave function on the space of positive $(1,1)$-currents.

\end{theorem}

Next we consider the situation when the right hand side of \eqref{eq: CMAE_intr} is  a non-pluripolar measure (not necessarily of the type $f\omega^n$), and answer a question of Guedj--Zeriahi in our relative context.

For this we need to introduce relative finite energy classes. Let $\phi \in \textup{PSH}(X,\theta)$ such that $\phi = P[\phi]$ and $\int_X \theta_\phi^n>0$.

By a foundational result of Witt Nystr\"om \cite{WN19}, if $[u] \leq [\phi]$ then $\int_X \theta_u^n \leq \int_X \theta_\phi^n$ (see Theorem \ref{thm: BEGZ_monotonicity_full} below). We say that $u$ has relative full mass with respect to $\phi$ (notation: $u \in \mathcal E(X,\theta,\phi)$) if this inequality is extremized, i.e., $[u] \leq [\phi]$ and $\int_X \theta_u^n = \int_X \theta_\phi^n$.

It is argued in Theorem \ref{thm: existence lambda =0} that for any Radon measure $\mu$ not charging pluripolar sets such that $\mu(X)=\int_X \theta_\phi^n $, there exists a unique solution $u \in \mathcal E(X,\theta,\phi))$ to the equation
\begin{equation}\label{eq: CMAE_intr_Borel}
\theta_u^n = \mu.
\end{equation}

As pointed out in \cite{GZ07}, for potential theoretic reasons, it is natural to consider weighted subspaces of $\mathcal E(X,\theta,\phi)$.
  A weight is a continuous increasing function  $\chi: [0,\infty) \rightarrow [0,\infty)$ such that $\chi(0)=0$ and $\lim_{t \to \infty}\chi(t)=\infty$. 
  
Let us assume that the weight $\chi$ satisfies the following condition 
 \begin{equation}
 	\label{eq: growth condition_intr}
 	\forall t \geq 0, \; \forall \lambda \geq 1, \; \chi(\lambda t) \leq \lambda^M \chi(t),
 \end{equation}
 where $M\geq 1$ is a fixed constant. 
Let $\mathcal{E}_{\chi}(X,\theta,\phi)$ denote the set of all $u\in \mathcal{E}(X,\theta,\phi)$ such that 
\[
E_{\chi}(u,\phi):=\int_X \chi(|u-\phi|) \theta_u^n <\infty. 
\]
When $\phi=V_{\theta}$, we denote $\mathcal{E}(X,\theta)=\mathcal{E}(X,\theta,V_{\theta})$, $\mathcal{E}_{\chi}(X,\theta)=\mathcal{E}_{\chi}(X,\theta,V_{\theta}$) and $E_{\chi}(u)=E_{\chi}(u,V_{\theta})$. Compared to \cite{GZ07},  we have changed the sign of the weight, but the weighted classes are the same.  Also, notice that both low and high energy classes of \cite{GZ07} satisfy the condition \eqref{eq: growth condition_intr}, allowing our treatment to be a bit more universal.

One may ask, under what condition is the solution $u \in \mathcal E(X,\theta,\phi)$ of \eqref{eq: CMAE_intr_Borel} an element of $\mathcal E_\chi(X,\theta,\phi)$. The theorem below provides precise answers to this question, containing perhaps the only novel result of this work:

\begin{theorem}
	\label{thm: E chi characterization_intro} Le $\mu$ be a  Radon measure with $\mu(\{\phi = -\infty\})=0$ and $\int_X \theta_{\phi}^n =\mu(X)>0$. For $\chi$ satisfying \eqref{eq: growth condition_intr} the following conditions are equivalent:
	\begin{enumerate}
		\item[(i)] There exists a constant $C>0$ such that, for all $u \in \mathcal{E}_{\chi}(X,\theta,\phi)$ with $\sup_X u=0$, we have 
		\[
		\int_X \chi(\phi-u) d\mu \leq C E_{\chi}(u,\phi)^{M/(M+1)} +C.  
		\]
	\item[(ii)] $\chi(|\phi-u|) \in L^1(\mu)$, for all $u\in \mathcal{E}_{\chi}(X,\theta,\phi)$.
	\item[(iii)] $\mu= \theta_\varphi^n$, for some $\varphi \in \mathcal{E}_{\chi}(X,\theta,\phi)$, with $\sup_X \varphi=0$. 
	\end{enumerate} 
\end{theorem}

For slightly less general $\chi$, the equivalence between (i) and (iii) was obtained in \cite{DV21}. The condition $\mu(\{\phi = -\infty\})=0$ is necessary. Without it, the $\mu$-integrability of $\chi(\phi-u)$ can not be discussed, as the values of this function are not defined on the set $\{\phi = -\infty\}$. Of course this condition is vacuous if $\phi =0$.

In the particular case when $\theta$ is K\"ahler and $\phi=0$, the equivalence between (ii) and (iii) answers a question asked by Guedj--Zeriahi in the comments following \cite[Theorem 4.1]{GZ07}. In \cite[Theorem C]{GZ07} the authors obtain the above result for $\chi(t) = t^p, \ p \geq 1$.

\paragraph{Prerequisites.}
 An effort has been made to keep prerequisites at a minimum. However
due to size constraints, such requirements on part of the reader are inevitable. We assume
that our reader is familiar with the basics of Bedfor-Taylor theory \cite{BT76,BT82}, and finite energy pluripotential theory in the big case, as elaborated in \cite{BEGZ10}. The recent book \cite{GZbook} is a comprehensive source that can initiate novice readers into this subject.

\paragraph{Organization.} In Chapter 2 we recall terminology of finite energy pluripotential theory from \cite{BEGZ10} and prove some preliminary results, touching \cite{DDL1, DNT21} in the process. In Chapter 3 we prove the monotonicity theorem due to Witt Nystr\"om \cite{WN19} and the authors \cite{DDL2} for non-pluripolar product masses. Here we also introduce and study the basic concepts of relative pluripotential theory.

In Chapter 4 we prove a result about comparison of capacities and prove an integrating by parts formula, making contact with \cite{Xia19, Lu21,Vu21}. 

In Chapter 5, we finally solve the complex Monge--Amp\`ere equations \eqref{eq: main_thm_eq1} and \eqref{eq: main_thm_eq2} using the variational method of \cite{BBGZ13} as in \cite{DDL2}. Due to the availability of integration parts, the technical assumption of small unbounded locus from \cite{DDL2} can be avoided, making our treatment here more transparent, compared to the original arguments in \cite{DDL2,DDL4}. 

The proof of Theorem \ref{thm: E chi characterization_intro} is given in Chapter 6, containing the novel results of this survey.

\paragraph{Relation to other works.} Since \cite{DDL1,DDL2,DDL3,DDL4} appeared, a number of works have taken up the topics of this survey, and developed it further. Due to size constraints we can not treat these here, but let us mention a few exciting directions.

The work \cite{DDL5} introduced a pseudo-metric on the space of singularity types and studied the stability of solutions to \eqref{eq: main_thm_eq1}, as the singularity type is varied. More precise results have been recently obtained in \cite{DoVu22}. Connections with Lelong numbers have been studied in \cite{Vu20}.

The works \cite{DX22, DZ22, DXZ23} used relative pluripotential theory to study partial Bergman kernels, K-stability and the Ross--Witt Nystr\"om correspondence.

The works \cite{Xia19b,Trus19,Gup22} have explored the metric geometry of the relative finite energy classes surveyed in this work.

Finally,  A. Trusiani started an elaborate study of Kahler--Einstein metrics with prescribed singularity type in the Fano case \cite{Tru1,Tru2}.

\paragraph{Acknowledgments.} The first named author was partially supported by an Alfred P. Sloan Fellowship and National Science Foundation grant DMS--1846942. The second author is supported by the project SiGMA ANR-22-ERCS-0004-02. The third named author is partially supported by Centre Henri Lebesgue ANR-11-LABX-0020-01. The second and third named authors are partially supported by the project PARAPLUI  ANR-20-CE40-0019.

\chapter{Preliminaries}

We recall results from pluripotential theory  of big cohomology classes, particularly results about non-pluripolar complex Monge--Ampere measures. We borrow notation and terminology from \cite{BEGZ10}, and we refer to this work for further details. 

Let $(X,\omega)$ be a compact K\"ahler manifold of dimension $n$. Let $\theta$ be a smooth closed $(1,1)$-form on $X$ such that $\{\theta\}$ is \emph{big}, i.e., there exists $\psi \in {\rm PSH}(X,\theta)$ such that $\theta +dd^c \psi \geq \vep \omega$ for some small constant $\vep>0$. Here, $d$ and $d^c$ are real differential operators defined as $d:=\partial +\bar{\partial},\,  d^c:=\frac{i}{2\pi}\left(\bar{\partial}-\partial \right).$ A function $\varphi: X\rightarrow \mathbb{R}\cup\{-\infty\}$ is quasi-plurisubharmonic (qpsh)  if it can be locally written as the sum of a plurisubharmonic function and a smooth function. $\varphi$ is called $\theta$-plurisubharmonic  ($\theta$\emph{-psh}) if it is qpsh and $\theta+dd^c \varphi\geq 0$ in the sense of currents. We let $\psh(X,\theta)$ denote the set of $\theta$-psh functions that are not identically $-\infty$.

A $\theta$-psh function $\varphi$ is said to have \emph{analytic singularity type} if there exists a constant $c>0$ such that locally on $X$,
$$
\varphi=\frac{c}{2}\log\sum_{j=1}^{N}|f_j|^2+g,
$$
where $g$ is bounded and $f_1,\dots,f_N$ are local holomorphic functions. The \emph{ample locus} $\Amp(\{\theta\})$ of $\{ \theta\}$ is the set of points $x\in X$ such that there exists a K\"ahler current $T\in \{ \theta \}$ with analytic singularity type and smooth in a neighbourhood of $x$. The ample locus $\Amp(\{\theta\})$ is a Zariski open subset, and it is nonempty \cite{Bou04}.

Let $x \in X$. Fixing a holomorphic chart $x \in U \subset X$, the \emph{Lelong number} $\nu(\varphi,x)$ of $\varphi \in \textup{PSH}(X,\theta)$ is defined as follows:
\begin{equation}\label{eq: Lelongdef}
\nu(\varphi,x) = \sup\{ \gamma \geq 0 \textup{ s.t. } \varphi(z) \leq \gamma \log \|z-x \| + O(1) \textup{ on } U\}.
\end{equation}
One can also associate to $\varphi$ a \emph{multiplier ideal sheaf} $\mathcal I (\varphi)$ whose germs are holomorphic functions $f$ for which $|f|^2e^{-\varphi}$ is integrable.

If $\varphi$ and $\varphi'$ are two $\theta$-psh functions on $X$, then $\varphi'$ is said to be \emph{less singular} than $\varphi$, i.e. $\varphi \preceq \varphi'$, if they satisfy $\varphi\le\varphi'+C$ for some $C\in \mathbb{R}$.
We say that $\varphi$ has the same singularity as $\varphi'$, i.e. $\varphi \simeq \varphi'$, if $\varphi \preceq \varphi'$ and $\varphi' \preceq \varphi$. The latter condition is easily seen to yield an equivalence relation, whose equivalence classes are denoted by $[\varphi]$, $\varphi \in \psh(X, \theta)$.

A $\theta$-psh function $\varphi$ is said to have \emph{minimal singularity type} if it is less singular than any other $\theta$-psh function.  Such $\theta$-psh functions with minimal singularity type always exist, one can consider for example
\begin{equation*}
V_\theta:=\sup\left\{ \varphi\,\,\theta\text{-psh}, \varphi\le 0\text{ on } X \right \}. 
\end{equation*}
Trivially, a $\theta$-psh function with minimal singularity type is locally bounded  in $\Amp(\{\theta\})$. It follows from \cite[Theorem 1.1]{DNT23} that $V_{\theta}$ is $C^{1, \bar{1}}$ in the ample locus ${\rm Amp}(\{\theta\})$.

Given $\theta^1+dd^c\varphi_1,..., $ $ \theta^p+dd^c \varphi_p$   positive $(1,1)$-currents, where $\theta^j$ are closed smooth real $(1,1)$-forms, following the construction of Bedford-Taylor \cite{BT87} in the local setting, it has been shown in \cite{BEGZ10} that the sequence of currents 
\[
{\bf 1}_{\bigcap_j\{\varphi_j>V_{\theta_j}-k\}}(\theta^1+dd^c \max(\varphi_1, V_{\theta_1}-k))\wedge...\wedge (\theta^p+dd^c\max(\varphi_p, V_{\theta_p}-k))
\] 
is non-decreasing in $k$ and converges weakly to the so called \emph{non-pluripolar product} 
\begin{equation}\label{eq: BEGZ10_def}
\langle \theta^1_{\varphi_1 } \wedge\ldots\wedge\theta^p_{\varphi_p}\rangle .
\end{equation}
In the following, with a slight abuse of notation, we will denote the non-pluripolar product simply by $\theta^1_{\varphi_1 } \wedge\ldots\wedge\theta^p_{\varphi_p}$.
When $p =n$, the resulting positive $(n,n)$-current is a Borel measure that does not charge pluripolar sets. Pluripolar sets are Borel measurable sets that are contained in some set $\{\psi = -\infty\}$, where $\psi \in \textup{PSH}(X,\theta)$.

For a $\theta$-psh function $\varphi$, the \emph{non-pluripolar complex Monge-Amp{\`e}re measure} of $\varphi$ is
$$
\theta_\varphi^n:=\langle(\theta+dd^c\varphi)^n\rangle.
$$

The volume of a big class $\{ \theta\}$  is defined by 
\[
{\rm Vol}(\{\theta\}):= \int_{{\rm Amp}(\{\theta\})} \theta_{V_\theta}^n.
\]

Alternatively, by \cite[Theorem 1.16]{BEGZ10}, in the above expression one can replace $V_\theta$ with any $\theta$-psh function with minimal singularity type. A $\theta$-psh function $\varphi$ is said to have \emph{full Monge--Amp\`ere mass} if
\[
\int_X \theta_\varphi^n={\rm Vol}(\{\theta\}),
\]
and we then write $\varphi\in \mathcal{E}(X,\theta)$.

An important property of the non-pluripolar product is that it is local with respect to the plurifine topology (see  \cite[Corollary 4.3]{BT87},\cite[Section 1.2]{BEGZ10}).  
This topology is the coarsest such that all qpsh functions with values in $\mathbb{R}$ are continuous. For convenience we record the following version of this result for later use. 
\begin{lemma} \label{lem: plurifine}
Fix closed smooth big $(1,1)$-forms $\theta^1,...,\theta^n$.  Assume that $\varphi_j,\psi_j,j=1,...,n$ are $\theta^j$-psh functions such that $\varphi_j =\psi_j$ on $U$ an open set in the plurifine topology. Then 
$$
\id_{U} \theta^1_{\varphi_1} \wedge ... \wedge \theta^n_{\varphi_n} = \id_{U} \theta^1_{\psi_1} \wedge ... \wedge \theta^n_{\psi_n}.
$$
\end{lemma}
Lemma \ref{lem: plurifine} will be referred to as the \emph{plurifine locality property}. We will often work with sets of the form $\{u<v\}$, where $u,v$ are quasi-psh functions. These are always open in the plurifine topology. 

\paragraph{Convergence theorems.} The Monge--Amp\`ere capacity of a Borel set $E\subset X$ is defined as 
\[
{\rm Cap}_{\omega}(E):= \sup \left\{\int_E (\omega+dd^c u)^n\; :\; u\in {\rm PSH}(X,\omega), \; -1\leq u\leq 0\right\}. 
\]
A function $u$ is called quasi-continuous if for each $\varepsilon>0$, there exists an open set $U$ such that ${\rm Cap}_{\omega}(U)<\varepsilon$ and the restriction of $u$ on $X\setminus U$ is continuous.

A sequence of functions $u_j$ converges in capacity to $u$ if, for any $\delta>0$,
\[
\lim_{j\to \infty} \capa_{\omega}(\{x\in X\; :\;
|u_j(x)-u(x)|>\delta\}) =0.
 \]

We recall a classical convergence theorem from Bedford-Taylor theory. We refer to \cite[Theorem 4.26]{GZbook} for a proof of this result, which is a slight generalization of \cite[Theorem 1]{X96}.  

\begin{prop}\label{prop: xing_conv} Let $U \subset \mathbb C^n$ be an open set. Suppose $\{f_j\}_j$ are uniformly bounded quasi-continuous
functions which converge in capacity to another quasi-continuous function $f$ on $U$. Let $\{ u^j_1\}_j, \{ u^j_2\}_j,\ldots, \{ u^j_n\}_j$ be uniformly bounded plurisubharmonic functions on $\Omega$, converging in capacity to $u_1, u_2, \ldots, u_n$
respectively. Then we have the following weak convergence of measures:
$$f_j i\partial \bar \partial u_1^j \wedge i\partial \bar \partial u_2^j \wedge \ldots \wedge i\partial \bar \partial u_n^j \to f i\partial \bar \partial u_1 \wedge i\partial \bar \partial u_2 \wedge \ldots \wedge i\partial \bar \partial u_n.$$
\end{prop}

\begin{definition}
A Borel set $E$ is called quasi-open if for each $\varepsilon>0$, there exists an open set $U$ such that 
\[
{\rm Cap}_{\omega}((U\setminus E) \cup (E\setminus U)) \leq \varepsilon. 
\]
Quasi-closed sets are defined similarly. 
\end{definition}

As mentioned above, we will often work with sets of the form $\{u<v\}$ where $u$ and $v$ are quasi-psh functions. In general these sets are not open but merely quasi-open.

If a sequence of positive measures $\mu_j$ converges weakly in $U$ to a positive measure $\mu$ then a elementary argument shows that if $E$ is open and $V$ is closed then 
$$
\liminf_{j\to \infty} \mu_j(E) \geq \mu(E), \qquad \limsup_{j\to \infty} \mu_{j}(V) \leq \mu(V).
$$ 
Using the above facts one can easily argue the following result:

\begin{lemma}\label{lem: cv MA quasi-open}
    Assume $u_j$ is a sequence of uniformly bounded $\theta$-psh functions in $U\subset X$ converging in capacity to a $\theta$-psh function $u$. Suppose $\{f_j\}_j$ are uniformly bounded non-negative quasi-continuous
functions which converge in capacity to another quasi-continuous function $f \geq 0$ on $U$.\\
If $E\subset U$ is a quasi-open set then 
    \[
   \liminf_{j\to \infty} \int_E f_j(\theta +dd^c u_j)^n \geq \int_E f(\theta+dd^c u)^n. 
    \]
If $V\subset U$ is a quasi-closed set then 
    \[
   \limsup_{j\to \infty} \int_V f_j(\theta +dd^c u_j)^n \leq \int_V f(\theta+dd^c u)^n. 
    \]
\end{lemma}

We will also need the following basic convergence result:
\begin{lemma}
        \label{lem: convergence}
        Assume that $\mu_j$ is a sequence of positive Borel measures converging weakly to $\mu$. Assume that there exists a continuous function $f: [0,\infty) \rightarrow [0,\infty)$ with $f(0)=0$ such that, for any Borel set $E$,
        \[
        \mu_j (E) + \mu(E) \leq f \left(\capa_{\omega} (E)\right).
        \]
        Let $u_j$ be a sequence of uniformly bounded quasi-continuous functions which converges in capacity to a bounded quasi-continuous function $u$. Then $u_j \mu_j \rightarrow u \mu$ in the sense of Radon measures on $X$. 
\end{lemma}

\begin{proof}
        Fixing $\varepsilon>0$ there exists a  continuous function $v$ on $X$ such that
        \[
        \capa_{\omega}(\{x\in X\; :
        u(x)\neq v(x)\}) < \varepsilon.
        \]
        Let $A>0$ be a constant such that $|u_j|+|u|+|v|\leq A$ on $X$. Fix $\delta>0$. For $j>N$ large enough we have, by the assumption that $u_j$ converges in capacity to $u$, that
        \[
        \capa_{\omega} (\{x\in X\; : \; |u_j(x) -u(x)|>\delta ) <
        \varepsilon.
        \]
        Fixing a continuous function $\chi$ and $j > N$, it follows from the above that
        \begin{align*}
                &|\int_X (\chi u_j \mu_j - \chi u d\mu) | \leq \int_X |\chi
                (u_j-u) | \mu_j + |\int_X \chi u (\mu_j -\mu) | \\
                & \leq \delta \int_X |\chi| \mu_j + A\sup_X |\chi| \mu_j(\{|u_j-u|>\delta \})
                       \\
                       & + |\int_X \chi (u-v) (\mu_j-\mu)| + |\int_X \chi v
                        (\mu_j-\mu)|\\
                &\leq \delta \int_X |\chi| \mu_j + A\sup_X |\chi|f(\varepsilon)
                + |\int_X \chi (u-v) (\mu_j-\mu)| + |\int_X \chi v
                (\mu_j-\mu)|\\
                &\leq \delta \int_X |\chi| \mu_j + 2A\sup_X |\chi|
                        f(\varepsilon) + |\int_X \chi v (\mu_j-\mu)|.
        \end{align*}
        Since $v$ is continuous on $X$ the last term converges to $0$ as $j\to \infty$. This completes the proof.
\end{proof}

The following lower-semicontinuity property of non-pluripolar products from \cite{DDL2} will be key in the sequel: 
\begin{theorem}
	\label{thm: lsc of MA measures}
	Let $\theta^j, j \in \{1,\ldots,n\}$ be smooth closed real $(1,1)$-forms on $X$ whose cohomology classes are big. Suppose that for all $j \in \{1,\ldots,n\}$  we have $u_j,u_j^k\in \textup{PSH}(X,\theta^j)$ such that  $u^k_j \to u_j$ in capacity as $k \to \infty$, and let $\chi_k,\chi \geq 0$ be quasi-continuous and uniformly bounded such that $\chi_k \to \chi$ in capacity. Then 
\begin{equation}\label{eq: liminf_lim}
	\liminf_{k\to \infty} \int_X \chi_k \theta^1_{u^k_1} \wedge \ldots \wedge \theta^n_{u^k_n}  \geq  \int_X \chi  \theta^1_{u_1} \wedge \ldots \wedge \theta^n_{u_n}. 
\end{equation}
If additionally,  
	\begin{equation}\label{eq: global_mass_semi_cont}
	\int_ X \theta^1_{u_1} \wedge \ldots \wedge \theta^n_{u_n} \geq \limsup_{k\rightarrow \infty} \int_X \theta^1_{u^k_1} \wedge \ldots \wedge \theta^n_{u^k_n},
	\end{equation}
then $\chi_k \theta^1_{u^k_1} \wedge \ldots \wedge \theta^n_{u^k_n}  \to  \chi \theta^1_{u_1} \wedge \ldots \wedge \theta^n_{u_n}$ in the weak sense of measures on $X$. 
\end{theorem}
\begin{proof}
	Set $\Omega:=\bigcap_{j=1}^n {\rm Amp}({\theta^j})$ and fix an open relatively compact subset $U$ of $\Omega$.  Then  the functions $V_{\theta^j}$ are bounded on $U$. We now use a classical idea in pluripotential theory. Fix $C>0,\varepsilon>0$ and consider
	\[
	f_j^{k,C,\varepsilon}:= \frac{\max(u_j^{k}-V_{\theta^j}+C,0)}{\max(u_j^{k}-V_{\theta^j}+C,0)+\varepsilon}, \ j=1,...,n, \ k\in \mathbb{N}^*,
	\]
	and 
	\[
	u_{j}^{k,C}:= \max(u_j^k, V_{\theta^j}-C). 
	\]
	Observe that for $C,j$ fixed, the functions $u_{j}^{k,C}\geq V_{\theta^j}-C$ are uniformly bounded in $U$ (since $V_{\theta^j}$ is bounded in $U$) and converge in capacity to $u_j^{C}$ as $k\to \infty$. Moreover, $f_{j}^{k,C,\varepsilon}=0$ if $u_j^{k}\leq V_{\theta^j}-C$. By locality of the non-pluripolar product we can write 
	\[
	f^{k,C,\varepsilon}  \chi_k \theta^1_{u^k_1} \wedge \ldots \wedge \theta^n_{u^k_n} =f^{k,C,\varepsilon}  \chi_k \theta^1_{u^{k,C}_1} \wedge \ldots \wedge \theta^n_{u^{k,C}_n},
	\]
	where $f^{k,C,\varepsilon}=f_1^{k,C,\varepsilon}\cdots f_n^{k,C,\varepsilon}$. 
For each $C,\varepsilon$ fixed the functions $f^{k,C,\varepsilon}$ are quasi-continuous,  uniformly bounded (with values in $[0,1]$) and converge in capacity to $f^{C,\vep}:=f_1^{C,\vep}\cdots f_n^{C,\vep}$, where $f_j^{C,\vep}$  is defined by 
	\[
	f_j^{C,\vep}  := \frac{\max(u_j-V_{\theta^j}+C,0)}{\max(u_j-V_{\theta^j}+C,0)+\vep}. 
	\]
With the information above we can apply Proposition \ref{prop: xing_conv}  to get that
	\[
	f^{k,C,\vep}  \chi_k \theta^1_{u^{k,C}_1} \wedge \ldots \wedge \theta^n_{u^{k,C}_n}  \rightarrow f^{C,\vep}  \chi \theta^1_{u^{C}_1} \wedge \ldots \wedge \theta^n_{u^{C}_n} \ \textrm{as}\ k\to \infty,
	\]
	in the weak sense of measures on $U$. In particular since $0\leq f^{k,C,\vep}\leq 1$ we have that 
	\begin{eqnarray*}
		\liminf_{k\to\infty} \int_X \chi_k \theta^1_{u^{k}_1} \wedge \ldots \wedge \theta^n_{u^{k}_n} & \geq &  \liminf_{k\to\infty} \int_U f^{k,C,\vep}  \chi \theta^1_{u^{k,C}_1} \wedge \ldots \wedge \theta^n_{u^{k,C}_n}\\
		&\geq & \int_U f^{C,\vep}  \chi \theta^1_{u^{C}_1} \wedge \ldots \wedge \theta^n_{u^{C}_n}.
	\end{eqnarray*}
	Now, letting $\vep\to 0$ and then $C\to \infty$, by definition of the non-pluripolar product we obtain 
	\[
		\liminf_{k\to\infty} \int_X \chi_k \theta^1_{u^{k}_1} \wedge \ldots \wedge \theta^n_{u^{k}_n}   
		\geq  \int_U  \chi \theta^1_{u_1} \wedge \ldots \wedge \theta^n_{u_n}.
	\]
	Finally, letting $U$ increase to $\Omega$ and noting that the complement of $\Omega$ is pluripolar, we conclude the proof of the first statement of the theorem.  
	
To prove the last statement, we first notice that we actually have equality in \eqref{eq: global_mass_semi_cont} and the limsup is a lim, as one can just plug $\chi =1$ in \eqref{eq: liminf_lim}.

Now let $B \in \mathbb R$ such that $\chi,\chi_k \leq B$. By \eqref{eq: liminf_lim} we get that

	\[
	\liminf_{k\to \infty} \int_X (B - \chi_k) \theta^1_{u^k_1} \wedge \ldots \wedge \theta^n_{u^k_n}  \geq  \int_X (B - \chi)  \theta^1_{u_1} \wedge \ldots \wedge \theta^n_{u_n}. 
	\]
Flipping the signs and using (equality in) \eqref{eq: global_mass_semi_cont}, we conclude the following inequality, finishing the proof:
	\[
	\limsup_{k\to \infty} \int_X \chi_k \theta^1_{u^k_1} \wedge \ldots \wedge \theta^n_{u^k_n}  \leq  \int_X \chi  \theta^1_{u_1} \wedge \ldots \wedge \theta^n_{u_n}. 
	\]
\end{proof}

\paragraph{Envelopes.} 
If $f$ is a function on $X$, we define the  envelope of $f$ in the class ${\rm PSH}(X,\theta)$ by
\[
P_{\theta}(f) := \left(\sup \{u\in {\rm PSH}(X,\theta) \; : \;  u\leq f\}\right)^*,
\]
with the convention that $\sup\emptyset =-\infty$. Observe that $P_{\theta}(f)\in {\rm PSH}(X,\theta)$ if and only if there exists some $u\in {\rm PSH}(X,\theta)$ lying below $f$. Note also that $V_{\theta}=P_{\theta}(0)$, and that $P_{\theta}(f+C)= P_{\theta}(f)+C$ for any constant $C$. 

\noindent In the particular case $f=\min(\psi,\phi)$, we denote the envelope as
$ P_\theta(\psi,\phi):=P_\theta(\min(\psi,\phi))$.
We observe that $ P_\theta(\psi,\phi)=  P_\theta(P_\theta(\psi),P_\theta(\phi))$, so w.l.o.g. we can assume  $\psi,\phi$ are two $\theta$-psh functions.

In our first technical result about envelopes, we show that the mass $\theta_{P_\theta(f)}^n$ is concentrated on the contact set $\{P_{\theta}(f)=f\}$:

\begin{theorem}\label{thm: envelope contact}
    Assume $f$ is quasi-continuous on $X$ and $P_{\theta}(f)\in {\rm PSH}(X,\theta)$. Then 
    \[
    \int_{\{P_{\theta}(f)<f\}} (\theta+dd^c P_{\theta}(f))^n =0. 
    \]
\end{theorem}
\begin{proof}
We can assume that $\theta\leq \omega$. 
    Since $P_{\theta}(f)\leq C$ is bounded from above, by replacing $f$ by $\min(f,C)$ we can assume that $f$ is bounded from above. Shifting $f$ by a constant we can also assume $f\leq 0$. 

    {\bf Step 1.} We assume that $f$ is bounded from below $f\geq -C_0$. For each $j\geq 1$, there is an open set $U_j\subset X$ such that ${\rm Cap}_{\omega}(U_j)\leq 2^{-j-1}$ and the restriction of $f$ on $X\setminus U_j$ is continuous. By taking $\cup_{k\geq j} U_k$ we can assume that the sequence $U_j$ is decreasing. By the Tietze extension theorem, there is a function $f_j$ continuous on $X$ such that $f_j=f$ on $D_j:=X\setminus U_j$, moreover $-C_0\leq f_j\leq 0$. For each $j$ we define 
\[
g_j:= \sup_{k\geq j} f_k.
\]
We observe that $g_j$ is lower-semicontinuous on $X$, $g_j=f$ on $D_j$, and $g_j$ decreases to some function $g$ on $X$.   Since the sequence $(D_j)$ is increasing, it follows that $g_j=f$ on $D_k$ for all $k\leq j$. Thus letting $j\to\infty$ gives $g=f$ on $D_k$ for all $k$. This implies $g=f$ in $X$ except for a set of capacity equal to zero. Hence $g=f$ quasi-everywhere in $X$, and $P_{\theta}(f)=P_{\theta}(g)$.  Since $g_j$ is lower-semicontinuous in $X$, by the balayage method, \cite[Corollary 9.2]{BT82}, we have
\[
\int_{\Omega} (1-e^{P_{\theta}(g_j)-g_j})(\theta +dd^c P_{\theta}(g_j))^n=0, 
\]
where $\Omega$ is the ample locus of $\{\theta\}$.  Fix an open set $G\Subset \Omega$. Since $f_j$ is uniformly bounded on $X$, we infer that $P_{\theta}(g_j)-V_{\theta}$ is uniformly bounded, hence there is a constant $B>0$ such that $-B\leq P_{\theta}(g_j) \leq 0$ in $G$. It follows from the plurifine locality that, for all Borel set $E$, 
\[
\id_{G\cap E} (\theta+dd^c P_{\theta}(g_j))^n \leq  \id_{G\cap E} (\omega+dd^c \max(P_{\theta}(g_j),-B))^n\leq B^n {\rm Cap}_{\omega}(E).
\]
It follows from all the above that
\begin{flalign*}
\int_{G}  &|1-e^{P_{\theta}(g_j)-f}|(\theta+dd^c P_{\theta}(g_j))^n \\ 
&= \int_{D_j \cap G} |1-e^{P_{\theta}(g_j)-f}|(\theta+dd^c P_{\theta}(g_j))^n
+ \int_{U_j\cap G}  |1-e^{P_{\theta}(g_j)-f}|(\theta+dd^c P_{\theta}(g_j))^n\\
& =\int_{D_j\cap G} (1-e^{P_{\theta}(g_j)-g_j})(\theta+dd^c P_{\theta}(g_j))^n
+ \int_{U_j\cap G}  |1-e^{P_{\theta}(g_j)-f}|(\theta+dd^c P_{\theta}(g_j))^n\\
&\leq B^n \sup_X |1-e^{P_{\theta}(g_j)-f}|  {\rm Cap}_{\omega}(U_j)\leq C2^{-j-1}. 
\end{flalign*}
The functions $|1-e^{P_{\theta}(g_j)-f}|$ are uniformly bounded and (by construction) converge in capacity to the quasi-continuous function $|1-e^{P_\theta(f)-f}|$. It thus follows from Proposition \ref{prop: xing_conv} that 
\[
|1-e^{P_\theta(g_j)-f}|(\theta+dd^c P_{\theta}(g_j))^n\; \text{ weakly converges to } \; |1-e^{P_\theta(f)-f}|(\theta+dd^c P_{\theta}(f))^n,
\] 
hence 
\begin{flalign*}
\liminf_{j\to \infty} (C2^{-j-1}) &\geq \liminf_{j\to \infty} \int_{G} |1-e^{P(g_j)-f}|(\theta+dd^c P_{\theta}(g_j))^n\\
&\geq \int_{G} |1-e^{P_{\theta}(f)-f}|(\theta+dd^c P_{\theta}(f))^n \geq 0. 
\end{flalign*}
We then infer that 
$$ \int_{G} |1-e^{P_{\theta}(f)-f}|(\theta+dd^c P_{\theta}(f))^n = 0. $$
Letting $G$ increase to $\Omega$, we can conclude that $(\theta+dd^c P_{\theta}(f))^n$ is concentrated on the contact set $\{P_{\theta}(f)=f\}$. 
\smallskip

{\bf Step 2.} For the general case we approximate $f$ by $f_j:=\max(f,-j)$. By the first step,  we have 
\[
\int_{\{P_{\theta}(f_j)<f_j\}} (\theta+dd^c P_{\theta}(f_j))^n=0. 
\]
Fix $C>0$ and an open set $G\Subset \Omega$. Consider the quasi-open set $U:= G \cap \{P_{\theta}(f) > V_{\theta}-C\}$.  Since $P_{\theta}(f_j) \geq P_{\theta}(f)$, thanks to the plurifine property we have  
\[
{\bf 1}_U (\theta+dd^c \max(P_{\theta}(f_j),V_{\theta}-C))^n = {\bf 1}_{U} (\theta+dd^c P_{\theta}(f_j))^n,
\]
hence 
\[
\int_U (1-e^{P_{\theta}(f_j) -f_j)}) (\theta+dd^c \max(P_{\theta}(f_j),V_{\theta}-C))^n=0.  
\]
Observe that $P_\theta(f_j)$ decreases to $P_\theta(f)$, hence it converges in capacity. Letting $j\to \infty$ and using Lemma \ref{lem: cv MA quasi-open} we obtain 
\[
\int_U (1-e^{P_{\theta}(f) -f)}) (\theta+dd^c \max(P_{\theta}(f),V_{\theta}-C))^n=0,  
\]
hence
\[
\int_U (1-e^{P_{\theta}(f) -f)}) (\theta+dd^c P_{\theta}(f))^n=0. 
\]
From this, letting $C\to \infty$ and $U\nearrow \Omega$ we obtain the result. 
\end{proof}

In order to prove a regularity result for $\theta_{V_{\theta}}^n$, we will need two preliminaries lemmas.
We recall that $C^{1, \overline{1}}(X)$ denotes the space of continuous function with bounded distributional Laplacian w.r.t. $\omega$.
\begin{lemma}\label{min}
If $f_1, f_2\in C^{1,\bar{1}}(X)$, then $P_{\omega}(f_1,f_2)\in C^{1,\bar{1}}(X)$ and  for $i= 1,2$  the functions $f_i$ and   $P_{\omega}(f_1,f_2)$ are equal up to second order at almost every point  on the set  $\{ P_{\omega}(f_1,f_2)=f_i  \}$.  In particular, the measures 
$$ {\bf 1}_{ \{ P_{\omega}(f_1,f_2)  = f_1 \} }  \omega_{f_1}^n, \quad  {\bf 1}_{ \{ P_{\omega}(f_1,f_2)  = f_2 \} }  \omega_{f_2}^n, \quad (j=1,2),$$
are positive and 
\begin{eqnarray}\label{eq: measure_pointwise}
{\omega}_{P_{\omega}(f_1,f_2)}^n &=& {\bf 1}_{ \{ P_{\omega}(f_1,f_2)  = f_1 \} }  \omega_{f_1}^n+  {\bf 1}_{ \{ P_{\omega}(f_1,f_2)  = f_2 \} }  \omega_{f_2}^n -  {\bf 1}_{ \{ P_{\omega}(f_1,f_2)  = f_1 = f_2 \} } \omega_{f_1}^n  \\ 
&=& {\bf 1}_{ \{ P_{\omega}(f_1,f_2)  = f_1 \} }  \omega_{f_1}^n +  {\bf 1}_{ \{ P_{\omega}(f_1,f_2)  = f_2 \} }  \omega_{f_2}^n -  {\bf 1}_{ \{ P_{\omega}(f_1,f_2)  = f_1 = f_2 \} } \omega_{f_2}^n.\nonumber
\end{eqnarray}
\end{lemma}    
The $C^{1,\bar{1}}$ regularity of the envelope $P_\omega(f_1)$ is due to Berman \cite{Ber19}. For the $C^{1,\bar{1}}$ regularity of the envelope $P_\theta(f_1)$ in the case $\{\theta\}$ is big we refer to \cite{DNT23}.  The $C^{1,\bar{1}}$ regularity of the rooftop envelope $P_\omega(f_1,f_2)$ was proved in \cite[Theorem 2.5(ii)]{DR16}. For a detailed presentation of these results, we  refer to \cite[Appendix A.1]{Dar18survey}. For Hessian estimates which give the optimal $C^{1,1}$ regularity, see \cite[Theorem 3.1]{Tos18}. Using the $C^{1,\bar{1}}$ estimates, one can reason the same way as in the short argument of \cite[Proposition 2.2]{Dar17AJM} to conclude \eqref{eq: measure_pointwise}.

\begin{lemma}\label{lem: concentration max} Let $\varphi,\psi\in {\rm PSH}(X,\theta)$. Then
\begin{equation}\label{eq: Demailly_est}
\theta_{\max(\varphi, \psi)}^n \geq {\bf 1}_{\{\psi \leq \varphi\}}\theta_\varphi^n + {\bf 1}_{\{\varphi <\psi\}}\theta_\psi^n.
\end{equation}
In particular, if $\varphi \leq \psi$ then ${\bf 1}_{\{\varphi=\psi\}} \theta_\varphi^n \leq {\bf 1}_{\{\varphi=\psi\}} \theta_\psi^n.$
\end{lemma}
\begin{proof} Let $\psi_k := \max(\psi, V_\theta -k)$ and $\varphi_k := \max(\varphi, V_\theta -k)$.

By the locality of the Monge--Amp\`ere measure with respect to the plurifine topology it follows that
    \[
   {\bf 1}_{\{\psi_k> \varphi_k\}} \theta_{\max(\psi_k,\varphi_k)}^n = {\bf 1}_{\{\psi_k>\varphi_k\}} \theta_{\psi_k}^n. 
   \]
   and 
     \[
   {\bf 1}_{\{\psi_k< \varphi_k\}} \theta_{\max(\psi_k,\varphi_k)}^n = {\bf 1}_{\{\psi_k<\varphi_k\}} \theta_{\varphi_k}^n
   \]
   holds in the ample locus of $\{\theta\}$ where all the functions above are locally bounded. As the non-pluripolar products are extended trivially over $X$, we see that the above inequality holds over $X$ in the sense of measures. 
 Considering $\max(\psi_k,\varphi_k+t)$ and letting $t\searrow 0$ we obtain 
$$\theta_{\max(\varphi_k, \psi_k)}^n \geq {\bf 1}_{\{\psi_k \leq \varphi_k\}}\theta_{\varphi_k}^n + {\bf 1}_{\{\varphi_k <\psi_k\}}\theta_{\psi_k}^n.$$

Multiplying with ${\bf 1}_{\{\min(\varphi,\psi) > V_\theta -k \}}$, and using plurifine locality  we arrive at
\[
{\bf 1}_{\{\min(\varphi,\psi) > V_\theta -k \}}\theta_{\max(\varphi, \psi)}^n \geq {\bf 1}_{\{\min(\varphi,\psi) > V_\theta -k \} \cap \{\psi \leq \varphi\}}\theta_{\varphi}^n + {\bf 1}_{\{\min(\varphi,\psi) > V_\theta -k \}\cap \{\varphi <\psi\}}\theta_{\psi}^n.
\]
Letting $k \to \infty$, \eqref{eq: Demailly_est} follows. 
\end{proof}

We are now ready to prove a regularity result about $\theta_{V_{\theta}}^n$, following \cite{DNT21}. 

\begin{theorem}\label{thm: GLZ_prop5.2}
    If $\phi \in {\rm PSH}(X,\theta)$ and $\phi\leq 0$, then 
    \[
    {\bf 1}_{\{\phi=0\}} (\theta +dd^c \phi)^n = {\bf 1}_{\{\phi=0\}} \theta^n.
    \]
    In particular, one has 
    \[
  (\theta+dd^c V_{\theta})^n = {\bf 1}_{\{V_{\theta}=0\}}\theta^n. 
    \]
\end{theorem}
 
\begin{proof}
It suffices to prove the first statement as the second follows from this and Theorem \ref{thm: envelope contact}. Without loss of generality, we can assume $\theta \leq \omega$. Note that $\phi$ is a $\omega$-psh function as well. We proceed in two steps. 

\noindent {\bf Step 1}.   We want to prove in this step that 
\[
{\bf 1}_{\{\phi=0\}}(\omega+dd^c \phi)^n={\bf 1}_{\{\phi=0\}}\omega^n. 
\]
Let $f_j$ be a sequence of smooth $\omega$-psh functions on $X$ decreasing to $\phi$. Set $\phi_j=P_\omega(f_j, 0)$ and observe that $\phi \leq \phi_j$.
Using
\begin{equation}\label{envelop}
\{ \phi = 0 \} \subseteq \{ \phi_j= 0 \}.
 \end{equation}and  Lemma \ref{min} we then get 
\begin{equation}\label{ineq1}
\omega_{\phi_j}^n =    {\bf 1}_{ \{ \phi_j  = 0\} } \omega^n \geq   {\bf 1}_{ \{ \phi = 0 \} }\omega^n.
\end{equation}

The functions $\max(\phi_j,-1)$ are uniformly bounded and decrease to $\max(\phi,-1)$, hence 
\[
(\omega+dd^c \max(\phi_j,-1))^n \to (\omega+dd^c \max(\phi,-1))^n
\]
weakly. Since $\{\phi=0\}=X \setminus \{\phi<0\}$ is a closed subset of $X$, we get 
\[
\limsup_{j\to \infty} \int_{\{\phi=0\}}(\omega+dd^c \max(\phi_j,-1))^n\leq \int_{\{\phi=0\}} (\omega+dd^c \max(\phi,-1))^n. 
\]

Observe also that $\{\phi=0\} \subset \{\phi_j>-1\}$ and $\{\phi=0\}\subset \{\phi>-1\}$.  By plurifine locality we thus have
\begin{equation}\label{eq: DDL5_prop}
\limsup_{j\to \infty} \int_{\{\phi=0\}}(\omega+dd^c \phi_j)^n\leq \int_{\{\phi=0\}} (\omega+dd^c \phi)^n. 
\end{equation}
By \eqref{ineq1} we then get
\[
\int_{\{\phi=0\}} \omega^n \leq \int_{\{\phi=0\}} (\omega+dd^c \phi)^n. 
\]
Then, by Lemma \ref{lem: concentration max} we get the identity
\begin{equation}\label{eqma1}
{\bf 1}_{ \{\phi = 0 \}}  \omega_{\phi}^n   = {\bf 1}_{ \{\phi = 0 \}}\omega^n .  
 \end{equation}

\noindent {\bf Step 2.}  Now, fix $A > 0$ such that $\theta + A \omega$ is a K\"ahler form. Since $\phi $ is $\theta$-psh function,  it follows that $\phi$ is  $(\theta + t \omega)$-psh, for $t \geq 0.$
Let $g\in C^0(X, \mathbb{R})$ and consider the function 
 \[ Q(t) := \int_{ \{ \phi = 0 \} } g ( \theta+ t \omega + dd^c \phi )^n - \int_{ \{ \phi = 0\} } g ( \theta + t \omega )^n \] 
 defined for $t \geq 0.$ Then by multilinearity of the non-pluripolar product and the multilinearity of the product of forms, it is clear that  $Q(t)$ is a polynomial in $t$.
 Thanks to (\ref{eqma1}) we can infer that for any $t>A$
$$ {\bf 1}_{ \{ \phi = 0 \}}  (\theta+t\omega+dd^c \phi)^n   = {\bf 1}_{ \{ \phi = 0 \}}(\theta+t\omega)^n.   $$
This implies that the polynomial  $Q(t)$ is identically zero for $t > A,$ hence $Q(t) \equiv 0$. It then follows that $Q(0) = 0.$ Since $g\in C^0(X, \mathbb{R})$ is arbitrary we have the desired equality between measures.
\end{proof}

As it turns out,  \eqref{eq: DDL5_prop}  follows from a more general result. This was proved in \cite[Proposition 4.6]{DDL5}, and the proof we give below fixes an imprecision in the original argument.

\begin{lemma}Suppose that $u_j,u \in \textup{PSH}(X,\theta)$ with $u_j,u \leq$0. If $\| u_j - u \|_{L^1} \to 0$ then
$$\limsup_{j\rightarrow +\infty } \int_{\{u_j = 0\}} \theta_{u_j}^n \leq \int_{\{u = 0\}} \theta_{u}^n.$$
\end{lemma}

\begin{proof} Let $v_k = \textup{usc}( \sup_{j \geq k} u_j)$. It is well known that $v_k \searrow u$. Also $u_k \leq v_k \leq 0$, so $\{u_k 
 = 0\} \subset \{v_k = 0\}$. As a result, due to Lemma \ref{lem: concentration max}, we have ${\bf 1}_{\{u_k=0\}} \theta_{u_k}^n \leq {\bf 1}_{\{u_k=0\}} \theta_{v_k}^n.$ 

Let $b,c > 0$. Using plurifine locality and the above  we get that
$$
 \int_{\{u_j=0\}} \theta_{u_j}^n \leq \int_{\{u_j=0\}} \theta_{v_j}^n \leq \int_{\{v_j=0\}} \theta_{v_j}^n =\int_{\{v_j=0\}} \theta_{\max(v_j,V_\theta - c)}^n \leq \int_X e^{b v_j} \theta_{\max(v_j,V_\theta - c)}^n.
$$
Since $e^{b v_j}, e^{b u}$ are uniformly bounded and non-negative function, taking the limit Theorem \ref{thm: lsc of MA measures} gives that
$$ \limsup_{j\rightarrow +\infty } \int_{\{u_j=0\}} \theta_{u_j}^n  \leq \int_X e^{b u} \theta_{\max(u,V_\theta - c)}^n. 
$$
Now letting $b \to \infty$, and subsequently $c \to \infty$, the conclusion follows.
\end{proof}

In our investigation of relative pluripotential theory, the following  envelope construction will be essential: given $\phi \in \textup{PSH}(X,\theta)$ we consider
$$\textup{PSH}(X,\theta) \ni \psi \to  \ P_\theta[\psi](\phi) \in \textup{PSH}(X,\theta).$$
This was introduced by Ross and Witt Nystr\"om \cite{RWN14} in their construction of geodesic rays, building on ideas of  Rashkovskii and Sigurdsson \cite{RS05} in the local setting. 
Starting from the  ``rooftop envelope'' $ P_\theta(\psi,\phi)$ we introduce
$$P_\theta[\psi](\phi) := \Big(\lim_{C \to \infty}P_\theta(\psi+C,\phi)\Big)^*.$$
It is easy to see that $P_\theta[\psi](\phi)$ only depends on the singularity type of $\psi$. When $\phi = V_\theta$, we will simply write $P_\theta[\psi]:=P_\theta[\psi](V_\theta)$, and we refer to this potential as the \emph{envelope of the singularity type} $[\psi]$. We note the following simple concavity result about the operator $P_\theta[\cdot](\phi)$. 

\begin{lemma}
\label{lem: concavity of P}
The operator $P_{\theta}[\cdot](\phi)$ is concave: if $u,v,\phi \in {\rm PSH}(X,\theta)$ and $t\in (0,1)$ then
 \[
P_{\theta}[tu+(1-t)v](\phi) \geq tP_{\theta}[u](\phi) + (1-t)P_{\theta}[v](\phi). 
\]
\end{lemma}
\begin{proof}
Assume $u,v \in {\rm PSH}(X,\theta)$ and $t\in (0,1)$. 
Then, for all $C>0$, 
\[
P_{\theta}(\min(tu+(1-t)v+C,\phi))\geq tP_{\theta}(\min(u+C,\phi))+ (1-t) P_{\theta}(\min(v+C,\phi)),
\]
because the right-hand side is a $\theta$-psh function lying below $\min(tu+(1-t)v+C,0)$. 
Letting $C\nearrow \infty$ we arrive at the result.
\end{proof}

\chapter{The basics of relative pluripotential theory}

\section{Monotonicity of non-pluripolar product masses}\label{sec 2}

 In what follows, unless otherwise stated, we work with $\theta$ a smooth closed real $(1,1)$-form whose cohomology class is big. We start with the following result, saying that potentials with the same singularity type have also the same global mass.

\begin{lemma} \label{lem: monotonicity new}
        Let $u,v\in\PSH(X,\theta)$. If $u$ and $v$ have the same singularity type, then $\int_{X}\theta_{u}^{n}=\int_{X}\theta_{v}^{n}$. 
\end{lemma}
The above result is due to Witt Nystr\"om \cite{WN19}. A different proof has been  given in \cite{LN22} using the monotonicity of the energy functional. We give below a direct proof using a standard approximation process. Another different proof has been recently given in \cite{Vu21}, where generalized non-pluripolar products of positive currents are studied.

\begin{proof}
        {\bf Step 1.}  Assume that $\theta$ is a Kähler form.
        \medskip
        
We first prove the following claim: if there exists a constant $C>0$ such that $u=v$ on the open set $U:=\{\min(u,v)<-C\}$ then $\int_X \theta_u^n = \int_X\theta_v^n$.
        
Fix $t>C$. Since $u=v$ on $U$, we have $\max(u,-t)=\max(v,-t)$ on $U$. Since $U$ is open, we have  
        \[
        {\bf 1}_U (\theta +dd^c \max(u,-t))^n={\bf 1}_U (\theta +dd^c \max(v,-t))^n.
        \]
We also have $\{u\leq -t\}=\{v\leq -t\}\subset U$. Indeed,  if $u(x) \leq -t$ then $x\in U$ because $-t< -C$. But on $U$ we have $u=v$, hence $v(x)=u(x) \leq -t$.  
 
By plurifine locality we thus have 
\begin{flalign*}
\theta_{\max(u,-t)}^n&={\bf 1}_{\{u>-t\}}\theta_{\max(u,-t)}^n+{\bf 1}_{\{u\leq-t\}}\theta_{\max(u,-t)}^n\\
        &={\bf 1}_{\{u>-t\}}\theta_u^n+{\bf 1}_{\{v\leq-t\}}\theta_{\max(v,-t)}^n,
\end{flalign*}
and 
\[\theta_{\max(v,-t)}^n={\bf 1}_{\{v>-t\}}\theta_v^n+{\bf 1}_{\{v\leq-t\}}\theta_{\max(v,-t)}^n.\]
Integrating  over $X$, since $\int_X \theta_{\max(u,-t)}^n=\int_X \theta_{\max(v,-t)}^n=\vol({\theta})$ (recall that $\max(u,-t)$ and $\max(u,-t)$ are bounded functions), we get
        \[
        \int_{\{u>-t\}}\theta_{u}^{n}=\int_{\{v>-t\}}\theta_{v}^{n}.
        \]
Letting $t\to\infty$, the claim follows.
        
Now we prove the general case when $\theta$ is K\"ahler. Since $u$ and $v$ have the same singularity type, we can assume that $v\leq u\leq v+B\leq 0$, for some positive constant $B$.  For each $a\in(0,1)$ we set $v_{a}:=av$, $u_a:=\max(u,v_a)$ and $C:= Ba(1-a)^{-1}$. To use the claim above we need to check that $u_a=v_a$ on the open set $U_a:= \{\min(u_a,v_a)<-C\}$. Observe that $\min(u_a,v_a)=v_a$ because $v_a\leq u_a$. If $x\in U_a$ then $av(x) < -C$ hence $(1-a)v(x) <-B$,   which implies (recall that $v+B\geq u$)
        \[
        av(x) \geq v(x)+B \geq u(x).
        \]
We infer that $v_a(x) \geq u(x)$, hence $u_a(x)=v_a(x)$. We can thus apply the claim above to get  $\int_{X}\theta_{u_{a}}^{n}=\int_{X}\theta_{v_{a}}^{n}$. Since non-pluripolar products are multilinear, see \cite[Proposition 1.4]{BEGZ10}, we have that 
        \[
        \int_{X}\theta_{v_{a}}^{n}=a^n\int_X \theta_v^n +\sum_{k=0}^{n-1} a^k (1-a)^{n-k} \int_X \theta_v^k \wedge \theta^{n-k} \to\int_{X}\theta_{v}^{n}
        \] 
as $a\nearrow 1$. Since $u_a\searrow u$ as $a\nearrow1$, by Theorem \ref{thm: lsc of MA measures} we have 
        \[
        \liminf_{a\to1^{-}}\int_{X}\theta_{u_{a}}^{n}\geq\int_{X}\theta_u^n.
        \]
We thus have $\int_{X}\theta_{u}^{n}\leq\int_{X}\theta_{v}^{n}$. By symmetry we get equality,  finishing the proof of Step 1. 
        
        \medskip
        
\noindent {\bf Step 2.} We treat the general case when $\{\theta\}$ is merely big. We use an idea from \cite{DNT21}. Fix $s>0$ so large that $\theta+s\omega$ is Kähler. For $t>s$ we have, by the first step, 
        \[
        \int_{X}(\theta+t\omega+dd^{c}u)^{n}=\int_{X}(\theta+t\omega+dd^{c}v)^n.
        \]
        Since non-pluripolar products are multilinear (\cite[Proposition 1.4]{BEGZ10}) we have for all $t>s$, 
        \[
        \sum_{k=0}^{n}\binom{n}{k}\int_{X}\theta_{u}^{k}\wedge\omega^{n-k}t^{n-k}=\sum_{k=0}^{n}\binom{n}{k}\int_{X}\theta_{v}^{k}\wedge\omega^{n-k}t^{n-k}.
        \]
        We thus obtain an equality between two polynomials for all $t \geq 0$, and identifying the coefficients we infer the desired equality.
\end{proof}

We continue with the following immediate generalization of the above result:
\begin{prop}
	\label{prop: comparison generalization}
	Let $\theta^j, j \in \{1,\ldots,n\}$ be smooth closed real $(1,1)$-forms on $X$ whose cohomology classes are pseudoeffective. Let $u_j,v_j \in \textup{PSH}(X,\theta^j)$ such that $u_j$ has the same singularity type as $v_j, \ j \in \{1,\ldots,n\}$. Then 
	\[
	\int_X \theta^1_{u_1} \wedge \ldots \wedge \theta^n_{u_n}  = \int_X \theta^1_{v_1} \wedge \ldots \wedge \theta^n_{v_n}. 
	\]
\end{prop}
The proof of this result uses the ideas from \cite[Corollary 2.15]{BEGZ10}.
\begin{proof} 
First we note that we can assume that the classes $\{\theta^j\}$ are in fact big. Indeed, if this is not the case we can just replace each $\theta^j$ with $\theta^j + \varepsilon \omega$, and using the multi-linearity of the non-pluripolar product (\cite[Proposition 1.4]{BEGZ10}) we can let $\varepsilon \to 0$ at the end of our argument to conclude the statement for pseudoeffective classes.

For each $t\in\Delta= \{t=(t_1,...,t_n) \in \mathbb{R}^n \setdef t_j > 0\}$ consider $u_t:=\sum_{j} t_ju_{j}$, $v_t:=\sum_{j} t_jv_{j}$ and $\theta^t:=\sum_{j} t_j\theta^j$.  Clearly,  $\{\theta^t\}$ is big, and $u_t$ has the same singularity type as $v_t$. Hence it follows from Lemma  \ref{lem: monotonicity new} that $\int_X (\theta^t_{u_t})^n =\int_X (\theta^t_{v_t})^n$ for all $t\in \Delta$. On the other hand, using multi-linearity of the non-pluripolar product again (\cite[Proposition 1.4]{BEGZ10}), we see that both $t \to \int_X (\theta^t_{u_t})^n$  and $t \to \int_X (\theta^t_{v_t})^n$ are homogeneous polynomials of degree $n$. Our last identity  forces all the coefficients of these polynomials to be equal, giving the statement of our result. 
\end{proof}

Using the above result we argue the monotonicity of non-pluripolar products, conjectured in \cite{BEGZ10} and proved in \cite{DDL2} (see \cite[Theorem 1.1]{Vu21} for a more general result):

\begin{theorem}\label{thm: BEGZ_monotonicity_full}
	Let $\theta^j, j \in \{1,\ldots,n\}$ be smooth closed real $(1,1)$-forms on $X$ whose cohomology classes are pseudoeffective. Let $u_j,v_j \in \textup{PSH}(X,\theta^j)$ be such that $u_j$ is less singular than $v_j$ for all $j \in \{1,\ldots,n\}$. Then 
	\[
	\int_X \theta^1_{u_1} \wedge \ldots \wedge \theta^n_{u_n}  \geq  \int_X \theta^1_{v_1} \wedge \ldots \wedge \theta^n_{v_n}. 
	\]
\end{theorem}

\begin{proof} By the same reason as in  Proposition \ref{prop: comparison generalization}, we can assume that the classes $\{\theta^j\}$ are in fact big. 
For each $t>0$ we set $v_{j}^{t}:= \max(u_{j}-t,v_j)$ for $j=1,...,n$. Observe that the $v_j^t$ converge decreasingly  to $v_j$ as $t \to \infty$. In particular, by \cite[Proposition 3.7]{GZ05} the convergence holds in capacity.  As $v_{j}^{t}$ and $u_j$ have the same singularity type, it follows from Proposition \ref{prop: comparison generalization} that 
\[
\int_X \theta^1_{u_1} \wedge \ldots \wedge  \theta^n_{u_n}=\int_X \theta^1_{v_1^t} \wedge \ldots \wedge  \theta^n_{v_n^t}.
\]
Letting $t \to \infty$, the first part of Theorem \ref{thm: lsc of MA measures} allows to conclude the argument.
\end{proof}

\begin{remark}\label{rem: increasing implies capacity} Condition \eqref{eq: global_mass_semi_cont} in Theorem \ref{thm: lsc of MA measures} is automatically satisfied if $u^k_j \nearrow u_j$ a.e. as $k \to \infty$. Indeed, in this case $u_j^k\to u_j$ in capacity (see \cite[Proposition 4.25]{GZbook}), and by Theorem \ref{thm: BEGZ_monotonicity_full} we have
$\int_ X \theta^1_{u_1} \wedge \ldots \wedge \theta^n_{u_n} \geq \limsup_k \int_X \theta^1_{u^k_1} \wedge \ldots \wedge \theta^n_{u^k_n}$. 

\noindent Since for all $u\in \PSH(X,\theta)$ we have  $P(u, V_\theta + C) \nearrow P_\theta[u]$ as $C \to \infty$, we conclude that
\[
\int_X \theta_u^n = \int_X \theta_{P_{\theta}[u]}^n. 
\]
Similarly, for $u_j \in \textup{PSH}(X,\theta^j)$ the same ideas allow to  conclude that 
\[
\int_X \theta^1_{u_1} \wedge \ldots \wedge \theta^n_{u_n} = \int_X \theta^1_{P_\theta[u_1]} \wedge \ldots \wedge \theta^n_{P_\theta[u_n]}. 
\]
\end{remark}

\section{Model potentials and relative full mass classes}\label{sec 3}

\paragraph{Rooftop envelopes revisited.} We survey some results from \cite{DDL2}. We first start with the following simple observation. 
\begin{lemma}\label{lem: model sup over X}
	Assume $u, v \in \PSH(X,\theta)$ and $u$ is more singular than $v$. Then 
	\[
	\sup_X(u-P_{\theta}[v]) = \sup_X u. 
	\]
\end{lemma}
\begin{proof}
	Since $P_{\theta}[v]\leq 0$, we have $\sup_X(u-P_{\theta}[v]) \geq \sup_X u$. By the assumption that $u$ is more singular than $v$, we have  $u-\sup_X u \leq \min(v+C,0)$ for some constant $C$, hence $u-\sup_X u\leq P_{\theta}[v]$. This proves the other inequality, finishing the proof. 
\end{proof}

We next prove the following result about the complex Monge--Amp\`ere measure of $P_{\theta}(u,v)$ from \cite[Lemma 4.1]{GLZ19} (see \cite{Dar17AJM} for a particular result in the K\"ahler case). 
\begin{theorem}\label{thm: MA of env sing type}
    Suppose $\varphi,\psi, P_{\theta}(\varphi,\psi)\in {\rm PSH}(X,\theta)$. Then 
    $$\theta_{P_{\theta}(\varphi,\psi)}^n \leq  {\bf 1}_{\{P_{\theta}(\varphi,\psi)=\varphi\}}\theta_\varphi^n + {\bf 1}_{\{P_{\theta}(\varphi,\psi)=\psi\}}\theta_\psi^n.$$
    In particular, 
    \[
    \theta_{P_{\theta}[\psi](\varphi)}^n \leq {\bf 1}_{\{P_{\theta}[\psi](\varphi)=\varphi\}} \theta_{\varphi}^n \quad  \text{and}\quad \theta_{P_{\theta}[\psi]}^n \leq {\bf 1}_{\{P_{\theta}[\psi]=0\}} \theta^n. 
 \]
\end{theorem}
\begin{proof}
   Since the function $\min(\varphi,\psi)$ is quasi-continuous on $X$, by Theorem \ref{thm: envelope contact}, the Monge--Amp\`ere measure $\theta_{P_{\theta}(\varphi,\psi)}^n$ is concentrated on the contact set $\{P_{\theta}(\varphi,\psi)=\min(\varphi,\psi)\}$. By Lemma \ref{lem: concentration max} we have 
   \[
   {\bf 1}_{\{P_{\theta}(\varphi,\psi)=\varphi\}} \theta_{P_{\theta}(\varphi,\psi)}^n \leq {\bf 1}_{\{P_{\theta}(\varphi,\psi)=\varphi\}} \theta_{\varphi}^n\quad \text{and}\quad {\bf 1}_{\{P_{\theta}(\varphi,\psi)=\psi\}} \theta_{P_{\theta}(\varphi,\psi)}^n \leq {\bf 1}_{\{P_{\theta}(\varphi,\psi)=\psi\}} \theta_{\psi}^n.
   \]
   From this the first inequality follows. Using this, for each $t>0$ we have 
   \begin{flalign*}
      \theta_{P_{\theta}(\psi+t,\varphi)}^n &\leq  {\bf 1}_{\{P_{\theta}(\psi+t,\varphi)=\psi+t\}}\theta_{\psi}^n + {\bf 1}_{\{P_{\theta}(\psi+t,\varphi)=\varphi\}}\theta_{\varphi}^n\\
      & \leq {\bf 1}_{\{\psi+t\leq \varphi\}}\theta_{\psi}^n + {\bf 1}_{\{P_{\theta}(\psi+t,\varphi)=\varphi\}}\theta_{\varphi}^n. 
   \end{flalign*}
  
Since $\theta_{\psi}^n$ vanishes on the pluripolar set $\{\psi=-\infty\}$, we have $\lim_{t\to \infty}{\bf 1}_{\{\psi\leq \varphi-t\}}\theta_{\psi}^n=0$. Observe also that $P_{\theta}(\psi+t,\varphi)\nearrow P_{\theta}[\psi](\varphi)$ as $t\nearrow \infty$, hence Theorem  \ref{thm: lsc of MA measures} and Remark \ref{rem: increasing implies capacity} ensure that $\theta_{P_{\theta}(\psi+t,\varphi)}^n$ converges weakly to $\theta_{P_{\theta}[\psi](\varphi)}^n$. We also have that $\{P_{\theta}(\psi+t,\varphi)=\varphi\} \subset \{P_{\theta}[\psi](\varphi)=\varphi\}$. Letting $t\to \infty$ thus gives 
\[
\theta_{P_{\theta}[\psi](\varphi)}^n \leq {\bf 1}_{\{P_{\theta}[\psi](\varphi)=\varphi\}} \theta_{\varphi}^n. 
\]
   In particular, when $\varphi=V_{\theta}$, this yields  
   \[
\theta_{P_{\theta}[\psi]}^n =\theta_{P_{\theta}[\psi](V_{\theta})}^n \leq {\bf 1}_{\{P_{\theta}[\psi](V_{\theta})=V_{\theta}\}} \theta_{V_{\theta}}^n\leq {\bf 1}_{\{P_{\theta}[\psi]=0\}} \theta^n,
\]
where the last inequality follows from Theorem \ref{thm: GLZ_prop5.2} since $\{P_{\theta}[\psi]=0\} \subset \{V_\theta=0\}$.
\end{proof}

\paragraph{The domination principle and uniqueness.}

\begin{definition}
Given a potential $\phi \in \PSH(X, \theta)$, the relative full mass class $\mathcal{E}(X,\theta,\phi)$ is the set of all $\theta$-psh functions $u$ such that $u$ is more singular than $\phi$ and $\int_X \theta_u^n=\int_X \theta_{\phi}^n$. 
\end{definition}

\begin{remark}
	If $u, v\in \mathcal{E}(X,\theta,\phi)$ then $\max(u,v)\in \mathcal{E}(X,\theta,\phi)$. Indeed, since both $u$ and $v$ are more singular than $\phi$, $\max(u,v)$ is also more singular than $\phi$. By Theorem \ref{thm: BEGZ_monotonicity_full}, 
	\[
	\int_X \theta_u^n \leq \int_X \theta_{\max(u,v)}^n \leq \int_X \theta_{\phi}^n= \int_X \theta_{u}^n,
	\]
	hence all the inequalities become equalities. 
\end{remark}

\begin{definition}
    A model potential is a $\theta$-psh function $\phi$ such that $P_{\theta}[\phi]=\phi$ and $\int_X \theta_{\phi}^n>0$.  
\end{definition} 

Let us prove an important technical result from \cite{DDL5}.

\begin{theorem}\label{thm: subextension b b-1}
Assume that $u,v\in {\rm PSH}(X,\theta)$, $u \leq v$, $\int_X \theta_u^n>0$ and $b>1$ is such that 
\begin{equation}\label{eq: non-collapsing n}
b^n\int_X \theta_{u}^n>(b^n-1)\int_X \theta_{v}^n.
\end{equation}
Then $P_\theta(bu+(1-b)v) \in {\rm PSH}(X,\theta)$. 
\end{theorem}

To be clear, $P_\theta(bu+(1-b)v) := \textup{usc} (\sup\{h \in \textup{PSH}(X,\theta) \textup{ such that } h + (b-1)v \leq bu\})$, and the set of candidates for this supremum is typically empty. Observe that, if $u$ and $v$ have the same positive mass then \eqref{eq: non-collapsing n} is trivially satisfied.

\begin{proof}
Set $\phi:= P_{\theta}[v]$. 
It suffices to prove that $P_\theta(bu-(b-1)\phi) \in {\rm PSH}(X,\theta)$ since $v \leq P_\theta[v]$, hence $bu + (1-b)v \geq bu + (1-b) P_\theta[v]$. By Remark \ref{rem: increasing implies capacity} we have $\int_X \theta_v^n=\int_X\theta_{\phi}^n$.

For $j\in \mathbb{N}$ we set $u_j:=\max(u,\phi-j)$ and $\varphi_j:=P(bu_j+(1-b)\phi)$. Observe that $\varphi_j$ is a decreasing sequence of $\theta$-psh functions, having the same singularity type as $\phi$. The proof is finished if we can show that $\varphi:=\lim_j \varphi_j$ is not identically  $-\infty$. Assume by contradiction that $\sup_X \varphi_j \to -\infty$.  
Set $\psi_j:=b^{-1} \varphi_j +(1-b^{-1})\phi$ and $D_j:= \{b^{-1} \varphi_j +(1-b^{-1})\phi=u_j\}$. Note that $\psi_j\leq u_j$ with equality on the contact set $D_j$.  Lemma \ref{lem: concentration max} thus ensures that 
\begin{equation}
	\label{eq: subextension b b-1}
	{\bf 1}_{D_j}b^{-n} \theta_{\varphi_j}^n \leq {\bf 1}_{D_j} \theta_{\psi_j}^n \leq  {\bf 1}_{D_j} \theta_{u_j}^n. 
\end{equation}
Also, observe that by Theorem \ref{thm: envelope contact}, the Monge--Amp\`ere measure $\theta_{\varphi_j}^n$ is supported on $\{\varphi_j =bu_j+ (1-b)\phi\}=D_j$.

Fix $j>k>0$.  We note that $u_j=u$ on $\{u>\phi-k\}$ and, since  $u_j$ has the same singularity type as $\phi$, by Lemma \ref{lem: monotonicity new} and the plurifine locality we have
\begin{equation}\label{eq: interm_est}
\int_{\{u\leq \phi -k\}} \theta_{u_j}^n = \int_X \theta_{u_j}^n -\int_{\{u>\phi-k\}} \theta_{u_j}^n=\int_X \theta_{\phi}^n -\int_{\{u>\phi-k\}} \theta_{u}^n.
\end{equation}
Since $\{u_j \leq \phi-k\} = \{u\leq \phi-k\}$,  Theorem \ref{thm: envelope contact} and \eqref{eq: subextension b b-1}  we obtain  
\begin{flalign}\label{eq: sub extension p p-1 2}
&\theta_{\varphi_j}^n(\{\varphi_j\leq \phi-bk\}) \leq  b^n {\bf 1}_{D_j} \theta_{u_j}^n(\{ \varphi_j\leq \phi-bk\}) \leq  b^n \theta_{u_j}^n (\{bu_j+(1-b)\phi \leq \phi-bk\}) \nonumber \\
    & =  b^n \theta_{u_j}^n (\{u_j\leq \phi-k\})\leq   b^n \theta_{u_j}^n (\{u\leq \phi-k\}) \leq  b^n\bigg(\int_X \theta_{\phi}^n -\int_{\{u>\phi-k\}} \theta_u^n\bigg),
\end{flalign}
where in the last inequality we have used \eqref{eq: interm_est}. 
Since $\phi = P_{\theta}[v]$, by Lemma \ref{lem: model sup over X} we have $\sup_X ( \varphi_j-\phi)=\sup_X \varphi_j\to -\infty$.  From this we see that $\{\varphi_j\leq \phi-bk\}=X$ for $j\geq j_0$ large enough, $k$ being fixed. Thus,

$$\int_X \theta_{\phi}^n=\int_X \theta_{\varphi_{j_0}}^n \leq  b^n\bigg(\int_X \theta_{\phi}^n -\int_{\{u>\phi-k\}} \theta_u^n\bigg) $$
where the first identity follows from the fact that $\varphi_{j_0}$ and $\phi$ have the same singularity type. Letting $j_0\to \infty$ in \eqref{eq: sub extension p p-1 2}, and then $k\to \infty$ gives
$$
\int_X \theta_{\phi}^n \leq b^n \left (\int_X \theta_{\phi}^n -\int_X \theta_u^n\right),
$$
contradicting  \eqref{eq: non-collapsing n}.  Consequently,  $\varphi_j$ decreases to a  function in ${\rm PSH}(X,\theta)$, finishing the proof.
\end{proof}

In some instances the conclusion of the above result can be strengthened:

\begin{lemma}
\label{lem: class E relative}
Assume $u,v\in {\rm PSH}(X,\theta)$ and $u\leq v$. Assume also that, for all $b>1$, 
 $P_{\theta}(bu-(b-1)v) \in {\rm PSH}(X,\theta)$. Then, for all $b>1$, 
\[
\int_X (\theta+dd^c P_{\theta}(bu-(b-1)v))^n = \int_X \theta_v^n.
\] 
\end{lemma}
\begin{proof}
Fix $t>b$ and observe that $bu-(b-1)v=t^{-1}b (tu-(t-1)v)+(1-t^{-1}b)v$, where $t^{-1}b<1$. Hence
     \[
   \varphi:=  P_{\theta}(bu-(b-1)v) \geq t^{-1}b P_{\theta}(tu-(t-1)v) +(1-t^{-1}b)v,
     \]
and by Lemma \ref{lem: concavity of P}, 
    \[
    P_{\theta}[\varphi] \geq  t^{-1}bP_{\theta} [P_{\theta}(tu-(t-1)v)] + (1-t^{-1}b)P_{\theta}[v].
    \]
Set $\psi_t:= P_{\theta} [P_{\theta}(tu-(t-1)v)]$. Then $\sup_X \psi_t =0$ and 
    $\psi_t\in {\rm PSH}(X,A\omega)$ for some fixed constant $A>0$. It follows
     from \cite[Proposition 2.7]{GZ05} that the functions $\psi_t$ stay in a compact set of $L^1(X)$. 
Letting $t\to \infty$ we thus have
     \[
     P_{\theta}[\varphi] \geq P_{\theta}[v]. 
     \]
      Since $u\leq v$, we also have $\varphi \leq v$. 
      Combining this and the above inequality we see that 
       $P_{\theta}[\varphi] = P_{\theta}[v]$, and using 
      Remark \ref{rem: increasing implies capacity}  we arrive at the result. 
\end{proof}

Next we prove the domination principle of relative pluripotential theory, extending a result of Dinew from the K\"ahler case \cite{BL12PA}.

\begin{theorem}\label{thm: domination principle}
	Assume $\phi$ is a model potential and $u, v \in \mathcal{E}(X,\theta,\phi)$. Then
	\[
	\theta_u^n (\{u<v\}) =0 \Longrightarrow u\geq v. 
	\]
\end{theorem}
\begin{proof}
	Since $\max(u,v)\in \mathcal{E}(X,\theta,\phi)$, we can assume that $u\leq v\leq 0$. By Lemma \ref{lem: concentration max} we have 
	\[
	\theta_v^n \geq {\bf 1}_{\{u=v\}} \theta_v^n \geq {\bf 1}_{\{u=v\}} \theta_u^n = \theta_u^n,
	\]
where the last identity follows from the fact that $\theta_u^n (u<v) =0$.
	Comparing the total mass we thus have $\theta_u^n=\theta_v^n$ and the measure is concentrated on $\{u=v\}$.  
	
	Fix $b>1$ and set $\varphi_b:= P_{\theta}(b u -(b-1)v)$.  Since the masses of $u$ and $v$ are equal the assumption in Theorem \ref{thm: subextension b b-1}  is trivially satisfied, hence $\varphi_b\in {\rm PSH}(X,\theta)$. By Lemma \ref{lem: class E relative} above, we have $\int_X \theta_{\varphi_b}^n =\int_X \theta_{\phi}^n>0$.

  Set $\psi_b:=b^{-1} \varphi_b +(1-b^{-1})v$ and $D_b:= \{b^{-1} \varphi_b+(1-b^{-1})v=u\}=\{\varphi_b=b u -(b-1)v\}$. Note that $\psi_b\leq u$ with equality on the contact set $D_b$.  Lemma \ref{lem: concentration max} and Theorem \ref{thm: envelope contact} thus ensure that 
\[
	b^{-n}\theta_{\varphi_b}^n \leq b^{-n}\theta_{\varphi_b}^n + {\bf 1}_{D_b} (1-b^{-1})^n\theta_v^n = {\bf 1}_{D_b}b^{-n} \theta_{\varphi_b}^n + {\bf 1}_{D_b}  (1-b^{-1})^n\theta_v^n \leq {\bf 1}_{D_b} \theta_{\psi_b}^n \leq  {\bf 1}_{D_b} \theta_{u}^n. 
\]
 We then infer 
\[
b^{-n}\theta_{\varphi_b}^n  \leq \theta_u^n,
\]
hence $\theta_{\varphi_b}^n$ is concentrated on $\{\varphi_b=bu-(b-1)v=u\}$. Since $\varphi_b\leq u$, it thus follows from Lemma \ref{lem: concentration max} that 
\[
\theta_{\varphi_b}^n= {\bf 1}_{\{\varphi_b=u\}} \theta_{\varphi_b}^n\leq \theta_{u}^n.
\]
Comparing the total mass, we obtain $\theta_{\varphi_b}^n=\theta_u^n$, hence 
\[
\int_X e^{\varphi_b} \theta_u^n =\int_X e^{\varphi_b} \theta_{\varphi_b}^n = \int_X e^{u}\theta_u^n>0. 
\]
Note also that, since $u\leq v$, $\varphi_b$ is decreasing in $b$. The above thus implies that $\varphi:= \lim_{b\to \infty} \varphi_b$ is not identically $-\infty$. Moreover, for any $x\in \{u<v\}$, we have $\varphi_b(x) \leq bu(x)-(b-1)v(x) \leq v(x) + b(u(x)-v(x))$. Letting $b\to \infty$, we see that $\varphi(x)=-\infty$. This implies that $\{u<v\}$ is pluripolar, hence empty. It then follows that $u\geq v$ as desired. 
\end{proof}

The following result states that the non-pluripolar Monge--Amp\`ere measure determines the potential within a relative full mass class. This extends a result of Dinew from the  K\"ahler case \cite{DiwJFA09}.

\begin{theorem}\label{thm: uniqueness}
	Assume $\phi$ is a model potential and $u, v \in \mathcal{E}(X,\theta,\phi)$. Then
	\[
	\theta_u^n  = \theta_v^n  \Longrightarrow u -v \; \text{is constant}. 
	\]
\end{theorem}
\begin{proof}
We normalize $u, v$ by $\sup_X u =\sup_X v=0$. The goal is then to prove that $u=v$. Then $\max(u,v)\in \mathcal{E}(X,\theta,\phi)$ and by Lemma \ref{lem: concentration max}, $\theta_{\max(u,v)}^n\geq \mu: = \theta_u^n$. Comparing the total mass we see that $\theta_{\max(u,v)}^n=\mu$. Thus, replacing $v$ with $\max(u,v)$, we can assume that $u\leq v$.  
	
	Fix $b>1$ and set $\varphi_b:= P_{\theta}(b u -(b-1)v)$. Since the masses of $u$ and $v$ are equal, Theorem \ref{thm: subextension b b-1} ensures that $\varphi_b\in {\rm PSH}(X,\theta)$. By Lemma \ref{lem: class E relative}, we have $\int_X \theta_{\varphi_b}^n =\int_X \theta_\phi^n>0$.  Since $\varphi_b\leq u$, we infer that $\varphi_b\in \mathcal{E}(X,\theta,\phi)$. 
	
	 Set $\psi_b:=b^{-1} \varphi_b +(1-b^{-1})v$ and $D_b:= \{b^{-1} \varphi_b+(1-b^{-1})v=u\}$. Note that $\psi_b\leq u$ with equality on the contact set $D_b$.  Lemma \ref{lem: concentration max} and Theorem \ref{thm: envelope contact} thus ensure that 
  \[
 {\bf 1}_{D_b}\theta_{\psi_b}^n \leq  {\bf 1}_{D_b}\theta_u^n=  {\bf 1}_{D_b}\mu,
  \]
  hence
\[
	 b^{-n}\theta_{\varphi_b}^n + {\bf 1}_{D_b} (1-b^{-1})^n\theta_v^n = {\bf 1}_{D_b}b^{-n} \theta_{\varphi_b}^n + {\bf 1}_{D_b}  (1-b^{-1})^n\theta_v^n \leq {\bf 1}_{D_b} \theta_{\psi_b}^n \leq  {\bf 1}_{D_b} \theta_{u}^n. 
\]
It thus follows that $\theta_{\varphi_b}^n = f \mu$ for some $f\in L^1(\mu)$ whose support is contained in $D_b$. The mixed Monge--Amp\`ere inequalities, \cite[Proposition 1.11]{BEGZ10} yield 
\[
\theta_{\varphi_b}^j\wedge\theta_v^{n-j}\geq f^{j/n} \mu,
\]
hence
\begin{flalign}\label{eq_mix}
 \theta_{\psi_b}^n &=  \sum_{j=0}^n \binom{n}{j} (b^{-1}\theta_{\varphi_b})^j \wedge \left((1-b^{-1}) \theta_v\right)^{n-j} \nonumber \\
 &\geq   \sum_{j=0}^n \binom{n}{j} b^{-j}(1-b^{-1})^{n-j}f^{\frac{j}{n}} \mu \nonumber \\
  &= (b^{-1} f^{1/n}+ 1-b^{-1})^n \mu. 
\end{flalign}
Since ${\bf 1}_{D_b}\mu \geq {\bf 1}_{D_b}\theta_{\psi_b}^n$, it follows from the above that $f^{1/n} \leq 1$ on $D_b$. Since $f$ is supported in $D_b$ we get $f\leq 1$, and so $\theta_{\varphi_b}^n  \leq \mu$. Comparing the total masses (noting that $\varphi_b\in \mathcal{E}(X,\theta,\phi)$) gives $f=1$. Hence $\theta_{\varphi_b}^n =\mu$, and from \eqref{eq_mix} we also get $\theta_{\psi_b}^n \geq \mu$. Since $\psi_b\leq u$,  we then infer, via Theorem \ref{thm: BEGZ_monotonicity_full},  that $\theta_{\varphi_b}^n =\theta_{\psi_b}^n=\mu$ is concentrated on $\{\varphi_b=bu-(b-1)v\}$.
 Now, $\{\psi_b<u\} = \{\varphi_b<bu-(b-1)v\}$, hence
\[
\theta_{\psi_b}^n(\psi_b<u) =0.
\]
Invoking the domination principle (Theorem \ref{thm: domination principle}) we obtain $\psi_b=u$, hence $\varphi_b= bu -(b-1)v$. Since $u\leq v \leq 0$ and $\sup_X u=0$, it follows that there exists $x\in X$ with $u(x)=v(x)=0$. We then infer $\sup_X \varphi_b=0$, hence the functions $\varphi_b= v+b(u-v)$ decrease to some $\varphi\in {\rm PSH}(X,\theta)$ which is $-\infty$ in $\{u<v\}$. It follows that $\{u<v\}$ is a pluripolar set, hence $u\geq v$ almost everywhere. Since $u,v$ are quasi-psh functions we can conclude that $u\geq v$ everywhere.
\end{proof}

\paragraph{Maximality of model potentials.}Based on our previous findings, one wonders if the following set of potentials has a maximal element:

\[ F_\phi:= \bigg\{v\in {\rm PSH}(X,\theta) \; : \;   \phi \leq v \leq 0 \ \textrm{and} \ \int_X \theta_v^n =\int_X \theta_{\phi}^n\bigg\}. 
\]
In other words, does there exist a \emph{least singular} potential that is less singular than $\phi$ but has the same full mass as $\phi$. As shown in the following result, if $\int_X \theta_\phi^n >0$, this is indeed the  case, moreover this maximal potential is equal to $P_\theta[\phi]$.

\begin{theorem}\label{thm: ceiling coincide envelope non collapsing}
Assume that   $\phi\in {\rm PSH}(X,\theta)$ and $\int_X \theta_{\phi}^n>0$. Then 
$$P_{\theta}[\phi] = \sup_{v \in F_\phi} v.$$
In particular, $P_{\theta}[\phi]=P_{\theta}[P_{\theta}[\phi]]$. 
\end{theorem}

As shown in Remark \ref{rem: increasing implies capacity}, $P_{\theta}[\phi] \in F_\phi$, hence by Theorem \ref{thm: ceiling coincide envelope non collapsing}, $P_{\theta}[\phi]$ is the maximal element of $F_\phi$.

\begin{proof}
Let $v\in F_{\phi}$. By Theorem \ref{thm: MA of env sing type} we have
\begin{flalign*}
\theta_{P_{\theta}[\phi]}^n (\{P_{\theta}[\phi] < v\}) &\leq {\bf 1}_{\{P_{\theta}[\phi]=0\}}\theta^n  (\{P_{\theta}[\phi] < v\}) \leq {\bf 1}_{\{P_{\theta}[\phi]=0\}}\theta^n  (\{P_{\theta}[\phi] < 0\})=0.
\end{flalign*}
As $\int_X \theta_\phi^n = \int_X \theta_v^n$, by Remark \ref{rem: increasing implies capacity} we have 
$$
\int_X \theta^n_{P_{\theta}[\phi]}=\int_X \theta_\phi^n = \int_X \theta_v^n=\int_X \theta^n_{P_{\theta}[v]}>0.
$$
Consequently, $P_{\theta}[\phi],v \in \mathcal E(X,\theta, P_{\theta}[v])$ and Theorem \ref{thm: domination principle} now ensures that $P_{\theta}[\phi] \geq v$, hence $P_{\theta}[\phi]\geq \sup_{v \in F_{\phi}} v$. As $P_{\theta}[\phi] \in F_{\phi}$, it follows that $P_{\theta}[\phi]= \sup_{v \in F_{\phi}} v$. 

For the last statement notice that $P_\theta[\phi]=\sup_{v \in F_\phi} v \geq  \sup_{v \in F_{P_{\theta}[\phi]}} v=P_\theta[P_\theta[\phi]]$, since $F_\phi \supset F_{P_{\theta}[\phi]}$. The reverse inequality is trivial. 
\end{proof}

As a consequence of this last result, we obtain the following characterization of membership in $\mathcal E(X,\theta,\phi)$. 
\begin{theorem}\label{thm: E_memb_char} Suppose $\phi \in \textup{PSH}(X,\theta)$ with $\int_X \theta_\phi^n >0$ and $\phi \leq 0$. The following are equivalent:\\
\noindent (i) $u \in \mathcal E(X,\theta,\phi)$.\\
\noindent (ii) $\phi$ is less singular than $u$, and $P_{\theta}[u](\phi)=\phi$.\\
\noindent (iii) $\phi$ is less singular than $u$, and $P_{\theta}[u]=P_{\theta}[\phi]$.
\end{theorem}

As a consequence of the equivalence between (i) and (iii), we see that the potential $P_{\theta}[u]$ stays the same for all $u \in \mathcal E(X,\theta,\phi)$, i.e., it is an invariant of this class. In particular, since $\mathcal E(X,\theta,\phi) \subset \mathcal E(X,\theta,P_{\theta}[\phi])$, by the last statement of Theorem \ref{thm: ceiling coincide envelope non collapsing}, it seems natural to only consider potentials $\phi$ that are in the image of the operator $\psi \to P_{\theta}[\psi]$, when studying classes of relative full mass $\mathcal E(X,\theta,\phi)$. What is more, in the next section it will be clear that considering such $\phi$ is not just more natural, but also necessary when trying to solve complex Monge--Amp\`ere equations with prescribed singularity.

\begin{proof}Assume that (i) holds. By Theorem \ref{thm: MA of env sing type} it follows that $P_{\theta}[u](\phi) \geq \phi$ a.e. with respect to $\theta^n_{P_{\theta}[u](\phi)}$. Theorem \ref{thm: domination principle} gives $P_{\theta}[u](\phi) = \phi$, hence (ii) holds. 

Suppose (ii) holds. We can assume that $u \leq \phi \leq 0$. Then $P_{\theta}[u] \geq P_{\theta}[u](\phi) = \phi$. By the last statement of the previous theorem, this implies that 
$$P_{\theta}[u] = P_{\theta}[P_{\theta}[u]] \geq P_{\theta}[\phi].$$
As the reverse inequality is trivial, (iii) follows.

Lastly, assume that (iii) holds. By Remark \ref{rem: increasing implies capacity}  it follows that 
$$\int_X \theta_u^n=\int_X \theta_{P_{\theta}[u]}^n=\int_X \theta_{P_{\theta}[\phi]}^n=\int_X \theta_\phi^n,$$ hence (i) holds.
\end{proof}

\begin{coro}\label{cor: convexity}
	Suppose $\phi \in \textup{PSH}(X,\theta)$ such that $\int_X \theta_{\phi}^n>0$. Then $\mathcal{E}(X,\theta,\phi)$ is convex. Moreover, given $\psi_1, \ldots, \psi_n \in \mathcal{E}(X, \theta, \phi)$ we have
   \begin{equation}\label{mixed mass} 
    \int_X \theta_{\psi_1}^{s_1}\wedge \ldots \wedge \theta_{\psi_n}^{s_n}= \int_X \theta_\phi^n,
    \end{equation}
 where $s_j \geq 0$ are integers such that $\sum_{j=1}^n s_j=n$.
\end{coro}
\begin{proof}
  Let $u,v\in \mathcal{E}(X,\theta,\phi)$ and fix $t\in (0,1)$. It follows from Theorem \ref{thm: E_memb_char} that $P_\theta[v](\phi)=P_\theta[u](\phi)=\phi$. By Lemma \ref{lem: concavity of P},  $$P_\theta[tv + (1-t)u](\phi)\geq tP_\theta[v](\phi) + (1-t)P_\theta[u](\phi)=\phi.$$
As the reverse inequality is trivial, another application of Theorem  \ref{thm: E_memb_char} gives  $t v + (1-t)u \in \mathcal E(X,\theta,\phi)$. 

We prove the last statement. Since $\mathcal E(X,\theta,\phi)$ is convex, given $\psi_1, \ldots, \psi_n \in \mathcal{E}(X, \theta, \phi)$ we know that any convex combination $\psi:=\sum_{j=1}^n s_j \psi_j$ with $0\leq s_j\leq 1 $ and $\sum_j s_j=n$, belongs to $\mathcal{E}(X, \theta, \phi)$. Hence
$$\int_X \bigg( \sum_j s_j \theta_{\psi_j}\bigg)^n=\int_X \theta_{\psi}^n= \int_X\theta_\phi^n=\int_X \bigg( \sum_j s_j\theta_\phi \bigg)^n.$$
We have an identity of two homogeneous polynomials of degree $n$. All the coefficients of these polynomials have to be equal, giving \eqref{mixed mass}.
\end{proof}

\begin{lemma}\label{lem: two functions in E}
Assume $\phi\in {\rm PSH}(X,\theta)$ is a model potential.  If  $b>1$ and $u,v \in \mathcal{E}(X,\theta,\phi)$ then $P_{\theta}(bu-(b-1)v) \in \mathcal{E}(X,\theta,\phi)$. 
\end{lemma}
\begin{proof}
The proof is similar to that of \cite[Corollary 3.20]{LN22}. 
Since $u,v\in \mathcal{E}(X,\theta,\phi)$, both $u$ and $v$ are more singular than $\phi$, 
we can assume that $\max(u,v)\leq \phi$. 
It follows from Theorem \ref{thm: subextension b b-1} (and the comment below it) that $P_{\theta}(bu-(b-1)\phi) \in {\rm PSH}(X,\theta)$ for all $b>1$. 
Lemma \ref{lem: class E relative} ensures that $P_{\theta}(bu-(b-1)\phi) \in \mathcal{E}(X,\theta,\phi)$ 
for all $b>1$. Set $\psi:= P_{\theta}(bu-(b-1)v)$ and $\varphi:= P_{\theta}(bu-(b-1)\phi)$. 
Then $\psi\geq \varphi$ and $\varphi \in \mathcal{E}(X,\theta,\phi)$. 
We also have 
\[
b^{-1}\psi + (1-b^{-1})v \leq u,
\]
which, by Lemma \ref{lem: concavity of P}, gives 
\[
b^{-1}P_{\theta}[\psi] + (1-b^{-1})P_{\theta}[v] \leq P_{\theta}[b^{-1}\psi + (1-b^{-1})v]
 \leq P_{\theta}[u]\leq P_{\theta}[\phi]=\phi. 
\]
Since $P_{\theta}[v]=\phi$, we can thus conclude that $P_{\theta}[\psi]\leq \phi$. On the other hand $ \phi=P_{\theta}[\varphi]\leq P_{\theta}[\psi]$. Thus $P_{\theta}[\psi]= \phi$ and $\psi\in \mathcal{E}(X,\theta,\phi)$. 
\end{proof}

\begin{lemma}\label{lem: P(u,v)_mass_exist}
Assume that $u,v,w \in \textup{PSH}(X,\theta)$ are  such that $\int_X \theta_u^n +\int_X \theta_v^n >\int_X \theta_w^n$ and $\max(u,v) \leq w$. Then $P_\theta(u,v)\in {\rm PSH}(X,\theta)$. 
\end{lemma}

\begin{proof}   
We can assume without loss of generality that $u,v,w \leq 0$. Replacing $w$ with $P_\theta[\varepsilon V_\theta + (1-\varepsilon)w]$ for small enough $\varepsilon>0$ we can also assume that $w$ is a model potential.

For $j \geq 0$ we set $u_j:=\max(u,w-j), v_j:=\max(v,w-j)$, $h_j:=P(u_j,v_j)$. Observe that $u_j, v_j, h_j$ have the same singularity type as $w$. Fix $s>0$ big enough, such that for all $j>s$, we have 
\[
\int_{\{u>w-s\}} \theta_{u_j}^n +\int_{\{v>w-s\}} \theta_{v_j}^n  = \int_{\{u>w-s\}} \theta_{u}^n +\int_{\{v>w-s\}} \theta_{v}^n >\int_X \theta_w^n, 
\]
where in the equality above we used Lemma \ref{lem: plurifine}. Observe that such $s$ exists thanks to the assumption.\\
It follows from  Theorem \ref{thm: MA of env sing type} and the above estimate, that for $j>s$,
\begin{flalign*}
\int_{\{h_j\leq w-s\}} \theta_{h_j}^n &\leq  \int_{\{u_j\leq w-s\}} \theta_{u_j}^n+\int_{\{v_j\leq w-s\}} \theta_{v_j}^n\\
&= 2 \int_X \theta_w^n -\int_{\{u>w-s\}} \theta_{u}^n -\int_{\{v>w-s\}} \theta_{v}^n 
<  \int_X \theta_w^n,
\end{flalign*}
where in the identity above we used the fact that $\{u_j\leq w-s\}= \{u\leq w-s\}$ and that $u_j, v_j$ and $w$ have the same singularity type. Since $u_j,v_j$ decrease to $u,v$ respectively, it follows that $h_j \searrow P(u,v)$. We now rule out the possibility that $P(u,v)\equiv -\infty$. Indeed, suppose  $\sup_X h_j$ decreases to $-\infty$. From Lemma \ref{lem: model sup over X} we obtain that  $\sup_X h_j = \sup_X (h_j - w) \searrow -\infty$. But then, for $j$ large enough the set $\{h_j\leq w-s\}$ coincides with $X$, contradicting our last integral estimate, since each $h_j$  has the same singularity type as $w$. 
\end{proof}
\begin{corollary}\label{cor: P(u,v) in E}
    If $\phi \in {\rm PSH}(X,\theta)$ be a model potential and $u,v\in \mathcal{E}(X,\theta,\phi)$ then $P_{\theta}(u,v)\in \mathcal{E}(X,\theta,\phi)$. 
\end{corollary}
\begin{proof}
    Fixing $b>1$, from Lemma \ref{lem: two functions in E} we infer that $u_b:=P_{\theta}(bu-(b-1)\phi)$ and $v_b:=P_{\theta}(bv-(b-1)\phi)$ belong to $\mathcal{E}(X,\theta,\phi)$. Lemma \ref{lem: P(u,v)_mass_exist} then ensures that $P_{\theta}(u_b,v_b)\in {\rm PSH}(X,\theta)$. We also have 
    \[
    P_{\theta}(u,v) \geq b^{-1}P_{\theta}(u_b,v_b) + (1-b^{-1})\phi. 
    \]
    Comparing the total mass via Theorem \ref{thm: BEGZ_monotonicity_full} and letting $b\to \infty$,  we obtain the result. 
\end{proof}

Plainly speaking, by the next lemma, the fixed point set of the map $\psi \to P[\psi]$ is stable under the operation $(\psi,\phi) \to P(\psi,\phi)$.

\begin{lemma}\label{lem: P_stab_fixed_point} Suppose $u_0,u_1 \in \textup{PSH}(X,\theta)$ are such that $P(u_0,u_1) \in \textup{PSH}(X,\theta)$, and $P[u_0]=u_0$ and $P[u_1]=u_1$. Then $P[P(u_0,u_1)]=P(u_0,u_1).$
\end{lemma}
\begin{proof} As $P(u_0,u_1)\leq \min(u_0,u_1) \leq 0$ and $P[P(u_0,u_1)] \leq \min(P[u_0], P[u_1])$, it follows that
$$P(u_0,u_1) \leq P[P(u_0,u_1)] \leq P(P[u_0],P[u_1])=P(u_0,u_1).
$$
This shows that all the inequalities above are in fact equalities.
\end{proof}

\begin{prop}\label{prop: envelope mixed} Let $\phi,\psi \in \textup{PSH}(X,\theta)$ be such that $\phi = P[\phi]$, $\psi = P[\psi]$, and $P(\phi,\psi) \in {\rm PSH}(X,\theta)$. If $u \in \mathcal E(X,\theta,\phi)$, $v \in \mathcal E(X,\theta,\psi)$ and $\int_X \theta_{P(\phi,\psi)}^n>0$ then $$P(u,v) \in \mathcal E(X,\theta,P(\phi,\psi)).$$ 
\end{prop}

\begin{proof} We can assume that $u\leq \phi$ and $v\leq \psi$. 

{\bf Step 1.}
We first prove that $P(u,\psi) \in \mathcal{E}(X,\theta,P(\phi,\psi))$. 
By assumption we  have $$\int_X \theta_u^n +\int_X \theta_{P(\phi,\psi)}^n >\int_X \theta_{\phi}^n,$$ and Lemma \ref{lem: P(u,v)_mass_exist} gives $P(u,\psi) = P(u,P(\phi,\psi)) \in {\rm PSH}(X,\theta)$. 
Also, it follows from Theorem \ref{thm: subextension b b-1} that $u_b := P_{\theta}(bu -(b-1) \phi) \in {\rm PSH}(X,\theta)$ for all $b>1$. By definition, for $1<b<t$ we have 
$$
\phi\geq u_b \geq bt^{-1}u_t + (1-bt^{-1}) \phi. 
$$
 Comparing the total mass via Theorem \ref{thm: BEGZ_monotonicity_full} and letting $t\to \infty$ we see that $u_b \in \mathcal E(X,\theta,\phi)$. As above, Lemma \ref{lem: P(u,v)_mass_exist} ensures that $P(u_b,\psi) \in {\rm PSH}(X,\theta)$. On the other hand we also have 
$$
\phi\geq u \geq b^{-1} u_b + (1-b^{-1})\phi, 
$$
therefore $P(\phi,\psi)\geq P(u,\psi) \geq b^{-1} P(u_b,\psi) + (1-b^{-1})P(\phi,\psi)$. Comparing the total mass via Theorem \ref{thm: BEGZ_monotonicity_full} and letting $b \to \infty$ we arrive at $\int_X \theta_{P(\phi,\psi)}^n \geq \int_X \theta_{P(u,\psi)}^n \geq \int_X \theta_{P(\phi,\psi)}^n$, hence the conclusion.

{\bf Step 2.} We prove that $P(u,v) \in {\rm PSH}(X,\theta)$. 
It follows from Theorem \ref{thm: BEGZ_monotonicity_full}, the assumption $v \in \mathcal{E}(X,\theta,\psi)$, and the first step that 
$$
\int_X \theta_{P(u,\psi)}^n + \int_X \theta_v^n= \int_X \theta_{P(\phi,\psi)}^n + \int_X \theta_{\psi}^n > \int_X \theta_{\psi}^n.
$$ 
Since $\max(P(u,\psi),v) \leq \psi$,  Lemma \ref{lem: P(u,v)_mass_exist} can be applied giving $P(u,v) = P(P(u,\psi),v) \in {\rm PSH}(X,\theta)$. 

{\bf Step 3.} We conclude the proof. It follows from Theorem \ref{thm: subextension b b-1} that  $v_b:= P_{\theta}(bv -(b-1)\psi) \in {\rm PSH}(X,\theta)$ for all $b>1$. For $1<b<t$ we have 
$$
\psi\geq v_b \geq bt^{-1}v_t + (1-bt^{-1}) \psi. 
$$
Comparing the total mass via Theorem \ref{thm: BEGZ_monotonicity_full} and letting $t\to \infty$ we see that $v_b \in \mathcal E(X,\theta,\psi)$. By the second step we have that $P(u,v_b) \in {\rm PSH}(X,\theta)$.  On the other hand we also have 
$$
\psi \geq v \geq b^{-1} v_b + (1-b^{-1})\psi,
$$
therefore $P(u,\psi) \geq P(u,v) \geq b^{-1} P(u,v_b) + (1-b^{-1})P(u,\psi)$. Comparing the total mass via Theorem \ref{thm: BEGZ_monotonicity_full} and letting $b \to \infty$ we arrive at $ \int_X \theta_{P(u,\psi)}^n \geq \int_X \theta_{P(u,v)}^n \geq \int_X \theta_{P(u,\psi)}^n$. Combining this and the first step we arrive at the conclusion.
\end{proof}

\paragraph{Comparison principle.} We note the \emph{partial comparison principle} for functions s of relative full mass, generalizing a result of Dinew from \cite{DiwJFA09}:

\begin{prop}\label{prop: general CP} Suppose $\psi_k\in \textup{PSH}(X,\theta^k), k=1,\ldots ,j\leq n$ and $\phi\in {\rm PSH}(X,\theta)$ is a model potential. If $u,v\in \mathcal{E}(X,\theta,\phi)$ then 
\[
\int_{\{u<v\}} \theta_{v}^{n-j}\wedge  \theta^1_{\psi_1} \wedge \ldots \wedge \theta^j_{\psi_j} \leq \int_{\{u<v\}} \theta_{u}^{n-j}\wedge  \theta^1_{\psi_1} \wedge \ldots \wedge \theta^j_{\psi_j}.
\] 
\end{prop}

\begin{proof} The proof follows the argument of \cite[Proposition 2.2]{BEGZ10} with a vital ingredient from Theorem \ref{thm: BEGZ_monotonicity_full}.

Since $\max(u,v)\in \mathcal{E}(X,\theta,\phi)$, we have, by Theorem \ref{thm: E_memb_char}, $P_{\theta}[\max(u,v)]=P_{\theta}[u]=P_{\theta}[v]=\phi$.    It thus follows from Theorem \ref{thm: BEGZ_monotonicity_full}  and Remark \ref{rem: increasing implies capacity} that
\begin{eqnarray*}
\int_X \theta_{\phi}^{n-j}\wedge  \theta^1_{\psi_1} \wedge \ldots \wedge \theta^j_{\psi_j} &=&
\int_X \theta_{v}^{n-j}\wedge  \theta^1_{\psi_1} \wedge \ldots \wedge \theta^j_{\psi_j} \\
&\leq & \int_X \theta_{\max(u,v)}^{n-j}\wedge  \theta^1_{\psi_1} \wedge \ldots \wedge \theta^j_{\psi_j}\\
&= &   \int_X \theta_{\phi}^{n-j}\wedge  \theta^1_{\psi_1} \wedge \ldots \wedge \theta^j_{\psi_j}. 
\end{eqnarray*}
Hence the inequality above is in fact equality. 
By locality of the non-pluripolar product we then can write:
\begin{eqnarray*}
	\int_X \theta_{\max(u,v)}^{n-j} \wedge \theta^1_{\psi_1} \wedge ... \wedge \theta^j_{\psi_j}
&\geq & \int_{\{u> v\}} \theta_{u}^{n-j} \wedge \theta^1_{\psi_1} \wedge ... \wedge \theta^j_{\psi_j} + \int_{\{v> u\}} \theta_{v}^{n-j} \wedge \theta^1_{\psi_1} \wedge ... \wedge \theta^j_{\psi_j} \\
&= & \int_X \theta_{u}^{n-j} \wedge \theta^1_{\psi_1} \wedge ... \wedge \theta^j_{\psi_j} - \int_{\{u\leq  v\}} \theta_{u}^{n-j} \wedge \theta^1_{\psi_1} \wedge ... \wedge \theta^j_{\psi_j} \\
&&+  \int_{\{v> u\}} \theta_{v}^{n-j} \wedge \theta^1_{\psi_1} \wedge ... \wedge \theta^j_{\psi_j}\\
& = & \int_X \theta_{\max(u,v)}^{n-j} \wedge \theta^1_{\psi_1} \wedge ... \wedge \theta^j_{\psi_j} - \int_{\{u\leq  v\}} \theta_{u}^{n-j} \wedge \theta^1_{\psi_1} \wedge ... \wedge \theta^j_{\psi_j} \\
&&+  \int_{\{v> u\}} \theta_{v}^{n-j} \wedge \theta^1_{\psi_1} \wedge ... \wedge \theta^j_{\psi_j}.
\end{eqnarray*}
We thus get
\begin{flalign*}
\int_{\{u < v\}} \theta_{v}^{n-j} \wedge \theta^1_{\psi_1} \wedge \ldots \wedge \theta^j_{\psi_j}  \leq \int_{\{u \leq v\}} \theta_{u}^{n-j} \wedge \theta^1_{\psi_1} \wedge \ldots \wedge \theta^j_{\psi_j}.
\end{flalign*}
Replacing $u$ with $u + \varepsilon$ in the above inequality, and letting $\varepsilon \searrow 0$, by the monotone convergence theorem we arrive at the result.
\end{proof}

The above result yields the following important consequence, generalizing \cite[Corollary 2.3]{BEGZ10}.:

\begin{coro}\label{cor: comparison principle} Suppose $\phi \in \textup{PSH}(X,\theta)$ is a model potential and assume that $u,v\in \mathcal{E}(X,\theta,\phi)$. Then 
\[
\int_{\{u<v\}} \theta_{v}^{n} \leq \int_{\{u<v\}} \theta_{u}^{n} \quad \textup{ and }\quad  \int_{\{u\leq v\}} \theta_{v}^{n} \leq \int_{\{u \leq v\}} \theta_{u}^{n}.\] 
\end{coro}

The second inequality follows from the first inequality applied to $u$ and $v+\varepsilon$ and $\varepsilon \searrow 0$.

\chapter{Generalized Monge--Amp\`ere capacities and integration by parts}

\section{Generalized Monge--Amp\`ere capacities}
In this section we prove a comparison of Monge--Amp\`ere capacities which will be used in the proof of the integration by parts formula in the next section. We first start with a version of the Chern-Levine-Nirenberg inequality.
\begin{lemma}
        \label{lem: CLN} Let $u,v,\psi\in{\rm PSH}(X,\omega)$. Assume that  $v\leq u\leq v+B$ for some positive constant $B$.  Then 
        \[
        \int_{X}\psi\omega_{u}^{n}\geq\int_{X}\psi\omega_{v}^{n}-nB\int_{X}\omega^{n}.
        \]
\end{lemma}

\begin{proof} By the monotone convergence theorem we can assume that $\psi$ is bounded.
        By subtracting a constant we can assume that $u\leq 0$. We first prove the lemma under the assumption that $u=v$ on the open set 
        \[
        U:=\{\min(u,v)=v<-C\},
        \] 
        for some positive constant $C$.\\
        We approximate $u$ and $v$ by $u^{t}:=\max(u,-t)$ and $v^{t}:=\max(v,-t)$. For $t>0$ we apply the integration by parts formula for bounded $\omega$-psh functions, which is a consequence of Stokes theorem, to get 
        \[
        \int_{X}\psi(\omega_{u^{t}}^{n}-\omega_{v^{t}}^{n})=\int_X(u^{t}-v^{t})dd^{c}\psi\wedge
        S^{t},
        \]
        where $S^{t}:=\sum_{k=0}^{n-1}\omega_{u^{t}}^{k}\wedge\omega_{v^{t}}^{n-1-k}$. Note that $\int_X \omega\wedge S^t =n\int_X \omega^n$ as there are $n$ terms in the sum and each of them is equal to $\int_X \omega^n$ by Stokes' theorem.  Since $u^{t}\geq v^{t}$ we can continue the above estimate and obtain
        \[
        \int_{X}\psi(\omega_{u^{t}}^{n}-\omega_{v^{t}}^{n})=\int_{X}(u^{t}-v^{t})(\omega_{\psi}\wedge
        S^{t}-\omega\wedge S^{t})\geq-Bn\int_{X}\omega^{n}.
        \]
        For $t>B+C$ we have that $u^{t}=v^{t}$ on the open set $U$ which contains $\{u\leq-t\}=\{v\leq-t\}$. It thus follows that ${\bf 1}_{U}\omega_{u^{t}}^{n}={\bf 1}_{U}\omega_{v^{t}}^{n}$. Thus, for $t>B+C$ we have 
        \begin{eqnarray*}
                \int_{\{v>-t\}}\psi(\omega_{u}^{n}-\omega_{v}^{n})= \int_{\{v>-t\}}\psi(\omega_{u_t}^{n}-\omega_{v_t}^{n}) & = &
                \int_{X}\psi(\omega_{u^{t}}^{n}-\omega_{v^{t}}^{n})\geq-Bn\int_{X}\omega^{n}.
        \end{eqnarray*}
        Letting $t\to\infty$ we finish the first step.
        
        We now treat the general case. By approximating $\psi$ from above by smooth $\omega$-psh functions, see \cite{Dem94}, \cite{BK07}, we can assume that $\psi$ is smooth (in fact, we only need the continuity of $\psi$). We fix $a\in(0,1)$ and set $v_{a}:=av$, $u_{a}:=\max(u,v_{a})$. Setting $C:= a(1-a)^{-1}B$ we have that  $u_{a}=v_{a}$ on $U=\{\min(u_a, v_a)=v_{a}<-C\}$ (see arguments in Step 1 of the proof of Lemma \ref{lem: monotonicity new}). We can thus apply the first step to get  
        \[
        \int_X\psi\omega_{u_a}^n\geq\int_X \psi\omega_{v_a}^n-nB\int_X\omega^n.
        \]
        Observe that $v_a\searrow v$ and $u_a\searrow u$ as $a\nearrow 1$. Also, by the multilinearity of non-pluripolar products, we have 
        \[
        \lim_{a\to 1^-}\int_X (\omega+dd^c v_a)^n = \lim_{a\to 1^-}\int_X ((1-a)\omega+a \omega_v)^n =\int_X (\omega+dd^cv)^n. 
        \]
         Recalling that $v\leq u\leq 0$, we have  $u\leq u_a = \max(u,av)\leq \max(u,au)=au$. By Theorem \ref{thm: BEGZ_monotonicity_full} we thus have 
         \[
  \int_X \omega_u^n \leq \int_X \omega_{u_a}^n \leq \int_X (\omega+a dd^c u)^n.
         \]
        Using the multilinearity of non-pluripolar products, we then have 
         \[
                \int_X (\omega+ dd^c u)^n \leq \lim_{a\to 1^-}\int_X (\omega+dd^c u_a)^n \leq \lim_{a\to 1^-}\int_X (\omega+a dd^c u)^n =\int_X (\omega+dd^c u)^n. 
         \]
          Hence the above inequalities are equalities. It then follows from Theorem \ref{thm: lsc of MA measures} that the positive measures $\omega_{u_a}^n, \omega_{v_a}^n$ converge respectively to $\omega_u^n, \omega_{v}^n$ in the weak sense of Radon measures as $a\nearrow 1$.  Since $\psi$ is continuous on $X$ we thus obtain 
        \[
        \int_X \psi\omega_{u}^{n}\geq\int_{X}\psi\omega_{v}^{n}-nB\int_X \omega^n.
        \]
\end{proof}

We next use the Chern-Levine-Nirenberg  inequality to compare Monge--Amp\`ere capacities. 
\begin{definition}
	Given $\phi\in {\rm PSH}(X,\theta)$ and $E\subset X$ a Borel subset we define 
	\[
	{\rm Cap}_{\theta,\phi}(E):= \sup\left\{\int_E \theta_u^n \; : \; u\in {\rm PSH}(X,\theta), \; \phi-1\leq u \leq \phi\right\}. 
	\]
\end{definition}
Note that in the K\"ahler case  a related notion of capacity has been studied in \cite{DnL17,DiLu15}. 
In the case when $\phi=V_{\theta}$ we recover the Monge--Amp\`ere capacity used in \cite[Section 4.1]{BEGZ10}. 
\begin{lemma}
	\label{lem: cap is inner regular}
	The relative Monge--Amp\`ere capacity ${\rm Cap}_{\theta,\phi}$ is inner regular, i.e.
\[
{\rm Cap}_{\theta,\phi}(E) =\sup \{{\rm Cap}_{\theta,\phi}(K)\; : \;  K\subset E, \; K \; \textrm{is compact}\}.
\]
\end{lemma}
\begin{proof}
By definition ${\rm Cap}_{\theta,\phi}(E) \geq {\rm Cap}_{\theta,\phi}(K) $ for any compact set $ K\subset E$. Fix $\varepsilon>0$. There exists $u\in {\rm PSH}(X,\theta)$ such that $\phi-1\leq u\leq \phi$ and 
\[
\int_E \theta_u^n \geq {\rm Cap}_{\theta,\phi}(E)-\varepsilon.
\]
Since $\theta_u^n$ is an inner regular Borel measure it follows that there exists a compact set $K\subset E$ such that $\int_K \theta_u^n \geq \int_E \theta_u^n -\varepsilon\geq {\rm Cap}_{\theta,\phi}(E)-2\varepsilon$. Hence ${\rm Cap}_{\theta,\phi}(K)\geq {\rm Cap}_{\theta,\phi}(E)-2\varepsilon$. Taking the supremum over all the compact set $K\subset E$, we arrive at the conclusion. 
\end{proof}

\begin{prop}\label{prop: Cap omega phi and Cap omega}
	Assume $\phi \in {\rm PSH}(X,\omega)$ is a model potential.  There exists a constant $C>0$ depending on $X,\omega,n$ such that, for all Borel set $E$, we have 
	\[
	{\rm Cap}_{\omega,\phi}(E) \leq C {\rm Cap}_{\omega}(E)^{1/n}. 
	\]
\end{prop}
\begin{proof}
	By inner regularity of the capacities, we can assume that $E=K$ is compact. We can also assume that $K$ is non-pluripolar, otherwise the inequality is trivial. Let $V_K^*$ be the global extremal function which it is defined as
 $$V_K:=\sup\{ u\in \psh(X, \omega), \;\; u\leq 0 \; {\rm on } \; K\}.$$
 We recall that $V_K^* \geq 0$ and that let $M_K:= \sup_X V_K^*< \infty$ if and only if $K$ is non-pluripolar. By \cite[Theorem 5.2]{GZ05}, $\omega_{V_K^*}^n$ is concentrated on $K$. If $M_K < 1$ then
	\[
	\int_X \omega^n= \int_X \omega_{V_K^*}^n =\int_K \omega_{V_K^*}^n=\int_K \omega_{V_{K^*}-1}^n\leq {\rm Cap}_{\omega}(K),
	\]
	while ${\rm Cap}_{\omega,\phi}(K)\leq \int_X \omega^n$. Hence for $C\geq (\int_X \omega^n)^{1-1/n}$, the desired inequality holds. 
	
	Assume now that $M_K\geq 1$. Then $v:=M_K^{-1}V_K^*-1$ is $\omega$-psh and takes values in $[-1,0]$, hence 
	\[
	M_K^{-n}\int_X \omega^n= M_K^{-n}\int_X \omega_{V_K^*}^n =M_K^{-n}\int_K \omega_{V_K^*}^n\leq \int_K \omega_v^n \leq {\rm Cap}_{\omega}(K). 
	\]
	 Set $\psi:=V_{K}^*-M_{K}$ and let $u$ be a $\omega$-psh function such that $\phi-1 \leq u \leq \phi$.  Since $\psi=-M_K$ on $K$ modulo a pluripolar set, it follows from Lemma \ref{lem: CLN} that 
\[
M_{K}\int_K \omega_u^n \leq \int_X (-\psi) \omega_u^n \leq \int_X (-\psi) \omega_{\phi}^n + n\int_X \omega^n \leq \int_X (-\psi) \omega^n + n\int_X \omega^n\leq A,
\]
for a constant $A>0$ depending on $X, \omega, n$. We have used above the fact that,   $\omega_{\phi}^n\leq {\bf 1}_{\{\phi=0\}} \omega^n$ since $\phi$ is a model potential (see Theorem \ref{thm: MA of env sing type}), hence $\int_X (-\psi) \omega_{\phi}^n \leq \int_X (-\psi) \omega^n $. Taking the supremum over all such $u$ yields 
\[
{\rm Cap}_{\omega,\phi}(K) \leq A M_{K}^{-1} \leq C {\rm Cap}_{\omega}(K)^{1/n},
\]
where the last inequality follows from \cite[Theorem 7.1]{GZ05}.
\end{proof}

\begin{lemma}
        \label{lem: from cap psi to cap varphi} Fix $\varphi,\psi\in{\rm PSH(X,\theta)}$ such that $\psi\leq\varphi$ and $\int_{X}\theta_{\varphi}^n=\int_{X}\theta_{\psi}^n$. Then there exists a continuous function $g:[0,\infty)\rightarrow[0,\infty)$ with $g(0)=0$ such that, for all Borel sets $E$, 
        \[
        {\rm Cap}_{\theta,\psi}(E)\leq
        g\left({\rm Cap}_{\theta,\varphi}(E)\right).
        \]
\end{lemma}
Our proof uses an idea from \cite{GLZ19}.
\begin{proof}
        We can assume that $\varphi\leq 0$.   
                 We claim that if $v\in{\rm PSH}(X,\theta)$ with $\varphi-t\leq v\leq\varphi$ (for $t\geq 0$)
         then for any Borel set $E$ we have 
        \[
        \int_{E}\theta_{v}^{n}\leq\max(t,1)^{n}\capa_{\theta,\varphi}(E).
        \]
        If $t\in [0,1]$ then $v$ is a candidate defining the capacity $\capa_{\theta,\varphi}$, hence the desired inequality holds. For $t> 1$, the function $v_{t}:=t^{-1}v+(1-t^{-1})\varphi$ is $\theta$-psh and $\varphi-1\leq v_t\leq\varphi$. Since non-pluripolar products are multilinear, we thus have 
        \[
        t^{-n}\int_{E}\theta_{v}^{n}\leq\int_{E}\theta_{v_{t}}^{n}\leq\capa_{\theta,\varphi}(E),
        \]
        yielding the claim. 
        
        Let $u$ be a $\theta$-psh function such that $\psi-1\leq u\leq\psi$. Fix $t>1$ and set $u_{t}:=\max(u,\varphi-2t)$, $E_{t}:=E\cap\{u>\varphi-2t\}$, $F_{t}:=E\cap\{u\leq\varphi-2t\}$. Observe that $\varphi-2t\leq u_t\leq \varphi$. By plurifine locality and the claim we have that 
        \[
        \int_{E_{t}}\theta_{u}^{n}=\int_{E_{t}}\theta_{u_{t}}^{n}\leq(2t)^{n}{\rm Cap}_{\theta,\varphi}(E_{t})\leq(2t)^{n}{\rm Cap}_{\theta,\varphi}(E).
        \]
        On the other hand, using the inclusions 
        \[
        F_{t}\subset\left\{ \psi-1\leq
        \frac{u+\varphi}{2}-t\right\}
        \subset\{\psi-1\leq\varphi-t\}\subset\{\psi-1\leq-t\}
        \]
        and the comparison principle, Corollary \ref{cor: comparison principle},  we infer 
        \begin{flalign*}
                \int_{F_{t}}\theta_{u}^{n}& \leq \int_{\{ \psi-1\leq\frac{u+\varphi}{2}-t\}}\theta_{u}^{n}
                 \leq 2^n\int_{\{ \psi-1\leq\frac{u+\varphi}{2}-t\}}\theta_{\frac{u+\varphi}{2}}^{n} \\
                &\leq 2^{n}\int_{\{\psi\leq\varphi-t+1\}}\theta_{\psi}^{n} \leq 2^{n}\int_{\{\psi\leq-t+1\}}\theta_{\psi}^n. 
        \end{flalign*}
        Taking the supremum over all candidates $u$ we obtain
        \[
        \capa_{\theta,\psi}(E)\leq(2t)^{n}\capa_{\theta,\varphi}(E)+2^{n}\int_{\{\psi\leq-t+1\}}\theta_{\psi}^n.
        \]
        Set $t:=(\capa_{\theta,\varphi}(E))^{-1/2n}$. If $t>1$ we get 
        $\capa_{\theta,\psi}(E)\leq g\left(\capa_{\theta,\varphi}(E)\right)$, 
        where $g$ is defined on $[0,\infty)$ by 
        \[
        g(s):=(2^{n}+\vol(\{\theta\}))s^{1/2}+2^{n}\int_{\{\psi\leq-s^{-1/(2n)}+1\}}\theta_{\psi}^n.
        \]
        Observe that $g(0)=0$ since $\theta_{\psi}^n$ does not charge the pluripolar set $\{\psi=-\infty\}$.
        
        If $t\leq 1$ then $s:=\capa_{\theta,\varphi}(E)\geq 1$, and by the choice of $g$ above we have 
        $$\capa_{\theta,\psi}(E)\leq \vol(\{\theta\})\leq g\left(\capa_{\theta,\varphi}(E)\right),$$ finishing the proof. 
\end{proof}

\begin{theorem}
        \label{thm: Cap comparison} 
        Assume that $\psi\in \PSH(X,\theta)$. Then there exists a continuous function $f:[0,\infty)\rightarrow[0,\infty)$ with $f(0)=0$ such that, for any Borel set $E$, 
        \[
        \capa_{\theta,\psi}(E)\leq
        f\left(\capa_{\omega}(E)\right). 
        \]
\end{theorem}

\begin{proof}
	We can assume that $\theta\leq \omega$. Since $\psh(X, \theta)\subseteq \psh(X, \omega)$, for any Borel set $E$ we have $\capa_{\theta,\psi}(E)\leq \capa_{\omega,\psi}(E)$. Also, by Proposition \ref{prop: Cap omega phi and Cap omega} and Lemma \ref{lem: from cap psi to cap varphi} we get
	\[
	\capa_{\theta,\psi}(E)\leq \capa_{\omega,\psi}(E)\leq  g\left(\capa_{\omega, P_{\omega}[\psi]}(E)\right) \leq  g\left(C\capa_{\omega}(E)^{1/n}\right). 
	\]
\end{proof}

\section{Integration by parts}
\label{sec: integration by parts}
The integration by parts formula was recently obtained \cite{Xia19} using Witt Nystr\"om's construction. We give an argument here following \cite{Lu21} (see \cite{Vu21} for a more general result). We first start with the following key lemma.

\begin{lemma}
        \label{lem: pre integration by parts} 
        Let $\varphi_{1},\varphi_{2},\psi_{1},\psi_{2}\in\PSH(X,\theta)$ be such that $[\varphi_{1}]= [\varphi_{2}]$ and $[\psi_{1}] =[\psi_{2}]$. Then 
        \[
        \int_{X}(\varphi_{1}-\varphi_{2})\left(\theta_{\psi_{1}}^{n}-\theta_{\psi_{2}}^{n}\right)=\int_{X}(\psi_{1}-\psi_{2})(S_{1}-S_{2}),
        \]
        where $S_{j}:=\sum_{k=0}^{n-1}\theta_{\varphi_{j}}\wedge\theta_{\psi_{1}}^{k}\wedge\theta_{\psi_{2}}^{n-k-1}$, $j=1,2$. 
\end{lemma}

\begin{proof}
It follows from Proposition \ref{prop: comparison generalization} that 
\[
\int_{X}(\theta_{\psi_{1}}^{n}-\theta_{\psi_{2}}^{n})= \int_{X}(S_{1}-S_{2})=0.
\]
By adding a constant we can assume that $\varphi_{1},\varphi_{2},\psi_{1},\psi_{2}$ are negative and
that $\varphi_{1}\leq\varphi_{2}$ and $\psi_{1}\leq\psi_{2}$. Let $B>0$ be a constant such that 
        \[
        \varphi_{2}\leq\varphi_{1}+B,\quad  \psi_{2}\leq\psi_{1}+B.
        \]
        
\smallskip 

\noindent {\bf Step 1.} We assume $\theta$ is K\"ahler and $\psi_1,\psi_2,\varphi_1,\varphi_2$ are $\lambda\theta$-psh 
for some $\lambda\in(0,1)$. Observe that the last condition ensures that their Monge--Amp\`ere mass is strictly positive.
\medskip
\noindent {\bf Step 1.1.} We also assume that there exists $C>0$ such that $\psi_1=\psi_2$ on $U:=\{\min(\psi_1,\psi_2)<-C\}$ and $\varphi_1=\varphi_2$ on $V:=\{\min(\varphi_1,\varphi_2)<-C\}$. 
        
         For a function $u\in \psh(X, \theta)$ we consider its canonical approximant $u^{t}:=\max(u,-t)$, $t>0$. It follows from Stokes' theorem that 
        \[
        \int_{X}(\varphi_1^t-\varphi_2^t)\left(\theta_{\psi_{1}^{t}}^{n}-\theta_{\psi_{2}^{t}}^{n}\right)=\int_{X}(\psi_{1}^{t}-\psi_{2}^{t})(S_{1}^{t}-S_{2}^{t}),
        \]
        where $S_{j}^{t}:=\sum_{k=0}^{n-1}\theta_{\varphi_{j}^{t}}\wedge\theta_{\psi_{1}^{t}}^{k}\wedge\theta_{\psi_{2}^{t}}^{n-k-1}$, $j=1,2$. We now consider the limit as $t\to \infty$ in the above equality. 
        
         Fix $t>C$. By assumption we have $\varphi_1^t = \varphi_2^t$ on $V$ and $\{\varphi_1\leq -t\} = \{\varphi_2\leq -t\}\subset V$. The same properties for $\psi_1,\psi_2$ also hold: $\psi_{1}^{t}=\psi_{2}^{t}$ in the open set $U$ and $\{\psi_1\leq-t\}=\{\psi_2\leq-t\}\subset U$. It thus follows that
      \[
        {\bf 1}_U\theta_{\psi_1^t}^n={\bf 1}_U\theta_{\psi_2^t}^n\quad
        \text{and} \quad  {\bf 1}_VS_1^t={\bf
                1}_VS_2^t,
        \]    
        hence multiplying with the characteristic functions ${\bf 1}_{\{\psi_1\leq-t\}}$ and $ {\bf 1}_{\{\varphi_1\leq-t\}}$ respectively gives
        \[
        {\bf 1}_{\{\psi_1\leq-t\}}\theta_{\psi_1^t}^n={\bf 1}_{\{\psi_1\leq-t\}}\theta_{\psi_2^t}^n\quad 
        \text{and} \quad  {\bf 1}_{\{\varphi_1\leq-t\}}S_1^t={\bf
                1}_{\{\varphi_1\leq-t\}}S_2^t. 
        \]

         By plurifine locality of the non-pluripolar product we thus have
\begin{flalign*}
\int_X(\varphi_1^t-\varphi_2^t)\left(\theta_{\psi_1^t}^{n}-\theta_{\psi_2^t}^n\right)
&=  \int_{\{\psi_1>-t\} \cap \{\varphi_1>-t\}} (\varphi_1^t-\varphi_2^t)(\theta_{\psi_1^t}^n-\theta_{\psi_2^t}^n)\\
& = \int_{\{\psi_1>-t\} \cap \{\varphi_1>-t\}}(\varphi_1-\varphi_2)(\theta_{\psi_{1}}^{n}-\theta_{\psi_{2}}^{n}),
\end{flalign*}
and 
	\[
        \int_{X}(\psi_{1}^{t}-\psi_{2}^{t})(S_{1}^{t}-S_{2}^{t})=
        \int_{\{\psi_{1}>-t\}\cap\{\varphi_{1}>-t\}}(\psi_{1}-\psi_{2})(S_{1}-S_{2}).
        \]
      Since $\varphi_1-\varphi_2$ and $\psi_1-\psi_2$ are bounded,  using the dominated convergence theorem we finish Step 1.1.
      \medskip
        
\noindent {\bf Step 1.2.} We remove the assumption made in Step 1.1.\\
Recall that $\psi_{1},\psi_{2},\varphi_{1},\varphi_{2}$ are $\lambda\theta$-psh for some
$\lambda\in(0,1)$.  For each $\varepsilon\in(0,1-\lambda)$ we define 
        \[
        \psi_{2,\varepsilon}:=\max(\psi_1,(1+\varepsilon)\psi_2)\ ;\
        \varphi_{2,\varepsilon}:=\max(\varphi_1,(1+\varepsilon)\varphi_2).
        \]
        Since $\theta$ is assumed to be K\"ahler, and $\varepsilon\in (0,1-\lambda)$, the functions $(1+\varepsilon)\psi_{2}$ and  
        $(1+\varepsilon)\varphi_{2}$ are still $\theta$-psh. Also,  $\psi_{1}\leq\psi_{2,\varepsilon}\leq(1+\varepsilon)(\psi_{1}+B)$ and 
        $\varphi_{1}\leq\varphi_{2,\varepsilon}\leq(1+\varepsilon)(\varphi_{1}+B)$. These are $\theta$-psh functions satisfying the 
        assumptions in Step 1.1 with $C=B+B\varepsilon^{-1}$. Indeed, if $\varphi_{1}(x)<-C$ then
        \[
        (1+\varepsilon)\varphi_{2}(x)=\varphi_{2}(x)+\varepsilon\varphi_{2}(x)\leq\varphi_{1}(x)+
        B+\varepsilon(B-C) =\varphi_{1}(x).
        \]
        It then follows that $\varphi_{2, \varepsilon}=\varphi_1$ on $V=\{\varphi_1=\min(\varphi_1, \varphi_{2, \varepsilon})<-C\}$. Similarly we have that  $\psi_{2, \varepsilon}=\psi_1$ on $U=\{\psi_1=\min(\psi_1, \psi_{2, \varepsilon})<-C\}$, with the same $C$.\\
        We can thus apply Step 1.1 to $\psi_{1}$, $\psi_{2,\varepsilon}$, $\varphi_{1}$,$\varphi_{2,\varepsilon}$
         to obtain 
        \[
        \int_X (\varphi_{1}-\varphi_{2,\varepsilon}) \left(\theta_{\psi_{1}}^{n}-\theta_{\psi_{2,\varepsilon}}^n\right)
        =\int_{X}\left(\psi_{1}-\psi_{2,\varepsilon}\right)(S_{1,\varepsilon}-S_{2,\varepsilon}),
        \]
where 
$S_{1,\varepsilon}:=\sum_{k=0}^{n-1}\theta_{\varphi_1}\wedge\theta_{\psi_1}^k\wedge\theta_{\psi_{2,\varepsilon}}^{n-k-1}$ and $S_{2,\varepsilon}:=\sum_{k=0}^{n-1}\theta_{\varphi_{2,\varepsilon}}\wedge\theta_{\psi_1}^k\wedge\theta_{\psi_{2,\varepsilon}}^{n-k-1}$. \\
By Theorem \ref{thm: Cap comparison} there exists a continuous function $f:[0,\infty)\rightarrow[0,\infty)$ with $f(0)=0$ such that for every Borel set $E$, 
        \[
        \capa_{\theta,\psi}(E)\leq f(\capa_{\theta}(E)),
        \]
where 
\[
\psi:=\frac{\varphi_1+\varphi_2+\psi_1+\psi_2}{5}-B.
\]
Note that $\psi$ is $\theta$-psh and  $\int_X\theta_{\psi}^n>0$. Indeed, recalling that in this step $\theta$ is K\"ahler, we have
\[
\int_X \theta_{\psi}^n = 5^{-n}\int_X (\theta+\theta_{\varphi_1}+\theta_{\varphi_2}
+\theta_{\psi_1}+\theta_{\psi_2})^n \geq 5^{-n} \theta^n>0. 
\]
Since we have assumed that 
$\psi_2-B\leq \psi_1\leq \psi_2\leq 0$ and $\varphi_2-B\leq \varphi_1\leq \varphi_2\leq 0$, we get
        \[
        \varphi_2-B\leq \varphi_1 \leq \varphi_{2,\varepsilon} = \max(\varphi_1,(1+\varepsilon)\varphi_2)
        \leq \max(\varphi_1,\varphi_2) = \varphi_2,
        \]
and 
        \[
        \psi_2-B \leq \psi_1\leq \psi_{2,\varepsilon}=\max(\psi_1,(1+\varepsilon \psi_2) )
        \leq \max(\psi_1,\psi_2) = \psi_2. 
        \]
In particular $\psi_1\simeq\psi_2\simeq \psi_{2,\varepsilon} $ and $\varphi_1\simeq\varphi_2\simeq \varphi_{2,\varepsilon}$.\\
Set $u_\varepsilon:=\frac{1}{5}( \varphi_1+\varphi_{2,\varepsilon}+\psi_{2,\varepsilon}+\psi_1)$. Using the above inequalities we get
        \begin{flalign*}
        \psi &\leq \frac{\varphi_{1}+\psi_{1}+\varphi_{2}+\psi_{2}}{5} -\frac{2B}{5} \\
        & \leq \frac{\varphi_{1}+\psi_{1}+\varphi_{2,\varepsilon}+\psi_{2,\varepsilon}}{5} = u_\varepsilon\\
        &\leq \frac{\varphi_{1}+\psi_{1}+\varphi_{2}+\psi_{2}}{5}
       = \psi+B. 
        \end{flalign*}
Also observe that there exists a positive constant $C'$ such that
$$S_{j,\varepsilon}\leq C'(5\theta+dd^c(\varphi_1+\varphi_{2,\varepsilon}+\psi_{2,\varepsilon}+\psi_1))^n= C' 5^n \theta_{u_\varepsilon}^n$$
We then obtain that
for any Borel set $E$ and any $\varepsilon\in(0,1-\lambda),\; j=1,2$, 
\[
\int_{E}S_{j,\varepsilon}\leq C'5^nB^n \capa_{\theta,\psi}(E) \leq  C''f(\capa_{\theta}(E)),
\]
where the first inequality follows from the arguments at the beginning of the proof of Lemma \ref{lem: from cap psi to cap varphi}, and the second follows from Theorem \ref{thm: Cap comparison}. 

 Since $\psi_{2,\varepsilon}$ and $\varphi_{2,\varepsilon}$ are increasing to $\psi_{2}$ and $\varphi_{2}$ respectively, for each $j\in\{1,2\}$ we also have that $S_{j,\varepsilon}\to S_j$ and
$\theta_{\psi_{2,\varepsilon}}^n\to\theta_{\psi_2}^n$ as $\varepsilon \to 0$ in the weak sense 
of Radon measures (see Theorem \ref{thm: lsc of MA measures} and Remark \ref{rem: increasing implies capacity}). 
By the above, these measures are uniformly dominated by $\capa_{\theta}$. \\
Note also that  $\varphi_1-\varphi_{2,\varepsilon},\varphi_1-\varphi_2,\psi_{2,\varepsilon}-\psi_1,\psi_2-\psi_1$ are uniformly bounded and quasi-continuous (because difference of quasi-psh functions). 
Moreover, $\psi_{2,\varepsilon}-\psi_1\to\psi_2-\psi_1$, and $\varphi_1-\varphi_{2,\varepsilon}\to\varphi_1-\varphi_2$ in capacity as $\varepsilon \to 0$.  
It  thus follows from Lemma \ref{lem: convergence} that 
\[
 \lim_{\varepsilon\to 0} \int_{X}(\varphi_1-\varphi_{2,\varepsilon})\left(\theta_{\psi_1}^n-
 \theta_{\psi_{2,\varepsilon}}^n\right) = \int_X (\varphi_1-\varphi_2)\left(\theta_{\psi_1}^n-
 \theta_{\psi_2}^n\right)
 \]
and 
\[
  \lim_{\varepsilon\to 0} \int_X\left(\psi_1-\psi_{2,\varepsilon}\right)(S_{1,\varepsilon}-S_{2,\varepsilon}) = 
  \int_X\left(\psi_1-\psi_2\right)(S_1-S_2),
 \]
finishing the proof of Step 1.2.
        \medskip

\noindent {\bf Step 2.} We merely assume that $\{\theta\}$ is big. 
We can assume that $ -\omega\leq \theta\leq \omega$. 
For $s>2$ we consider $\theta_s:=\theta+s\omega$, which is K\"ahler, and we observe that
$\varphi_1, \varphi_2, \psi_1, \psi_2$ are $\lambda\theta_s$-psh for some $\lambda\in (0,1)$. 
We can thus apply the first step to get 
        \[
        \int_{X}u\left((\theta_{s}+dd^{c}\psi_{1})^{n}-(\theta_{s}+dd^{c}\psi_{2})^{n}\right)=\int_{X}v
        T_{s},
        \]
where $u=\varphi_{1}-\varphi_{2}$, $v=\psi_{1}-\psi_{2}$ and 
        \begin{flalign*}
                T_{s}= & \sum_{k=0}^{n-1} (\theta_s +dd^c
                \varphi_1)\wedge(\theta_s+dd^{c}\psi_{1})^{k}\wedge(\theta_{s}+dd^{c}\psi_{2})^{n-k-1}\\
                - & \sum_{k=0}^{n-1}(\theta_s+dd^c
                \varphi_2)\wedge(\theta_{s}+dd^{c}\psi_{1})^{k}\wedge(\theta_{s}+dd^{c}\psi_{2})^{n-k-1}.
        \end{flalign*}
        We thus obtain an equality between two polynomials in $s$. Identifying the coefficients we arrive at 
        the conclusion.
\end{proof}

Next we prove the integration by parts formula, extending the one in \cite{BEGZ10} which only applies to the case of potentials with small unbounded locus.

\begin{theorem}\label{thm: int_by_parts} Let $u,v \in L^\infty(X)$ be differences of quasi-psh functions, and $\phi_j \in \textup{PSH}(X,\theta^j)$, $j \in  \{1, \ldots,n-1\}$ with $\{\theta^j\}$ big. Then
$$\int_X u dd^c v  \wedge \theta^1_{\phi_1} \wedge \ldots \wedge \theta^{n-1}_{\phi_{n-1}}= \int_X v dd^c u \wedge \theta^1_{\phi_1} \wedge \ldots \wedge \theta^{n-1}_{\phi_{n-1}}.$$
\end{theorem}

\begin{proof}
We first assume that $\theta$ is K\"ahler, $u=\varphi_1-\varphi_2$ and $v=\psi_1-\psi_2$ where 
$\psi_1,\psi_2,\varphi_1,\varphi_2$ are $\theta$-psh. Fix $\phi\in{\rm PSH}(X,\theta)$ 
and for each $s\in[0,1]$, $j=1,2$, we set $\psi_{j,s}:=s\psi_{j}+(1-s)\phi$. 
Note that $\psi_{1,s}\simeq \psi_{2,s}$. 
It follows from Lemma \ref{lem: pre integration by parts} that for any $s\in[0,1]$,
\[
\int_X u\left(\theta_{\psi_{1,s}}^n-\theta_{\psi_{2,s}}^n\right)= \int_X (\psi_{1,s}-
\psi_{2,s}) T_s=\int_X sv T_s,
\]
where 
\[
T_s:= \sum_{k=0}^{n-1}\theta_{\varphi_1} \wedge \theta_{\psi_{1,s}}^k \wedge \theta_{\psi_{2,s}}^{n-k-1}
-\sum_{k=0}^{n-1}\theta_{\varphi_2} \wedge \theta_{\psi_{1,s}}^k \wedge \theta_{\psi_{2,s}}^{n-k-1}.
\]
We thus have an identity between two polynomials in $s$. 
We then compute the first derivative in $s=0$ and we find that for $j=1,2$
$$\frac{\partial}{\partial s} (s\theta_{\psi_j} + (1-s)\theta_{\phi})^n |_{s=0} = n \theta_{\psi_j} \wedge \theta_{\phi}^{n-1} -n \theta_\phi^n $$
and
$$\frac{\partial}{\partial s} (sT_s)|_{s=0}=T_0.$$
Noticing that $T_0= n dd^c u \wedge \theta_\phi^{n-1}$, we obtain 

\[
\int_X udd^c v\wedge\theta_{\phi}^{n-1}=\int_Xvdd^c u\wedge\theta_{\phi}^{n-1}.
\]
For the general case, i.e. $\{\theta\}$ is merely big, we can write 
$u=\varphi_1-\varphi_2$ and $v=\psi_1-\psi_2$, where $\psi_1,\psi_2,\varphi_1,\varphi_2$ are 
$A\omega$-psh, for some $A>0$ large enough. 
We apply the first step with $\theta$ replaced by $\theta + t\omega$, for $t>A$ to get
\[
\int_Xu dd^c v\wedge(t\omega+\theta_
{\phi})^{n-1}=\int_X vdd^c u\wedge(t\omega+\theta_{\phi})^{n-1}.
\]
Identifying the coefficients of these two polynomials in $t$ we obtain
\[
\int_X udd^c v\wedge\theta_{\phi}^{n-1}=\int_{X}vdd^c u\wedge\theta_{\phi}^{n-1}.
\] 
We now consider $\theta=s_1\theta_1+....+s_{n-1} \theta_{n-1}$, $\phi:=s_1\phi_1+...+s_{n-1}\phi_{n-1}$ 
with $s_1,...,s_{n-1}\in[0,1]$ and $\sum_{j=1}^{n-1} s_j=1$. 
We obtain an identity between two polynomials in $(s_1,...,s_{n-1})$, and identifying the coefficients we arrive at the result.
\end{proof}

\chapter{Complex Monge--Amp\`ere equations with prescribed singularity type}\label{sec 4}

Let $\theta$ be a smooth closed real $(1,1)$-form on $X$ such that $\{\theta\}$ is big and $\phi \in \textup{PSH}(X,\theta)$. By $\mathrm{PSH}(X,\theta,\phi)$  we denote the set of $\theta$-psh functions that are more singular than $\phi$. We say that $v \in \textup{PSH}(X,\theta,\phi)$ has \emph{relatively minimal singularity type} if $v$ has the same singularity type as $\phi$. Clearly, $\mathcal E(X,\theta,\phi) \subset \textup{PSH}(X,\theta,\phi)$. 

Let $\mu$ be a non-pluripolar positive  measure on $X$ such that $\mu(X)=\int_X \theta_\phi^n > 0$.
Our aim is to study existence and uniqueness of solutions to the following equation of complex Monge--Amp\`ere type:
\begin{equation}\label{eq: CMAE_sing}
\theta_u^n = \mu, \ \ \ u \in \mathcal E(X,\theta,\phi).
\end{equation}
It is not hard to see that this equation does not have a solution for arbitrary $\phi$. Indeed, suppose for the moment that $\theta = \omega$, and  choose $\phi \in \mathcal E(X,\omega) :=\mathcal E(X,\omega,0)$ unbounded. It is clear that $\mathcal E(X,\omega,\phi) \subsetneq \mathcal E(X,\omega,0)$. By \cite[Theorem A]{BEGZ10}, the (trivial) equation $\omega_{u}^n = \omega^n, \ u \in \mathcal E(X,\omega,0)$ is \emph{only} solved by potentials $u$ that are constant over $X$, hence we cannot have $u \notin \mathcal E(X,\omega,\phi)$.

This simple example suggests that we need to be more selective in our choice of $\phi$, to make \eqref{eq: CMAE_sing} well posed. As it turns out, the natural choice is to take $\phi$ such that  $P_{\theta}[\phi]=\phi$ (see Theorem \ref{thm: naturality_of_model} for concrete evidence). Therefore, for the rest of this section we ask that $\phi$ additionally satisfies:
\begin{equation}\label{eq: fixed_point}
\phi=P_{\theta}[\phi].
\end{equation}
We recall that such a potential $\phi$  is model, and $[\phi]$ is a \emph{model type singularity}. As $V_\theta = P_\theta[V_\theta]$, one can think of such $\phi$ as generalizations of $V_\theta$, the potential with minimal singularity from \cite{BEGZ10}. 

 One wonders if maybe model type potentials (those that satisfy \eqref{eq: fixed_point}) always have small unbounded locus. Sadly, this is not the case, as we now point out. Suppose $\theta$ is a K\"ahler form, and $\{x_j\}_j \subset X$ is a dense countable subset. Also let $v_j \in \textup{PSH}(X,\theta)$ such that $v_j < 0$, $\int_X (-v_j) \theta^n = 1$, and $v_j$ has a positive Lelong number at $x_j$. Then $u = \sum_j \frac{1}{2^j} v_j \in \textup{PSH}(X,\theta)$ has positive Lelong numbers at all $x_j$. As we have argued in Lemma \ref{lem: Lelong} below, the Lelong numbers of $P_{\theta}[u]$ are the same as those of $u$, hence the model type potential $P_{\theta}[u]$ cannot have small unbounded locus. 

\begin{lemma} \label{lem: Lelong}Suppose that $u \in \textup{PSH}(X,\theta)$. Then $\mathcal I (tu) = \mathcal I(tP_\theta[u])$ for all $t>0$. In particular, all the Lelong numbers of $u$ and $P_\theta[u]$ are the same.
\end{lemma}

Recall that $\mathcal I(v)$ is the multiplier ideal sheaf of a quasi-psh function $v$, whose germs are defined by 
 $$\mathcal I(v, x):=\left \{ f\in {\mathcal O}_x \ :\ \int_U |f|^2 e^{-v} \omega^n < +\infty \quad  {\rm for\ some \ open \ subset\ }  U \subset X, x\in U \right\}.$$

We use the valuative criteria of integrability of Boucksom--Favre--Jonsson. For a more elementary argument see the proof of \cite[Theorem 1.1]{DDL1}.

\begin{proof} Since $P_\theta(u+C,V_\theta) \nearrow P_\theta[u]$ a.e., as $C \nearrow \infty$, from the Guan--Zhou openess theorem \cite{GZh15} it follows that $\mathcal I (tP_\theta(u+C,V_\theta)) = \mathcal I(tP_\theta[u])$ for $C>0$ big enough. However $[P_\theta(u+C,V_\theta)] =[u]$ for all $C>0$, so we get that $\mathcal I (tu) = \mathcal I(tP_\theta[u])$.
 The conclusion about Lelong numbers follows from the equivalence between (i) (applied with the proper modification $\pi=id$) and (ii) in \cite[Theorem A]{BFJ08}.
\end{proof}

\section{The relative finite energy class}

To develop the variational approach to \eqref{eq: CMAE_sing}, we study the relative version of  the Monge--Amp\`ere energy, and its bounded locus $\mathcal E^1(X,\theta,\phi)$. For $u\in \Ec(X,\theta, \phi) $ with relatively minimal singularity type, we define the Monge--Amp\`ere energy of $u$ relative to $\phi$ as 
\[
\mathrm{I}_{\phi} (u) :=\frac{1}{n+1} \sum_{k=0}^n \int_X (u-\phi) \theta_u^k \wedge\theta_{\phi}^{n-k}. 
\]
Before we study this energy closely, we note the following inequality:

\begin{lemma}\label{lem: MA_monotone_ineq}

Assume $u, v \in \mathcal{E}(X,\theta,\phi)$ have the same singularity type. Then
\[
\int_X (u-v)\theta_u^k \wedge \theta_v^{n-k} \leq \int_X (u-v) \theta_u^{k-1}\wedge \theta_v^{n-k+1}.  
\]
\end{lemma}

\begin{proof}
Adding a constant to $u$ does not affect the inequality (by Theorem \ref{thm: BEGZ_monotonicity_full}), so we can assume that $v\leq u$.  
By the partial comparison principle (Proposition \ref{prop: general CP}), we have 
\begin{flalign*}
    \int_X (u-v)\theta_u^k \wedge \theta_v^{n-k}& = \int_0^{\infty} \theta_u^k \wedge \theta_v^{n-k} (u>v+t) dt \\
    & \leq \int_0^{\infty} \theta_u^{k-1} \wedge \theta_v^{n-k+1} (u>v+t) dt\\
    &= \int_X (u-v)\theta_u^{k-1} \wedge \theta_v^{n-k+1}. 
\end{flalign*}
\end{proof}

In the next theorem we collect basic properties of the Monge--Amp\`ere energy:

\begin{theorem}\label{thm: basic I energy} Suppose $u,v \in \mathcal E(X,\theta,\phi)$ have relatively minimal singularity type. Then:\\
\noindent (i) $ \mathrm{I}_{\phi}(u)-\mathrm{I}_{\phi}(v) = \frac{1}{n+1}\sum_{k=0}^n \int_X (u-v) \theta_{u}^k \wedge \theta_{v}^{n-k}.$ In particular $\mathrm{I}_{\phi}(u)\leq \mathrm{I}_{\phi}(v)$ if $u\leq v$.\\
\noindent (ii) If $u\leq \phi$ then, $
\int_X (u-\phi) \theta_u^n \leq I_{\phi}(u) \leq \frac{1}{n+1} \int_X (u-\phi) \theta_{u}^n. $ \\
\noindent (iii) $\mathrm{I}_{\phi}$ is concave along affine curves. Also, the following estimates hold: $$
	\int_X (u-v) \theta_u^n \leq I_{\phi}(u) -I_{\phi}(v) \leq \int_X (u-v) \theta_v^n.$$
\end{theorem}

\begin{proof}
Using Theorem \ref{thm: int_by_parts}, it is possible to repeat  the arguments of  \cite[Proposition 2.8]{BEGZ10}, almost word for word. As a courtesy, we present a detailed proof.

We compute the derivative of  $f(t):=I_{\phi}(u_t), t\in [0,1]$, where $u_t:=tu +(1-t)v$. By the multi-linearity property of the non-pluripolar product we see that $f(t)$ is a polynomial in $t$. Using integration by parts (Theorem \ref{thm: int_by_parts}), one can check the following formula:
\begin{eqnarray*}
f'(t) &=&\frac{1}{n+1}\bigg(\sum_{k=0}^n \int_X (u-v) \theta_{u_t}^k\wedge \theta_{\phi}^{n-k} +\sum_{k=1}^n \int_X k(u_t-\phi) dd^c (u-v)\wedge \theta_{u_t}^{k-1}\wedge \theta_{\phi}^{n-k}\bigg)\\
&=& \frac{1}{n+1}\bigg(\sum_{k=0}^n \int_X (u-v) \theta_{u_t}^k\wedge \theta_{\phi}^{n-k} +\sum_{k=1}^n \int_X k(u-v) (\theta_{u_t}-\theta_{\phi})\wedge \theta_{u_t}^{k-1}\wedge \theta_{\phi}^{n-k}\bigg)\\
&=& \int_X (u-v) \theta_{u_t}^n. 
\end{eqnarray*}

Computing one more derivative, we arrive at
$$f''(t)=n \int_X (u-v) dd^c (u-v) \wedge \theta_{u_t}^{n-1}= n \int_X (u-v)  (\theta_u-\theta_v) \wedge \theta_{u_t}^{n-1}\leq 0,$$
where the inequality follows from Lemma \ref{lem: MA_monotone_ineq}.

Now, the function $t\mapsto f'(t)$ is continuous on $[0,1]$, thanks to convergence property of the Monge--Amp\`ere operator (see Theorem \ref{thm: lsc of MA measures}). It thus follows that 
\[
I_{\phi}(u_1) -I_{\phi}(u_0) = \int_0^1 f'(t)dt = \int_0^1 \int_X (u-v) \theta_{u_t}^n dt. 
\]
Using the multi-linearity of the non-pluripolar product again, we get that
\begin{eqnarray*}
\int_0^1 \int_X (u-v) \theta_{u_t}^n dt &=& \sum_{k=0}^n \left(\int_0^1\binom{n}{k} t^k (1-t)^{n-k} dt\right) \int_X   (u-v) \theta_{u}^k \wedge \theta_v^{n-k}\\
&=&  \frac{1}{n+1}\sum_{k=0}^n \int_X   (u-v) \theta_{u}^k \wedge \theta_v^{n-k}.
\end{eqnarray*}
This verifies (i), and another application of Lemma \ref{lem: MA_monotone_ineq} finishes the proof of (iii).

 For (ii) we observe that since $u-\phi\leq $ each term $ \int_X (u-\phi)\theta_u^k \wedge \theta_\phi^{n-k}$ is negative, hence $I_\phi(u) \leq \int_X (u-\phi) \theta_{u}^n$. The left-hand side inequality of (ii) follows from Lemma \ref{lem: MA_monotone_ineq}. 
\end{proof}

\begin{lemma}\label{lem: convergence of AM} Suppose $u_j,u \in \mathcal{E}(X,\theta,\phi)$ have relatively minimal singularity type such that $u_j\searrow u$. Then $\AM_{\phi}(u_j)\searrow \mathrm{I}_{\phi}(u)$.   
\end{lemma}
\begin{proof} From Theorem \ref{thm: basic I energy} it follows that $$|\AM_\phi(u_j) - \AM_\phi(u)|=\AM_\phi(u_j) - \AM_\phi(u) \leq \int_X (u_j - u)\theta_u^n.$$ 
An application of the dominated convergence theorem finishes the argument.
\end{proof}

We can now define the Monge--Amp\`ere energy for arbitrary $u\in \mathrm{PSH}(X,\theta,\phi)$ using  a familiar formula:
\[
\AM_{\phi}(u) : =\inf \{\AM_{\phi}(v) \setdef  v\in \Ec(X,\theta,\phi), \; v\ \textrm{has relatively minimal singularity type, and } u\leq v\}. 
\]
\begin{lemma}\label{lem: easy convergence AMO}
	If $u\in \psh(X,\theta,\phi)$ then $\AMO(u)=\lim_{t\to \infty} \AMO(\max(u,\phi-t))$. 
\end{lemma}
\begin{proof}
	It follows from the above definition that $\AMO(u)\leq \lim_{t\to \infty} \AMO(\max(u,\phi-t))$. Assume now that $v\in \psh(X,\theta,\phi)$ is such that $u\leq v$, and $v$ has the same singularity type as $\phi$ (i.e. $v$ is a candidate in the definition of $I_{\phi}(u)$). Then for $t$ large enough we have $\max(u,\phi-t)\leq v$, hence the other inequality follows from monotonicity of $\AMO$.
\end{proof}

Let $\mathcal{E}^1(X,\theta,\phi)$ be the set of all $u\in \psh(X,\theta,\phi)$  such that $\AMO(u)$ is finite. As a result of Lemma \ref{lem: easy convergence AMO} and Theorem \ref{thm: basic I energy}$(i)$ we observe that $\AMO$ is non-decreasing in $\psh(X,\theta,\phi)$. Consequently, $\mathcal{E}^1(X,\theta,\phi)$ is stable under the max operation, moreover we have the following familiar characterization of $\mathcal E^1(X,\theta,\phi)$:

\begin{lemma}\label{charcterization class E^1}
Let $u\in \mathrm{PSH}(X,\theta,\phi)$. Then $u\in \mathcal{E}^1(X,\theta,\phi)$ if and only if $u\in \mathcal{E}(X,\theta,\phi)$ and $\int_X (u-\phi) \theta_u^n >-\infty$. 
\end{lemma}
\begin{proof} Let $u\in \mathcal{E}^1(X,\theta,\phi)$. We can assume that $u\leq \phi$. For each $C>0$ we set $u^C:= \max(u, \phi-C)$. 

If $I_{\phi}(u) >-\infty$ (say $I_{\phi}(u) >-A$), then by the monotonicity property we have $I_\phi(u^C) \geq  I_{\phi}(u)$. Since $u^C \leq \phi$, an application of Theorem \ref{thm: basic I energy}$(ii)$ gives that $\int_X (u^C-\phi) \theta_{u^C}^n\geq -A, \ \forall C$, for some  $A>0$. From this we  obtain that 
\[
\int_{\{u\leq \phi-C\}} \theta_{u^C}^n\leq \int_{\{u\leq \phi-C\}} \frac{\phi-u}{C} \,\theta_{u^C}^n \leq \int_{X} \frac{\phi-u^C}{C} \, \theta_{u^C}^n \leq  \frac{A}{C} \to 0, 
\]
as $C\to \infty$. By plurifine locality and the above, we have 
\[
\int_{\{u>\phi-C\}} \theta_u^n = \int_{\{u>\phi-C\}} \theta_{u^C}^n = \int_X \theta_{u^C}^n -  \int_{\{u\leq \phi-C\}} \theta_{u^C}^n =\int_X \theta_{\phi}^n -  \int_{\{u\leq \phi-C\}} \theta_{u^C}^n \to \int_X \theta_{\phi}^n
\]
as $C\to \infty$, showing that $u\in \mathcal{E}(X,\theta,\phi)$. Moreover,
\[
\int_X (u^C -\phi) \theta_{u^C}^n \leq \int_{\{u>\phi-C\}} (u-\phi)\theta_u^n.
\]
Letting $C\to \infty$ we see that $\int_X (u-\phi)\theta_u^n >-A$. 

To prove the converse statement,  assume that $u\in \mathcal{E}(X,\theta,\phi)$ and $\int_X (u-\phi)\theta_u^n >-\infty$.
For each $C>0$ the measures $\theta_u^n$ and $\theta_{u^C}^n$ have the same total mass and they coincide on $\{u>\phi-C\}$.  It follows that $\int_{\{u\leq \phi-C\}} \theta_{u^C}^n =\int_{\{u\leq \phi-C\}} \theta_{u}^n$, and from this we deduce that 
\begin{eqnarray*}
\int_X (u^C-\phi)  \theta_{u^C}^n &=&-\int_{\{u\leq \phi-C\}} C \theta_u^n  + \int_{\{u>\phi-C\}}(u-\phi) \theta_{u}^n \geq  \int_X (u-\phi) \theta_u^n>-A. 
\end{eqnarray*}
It thus follows from Theorem \ref{thm: basic I energy}$(ii)$ that $I_{\phi}(u^C)$ is uniformly bounded.  Finally, it follows from Lemma \ref{lem: easy convergence AMO}  that $I_{\phi}(u^C)\searrow I_{\phi}(u)$ as $C\to \infty$, finishing the proof. 
\end{proof}

We finish this section with a series of standard results listing various properties of the class $\mathcal E^1(X,\theta,\phi)$.

\begin{lemma}\label{lem: AM continuous along decreasing sequence}
Assume that $u_j,u \in \mathcal{E}^1(X,\theta,\phi)$ such that $u_j \searrow u$. Then $\AMO(u_j)$ decreases to $\AMO(u)$.
\end{lemma}
\begin{proof}
Without loss of generality we can assume that $u_j\leq \phi$ for all $j$. For each $C>0$ we set $u_j^C:=\max (u_j, \phi-C)$ and $u^C:= \max(u, \phi-C)$. Note that $u_j^C, u^C$ have the same singularity type as $\phi$. Then Lemma \ref{lem: convergence of AM} insures that $\lim_{j} I_{\phi}(u_j^C) = I_{\phi}(u^C)$. Monotonicity of $I_\phi$ gives now that $I_{\phi}(u) \leq \lim_j I_\phi(u_j) \leq \lim_j I_\phi(u_j^C)= I_\phi(u^C)$. Letting $C \to \infty$, the result follows from Lemma \ref{lem: easy convergence AMO}.
\end{proof}

\begin{lemma}
	\label{lem: criteria in E1}
	Assume that $\{u_j\}_j \subset \Ec^1(X,\theta,\phi)$  is decreasing, such that $\AMO(u_j)$ is uniformly bounded. Then the limit $u:=\lim_{j} u_j$ belongs to $\Ec^1(X,\theta,\phi)$ and $\AMO(u_j)$ decreases to $\AMO(u)$.
\end{lemma}
\begin{proof}
We can assume that $u_j\leq \phi$ for all $j$. As all the  terms in the definition of $I_\phi(u_j)$ are negative, we notice that $- C \leq (n+1)\AMO(u_j)\leq \int_X (u_j-\phi) \theta_{\phi}^n \leq 0$ for some $C > 0$.

If $u \equiv -\infty$, then $\sup_X u_j = \sup_X (u_j - \phi) \to -\infty$ (Lemma \ref{lem: model sup over X}). This implies that $\int_X (u_j-\phi) \theta_{\phi}^n \leq \sup_X (u_j - \phi) \int_X \theta_\phi^n \to -\infty$, a contradiction. Hence $u \in \textup{PSH}(X,\theta)$. 

By continuity along decreasing sequences (Lemma \ref{lem: AM continuous along decreasing sequence}) we have $$\lim_{j\rightarrow \infty} \AMO(\max(u_j, \phi-C))= \AMO(\max(u, \phi-C)).$$  
It follows that  $\AMO(\max(u,\phi-C))$ is uniformly bounded. Lemma \ref{lem: easy convergence AMO} then insures that  $\AMO(u)$ is finite, i.e., $u\in \Ec^1(X,\theta,\phi)$. 
\end{proof}

\begin{coro}\label{cor: I_conc_on_E}
$\AMO$ is concave along affine curves in $\psh(X,\theta,\phi)$. In particular, the set $\Ec^1(X,\theta,\phi)$ is convex. 
\end{coro}
\begin{proof}
Let $u,v\in \psh(X,\theta,\phi)$ and $u_t:=tu+(1-t)v, t\in (0,1)$. If one of $u,v$ is not in $\Ec^1(X,\theta,\phi)$ then the conclusion is obvious. So, we can assume that both $u$ and $v$ belong to $\Ec^1(X,\theta,\phi)$. For each $C>0$ we set $u_t^C:=t \max(u,\phi-C) + (1-t)\max(v,\phi-C)$. By Theorem \ref{thm: basic I energy}$(iii)$, $t \to \AMO(u_t^C)$ is concave. Since $u_t^C$ decreases to $u_t$ as $C\to \infty$, Lemma  \ref{lem: criteria in E1} gives the conclusion. 
\end{proof}

\begin{lemma}\label{lem: inequalities energy}
Suppose $u,v \in \mathcal E^1(X,\theta,\phi)$ have the same singularity type. Then 
$$ \int_X (u-v) \theta_{u}^n \leq \AMO(u)-\AMO(v) \leq   \int_X (u-v) \theta_{v}^n.
$$
\end{lemma}

With a bit of extra work, one can also get rid of the assumption that $u$ and $v$ have the same singularity type \cite[Proposition 2.5]{DDL3}.

\begin{proof}
First, note that these estimates hold for $u^C := \max(u,\phi -C), v^C:=\max(v,\phi -C)$, by Theorem \ref{thm: basic I energy}$(iii)$. It is easy to see that $u^C - v^C$ is uniformly bounded and converges in capacity to $u - v$.  Putting these last two facts together, Theorem \ref{thm: lsc of MA measures} gives that 
\begin{flalign*}
\bigg|\int_X (u^C - v^C) \theta_{v^C}^n-\int_X (u - v) \theta_{v}^n\bigg| \to 0, \ \ \ \bigg|\int_X (u^C - v^C) \theta_{u^C}^n-\int_X (u - v) \theta_{u}^n\bigg| \to 0.
\end{flalign*}
The result follows from Lemma \ref{lem: easy convergence AMO}.
\end{proof}

\section{Monge--Amp\`ere equations with prescribed singularity type}

Throughout this section  $\phi$ is a $\theta$-psh function satisfying 
$\phi=P_\theta[\phi]$,  and $\int_X \theta_{\phi}>0$. We additionally assume the normalization  
$$\int_X \theta_{\phi}^n=1.$$
This can always be achieved by rescaling our big class $\{\theta\}$. In this section we consider the following complex Monge--Amp\`ere equation in our prescribed setting: 

\begin{equation}
\label{eq: MA lambda}
\theta_u^n = e^{\lambda u} \mu, \ u\in \Ec(X,\theta,\phi). 
\end{equation}
where $\lambda \geq 0$, $\mu$ is a positive non-pluripolar probability measure on $X$.  

When $\lambda>0$, we adapt the variational method in \cite{BBGZ13} to our needs. When $\lambda=0$, it is also possible to use the variational method to solve the equation, but it requires a detailed study of the relative Monge--Amp\`ere capacities, as carried out in \cite{DDL2}.  In this survey we provide a simpler approach by using the solutions for the case $\lambda>0$ and letting $\lambda\to 0^+$.

\begin{prop}\label{prop: usc of AMO}
$\AMO:\Ec^1(X,\theta,\phi) \to \mathbb R$  is upper semicontinuous with respect to the weak $L^1$ topology of potentials. 
\end{prop}
\begin{proof}
	Assume that $\{u_j\}_j$ is a sequence in $\mathcal{E}^1(X,\theta,\phi)$ $L^1$-converging to $u\in \Ec^1(X,\theta,\phi)$. We can assume that $u_j\leq 0$ for all $j$.  For each $k,\ell \in \mathbb{N}$ we set $v_{k,\ell}:= \max(u_{k},...,u_{k+\ell})$. As $\Ec^1(X,\theta,\phi)$ is stable under the max operation,  we have that $v_{k,\ell}\in \Ec^1(X,\theta,\phi)$. 

Moreover $v_{k,\ell}\nearrow \varphi_k:= \textup{usc}\left(\sup_{j \geq k} u_j\right)$, hence by the monotonicity property we get $\AMO(\varphi_k)\geq \AMO(v_{k,\ell})\geq \AMO(u_k)>-\infty$. As a result, $\varphi_k\in  \Ec^1(X,\theta,\phi)$.
     By Hartogs' lemma $\varphi_k \searrow u$ as $k\to \infty$.  By Lemma \ref{lem: AM continuous along decreasing sequence} it follows that $\AMO(\varphi_k)$ decreases to $\AMO(u)$. Thus, using the monotonicity of $\AMO$ we get $\AMO(u)= \lim_{k\rightarrow \infty} \AMO(\varphi_k) \geq \limsup_{k\rightarrow \infty} \AMO(u_k),$ finishing the proof.
\end{proof}

Next we describe the first order variation of $I_\phi$, shadowing a result from \cite{BB10}:

\begin{prop}\label{prop: derivative of AMO}
Let $u\in \mathcal{E}^1(X,\theta,\phi)$ and $\chi$ be a continuous function on $X$. For each $t\in \mathbb R$ set $u_t:= P_{\theta}(u+t\chi)$. Then $u_t\in \mathcal{E}^1(X,\theta,\phi)$, $t\mapsto \AMO(u_t)$ is differentiable, and its derivative is given by
\[
\frac{d}{dt} \AMO(u_t) = \int_X \chi \theta_{u_t}^n, \ t \in \mathbb R. 
\]
\end{prop}
\begin{proof}
For $t\geq 0$  the potential $u+t \inf_X \chi $ is a candidate in each envelope, hence $u+t \inf_X \chi\leq u_t$. Monotonicity of $I_\phi$ now implies that $u_t \in \mathcal E^1(X,\theta,\phi)$. A similar argument implies that $u_t \in \mathcal E^1(X,\theta,\phi)$ for $t \leq 0$.

Let $t \in \mathbb R$ and $s > 0$. As the singularity type of each $u_t$ is the same of that of $u$, we can apply Lemma \ref{lem: inequalities energy}  and conclude: 
\[
 \int_X (u_{t+s}-u_t) \theta_{u_{t+s}}^n \leq \AMO(u_{t+s})-\AMO(u_t) \leq   \int_X (u_{t+s}-u_t) \theta_{u_{t}}^n.
\]
It follows from Theorem \ref{thm: envelope contact} that $\theta_{u_t}^n$ is supported on $\{u_t=u+t\chi\}$. We thus have
\[
\int_X (u_{t+s}-u_t) \theta_{u_{t}}^n = \int_X (u_{t+s}-u-t\chi) \theta_{u_t}^n\leq \int_X s\chi \theta_{u_t}^n,
\]
since $u_{t+s}\leq u+(t+s)\chi$. Similarly we have 
\[
\int_X (u_{t+s}-u_t) \theta_{u_{t+s}}^n = \int_X (u+(t+s)\chi-u_t) \theta_{u_{t+s}}^n\geq \int_X s\chi \theta_{u_{t+s}}^n.
\]
Since $u_{t+s}$ converges uniformly to $u_t$  as $s\to 0$, by Theorem \ref{thm: lsc of MA measures} it follows that $\theta_{u_{t+s}}^n$ converges weakly to $\theta_{u_t}^n$. As $\chi$ is continuous, dividing by $s>0$ and letting $s\to 0^+$ we see that the right derivative of $\AMO(u_t)$ at $t$ is $\int_X \chi \theta_{u_t}^n$. The same argument applies for the left derivative.
\end{proof}

\paragraph{The case $\lambda>0$.} It suffices to treat the case $\lambda=1$ as the other cases can be done similarly. 
 
We introduce the following functional on $\mathcal{E}^1(X,\theta,\phi)$: 
\[
F(u) :=F_{\mu}(u):= \AMO(u) -L_{\mu}(u) , \ u\in \Ec^1(X,\theta,\phi),
\]
where $L_{\mu}(u) :=\int_X e^{u} d\mu$.

\begin{theorem}
\label{thm: maximizers are solutions}
Assume that $u\in \Ec^1(X,\theta,\phi)$ maximizes $F$ on $\mathcal{E}^1(X,\theta,\phi)$. Then $u$ solves the equation \eqref{eq: MA lambda}. 
\end{theorem}

\begin{proof}
Let $\chi$ be an arbitrary continuous function on $X$ and set $u_t:= P_{\theta}(u+t\chi)$. It follows from  Proposition \ref{prop: derivative of AMO} that $u_t\in \Ec^1(X,\theta,\phi)$ for all $t\in \mathbb{R}$, that the function 
$$g(t):=\AMO(u_t) -L_{\mu}(u+t\chi)$$ 
is differentiable on $\mathbb{R}$, and its derivative is given by $$g'(t)=\int_X \chi \theta_{u_t}^n-\int_X \chi e^{(u+t\chi)}d\mu.$$ Moreover, as $u_t\leq u+t\chi$, we have $$g(t) \leq \AMO(u_t) -L_{\mu}(u_t)= F_{\mu}(u_t)\leq \sup_{\Ec^1(X,\theta,\phi)} F_{ \mu}=F(u)=g(0).$$
This means that $g$ attains a maximum at $0$, hence  $g'(0)=0$. Since $\chi \in C^0(X)$ is arbitrary  it follows that $\theta_u^n=e^{\lambda u}\mu$. 
\end{proof}

Having computed the first order variation of the Monge--Amp\`ere energy, we establish the following  existence and uniqueness result. 

\begin{theorem}\label{thm: existence_MA_eq_exp}
Assume that $\mu$ is a positive non-pluripolar measure on $X$. Then there exists a unique $u\in \Ec^1(X,\theta,\phi)$ such that 
\begin{equation}\label{eq: MA_exp_version}
	\theta_{u}^n =e^{ u} \mu. 
\end{equation}
\end{theorem}
\begin{proof}
	We use the variational method. Let $\{u_j\}_j$ be a sequence in $\mathcal{E}^1(X,\theta,\phi)$ such that $\lim_{j}F(u_j)=\sup_{\mathcal{E}^1(X,\theta,\phi)}F>-\infty$. We claim that $\sup_X u_j$ is uniformly bounded from above. Indeed, assume that it were not the case. Then by relabeling the sequence we can assume that $\sup_X u_j$ increase to $\infty$. By the compactness property \cite[Proposition 2.7]{GZ05} it follows that the sequence $\psi_j:=u_j-\sup_X u_j$ converges in $L^1(X,\omega^n)$ to some $\psi\in \psh(X,\theta)$ such that $\sup_X \psi=0$.  In particular, $\int_X e^{\psi} d\mu >0$. It thus follows that 
\begin{equation}
	\label{eq: proof of Theorem 2.23 1}
	\int_X e^{u_j}d\mu = e^{\sup_X u_j}  \int_X e^{\psi_j} d\mu  \geq c \ e^{\sup_X u_j}
\end{equation}
for some $c>0$. Note also that $\psi_j \leq \phi$ since $\psi_j\in \Ec(X,\theta,\phi)$ and $\psi_j\leq 0$ and $\phi$ is the maximal function with these properties (see Theorem \ref{thm: ceiling coincide envelope non collapsing}). It then follows that 
\begin{equation}
	\label{eq: proof of Theorem 2.23 2}
\AMO(u_j) = \AMO(\psi_j) + \sup_X u_j \leq \sup_X u_j. 
\end{equation}
From \eqref{eq: proof of Theorem 2.23 1} and \eqref{eq: proof of Theorem 2.23 2} we arrive at 
\[
\lim_{j\to \infty} F(u_j) \leq  \lim_{j\to \infty} \left(\sup_X u_j - ce^{\sup_X u_j}\right) = -\infty, 
\]
which is a contradiction. Thus $\sup_X u_j$ is bounded from above as claimed. 
 Since $F(u_j)\leq\AMO(u_j)\leq \sup_X u_j$ it follows that $\AMO(u_j)$ and hence $\sup_X u_j$ is also bounded from below. It follows again from \cite[Proposition 2.7]{GZ05} that a subsequence of $u_j$ (still denoted by $u_j$) converges in $L^1(X,\omega^n)$ to some $u\in {\rm PSH}(X,\theta)$. Since ${\rm I}_{\phi}$ is upper semicontinuous we have
 $$\limsup_{j\to \infty} {\rm I}_\phi (u_j)\leq {\rm I}_\phi(u),$$
 hence $u\in \Ec^1(X,\theta,\phi)$. Moreover, by continuity of $\varphi \mapsto \int_Xe^{\varphi}d\mu$ we get that $F(u) \geq \sup_{\Ec^1(X,\theta,\phi)} F$.  Hence $u$ maximizes $F$ on $\Ec^1(X,\theta,\phi)$. Now  Theorem \ref{thm: maximizers are solutions} shows that $u$ solves the desired complex Monge--Amp\`ere equation. The next lemma address the uniqueness question.
\end{proof}
\begin{lemma}
	\label{lem: comp_sol_subsol}
	Assume that $u\in \Ec(X,\theta,\phi)$ is a solution of \eqref{eq: MA_exp_version} and $v\in \Ec(X,\theta,\phi)$ is a subsolution, i.e., $\theta_{v}^n \geq e^{ v}\mu.$
	Then $u\geq v$ on $X$. 
\end{lemma} 
\begin{proof}
	By the comparison principle for the class $\Ec(X,\theta,\phi)$ (Corollary \ref{cor: comparison principle}) we have 
	$$
	\int_{\{u<v\}} \theta_{v}^n \leq \int_{\{u<v\}} \theta_{u}^n.
	$$
	As $u$ is a solution and $v$ is a subsolution to \eqref{eq: MA_exp_version} we then have
	$$
	\int_{\{u<v\}} e^{v} d\mu \leq \int_{\{u<v\}} \theta_{v}^n \leq \int_{\{u<v\}} \theta_{u}^n = \int_{\{u<v\}} e^{ u} d\mu\leq \int_{\{u<v\}} e^{v} d\mu.
	$$
	It follows that all inequalities above are equalities, hence $\mu(\{u<v\})=0$. Since $\mu=e^{u} \theta_u^n$, it follows that $\theta_u^n (\{u<v\})=0$. By the domination principle (Theorem \ref{thm: domination principle}) we get that $u\geq v$ everywhere on $X$.
\end{proof}

\paragraph{The case $\lambda=0$.}

It immediately follows from Lemma \ref{lem: concentration max} that subsolutions are preserved under taking maximums. In addition to this, the $L^1$-limit of subsolutions is also a subsolution:
\begin{lemma}\label{lem: stability of subsolutions}
	Let $(u_j)$ be a sequence of $\theta$-psh functions such that $\theta_{u_j}^n \geq f_j \mu$, where $f_j \in L^1(X,\mu)$ and $\mu$ is a positive non-pluripolar Borel measure on $X$. Assume that $f_j$ converge in $L^1(X,\mu)$ to $f \in L^1(X,\mu)$,  and $u_j$ converge in $L^1(X,\omega^n)$ to $u\in \PSH(X,\theta)$. Then  $\theta_u^n \geq f \mu$. 
\end{lemma}
\begin{proof}
	By extracting a subsequence if necessary, we can assume that $f_j$ converge $\mu$-a.e. to $f$. For each $k$ we set $v_k:= \textup{usc}(\sup_{j\geq k} u_j)$. Then $v_k$ decreases pointwise to $u$ and Lemma \ref{lem: concentration max} together with Theorem \ref{thm: lsc of MA measures} gives 
	\[
	\theta_{v_k}^n  \geq \left(\inf_{j\geq k} f_j \right) \mu. 
	\]
Indeed $\max(u_k, \ldots, u_{k+l}) \nearrow v_k$ a.e., as $l \to \infty$, making Theorem \ref{thm: lsc of MA measures} applicable due to Remark \ref{rem: increasing implies capacity}.

To explain our notation below, for $t>0$ and a function $g$ we set $g^t:= \max(g,V_{\theta}-t)$. 

Note that $\{ u >V_\theta-t\}\subset \{ v_k >V_\theta-t\}$. Multiplying both sides of the above estimate with $\id_{\{u >V_{\theta}-t\}}$, $t>0$, and using the locality of the complex Monge--Amp\`ere operator with respect to the plurifine topology we arrive at 
	\[
	\theta_{v_k^t}^n \geq \id_{\{u >V_{\theta}-t\}} \left(\inf_{j\geq k} f_j\right) \mu.
	\]
	Note that for $t>0$ fixed, $v_k^t$ decreases to $u^t$, all having minimal singularity type (in particular they all have full mass); also the sequence $\left(\inf_{j\geq k} f_j\right)$ is increasing to $f$. Letting $k\to \infty$ and using Theorem \ref{thm: lsc of MA measures} we obtain 

	\[
	\theta_{u^t}^n \geq \id_{\{u >V_{\theta}-t\}} f \mu,\ t>0.
	\]
  	Again, multiplying both sides with $\id_{\{u >V_{\theta}-t\}}$, $t>0$, and using the locality of the complex Monge--Amp\`ere operator with respect to the plurifine topology we arrive at 
	\[
	\id_{\{u >V_{\theta}-t\}}\theta_{u}^n \geq \id_{\{u >V_{\theta}-t\}} f \mu.
	\]
	Finally, letting $t\to \infty$ we obtain the result. 
\end{proof}

\begin{theorem}
	\label{thm: existence lambda =0}
	Assume that $\mu$ is a non-pluripolar positive measure such that $\mu(X)=1$. Then there exists a unique $u\in \mathcal{E}(X,\theta,\phi)$ such that $\theta_{u}^n=\mu$ and $\sup_X u=0$. 
\end{theorem}

\begin{proof}
The uniqueness follows from Theorem \ref{thm: uniqueness}.  
For each $j>0$, using Theorem \ref{thm: existence_MA_eq_exp} we solve
\[
   \theta_{v_j}^n = e^{ j^{-1}v_j } \mu, \quad v_j \in \mathcal{E}(X,\theta,\phi). 
\]
 We set $u_j:= v_j -\sup_X v_j$, so that $\sup_X u_j=0$. Up to extracting a subsequence, we can assume that $u_j \to u\in {\rm PSH}(X,\theta)$ in $L^1$ and almost everywhere. Since $u_j\leq \phi$, we have $u\leq \phi$. Observe also that $j^{-1}u_j$ converge in capacity to $0$.  Indeed, for any fixed $\varepsilon>0$ we have
 \[
 {\rm Cap}_{\omega}(j^{-1}u_j<-\varepsilon) = {\rm Cap}_{\omega}(u_j<-\varepsilon j) \to 0
 \]
 as follows from \cite[Proposition 3.6]{GZ05}. 
 
Hence  the functions $e^{j^{-1}u_j}$ converge in capacity to $1$.   It thus follows from Theorem \ref{thm: lsc of MA measures} that 
\[
\lim_{j\to \infty}\int_X e^{j^{-1} u_j} d\mu  = \mu(X)=1. 
\]
Since $\theta_{u_j}^n = e^{ j^{-1}\sup_X v_j } e^{ j^{-1}u_j } \mu$, it follows from the above that $j^{-1}\sup_X v_j \to 0$ as $j\to \infty$, hence $e^{j^{-1}v_j}$ converges to $1$ in $L^1(\mu)$ and almost everywhere.  It thus follows from Lemma \ref{lem: stability of subsolutions} that $\theta_u^n \geq \mu$. 
Since $u\leq \phi$, Theorem \ref{thm: BEGZ_monotonicity_full} ensures that $\int_X \theta_u^n \leq  \int_X \theta_{\phi}^n= 1 =  \mu(X)$. Hence $\theta_u^n = \mu$ as desired. 
\end{proof}

\paragraph{Log concavity of non-pluripolar masses.}We give a proof for the following result from \cite{DDL4}, initially conjectured in \cite[Conjecture 1.23]{BEGZ10}:

\begin{theorem}\label{thm: log concavity s.u.l.}
	Let $T_1,...,T_n$ be positive $(1,1)$-currents  on a compact K\"ahler manifold $X$. Then 
	\[
	\int_X \langle T_1 \wedge ...\wedge T_n\rangle \geq \bigg(\int_X \langle T_1^n\rangle \bigg)^{\frac{1}{n}} ... \bigg(\int_X \langle T_n^n\rangle \bigg)^{\frac{1}{n}}.
	\] 
\end{theorem}

\begin{proof}We can assume that the classes of $T_j$ are big and their total masses are non-zero. Other\-wise the right-hand side of the inequality to be proved is zero.  Consider smooth closed real $(1,1)$-forms $\theta^j$,  and $u_j \in \psh(X,\theta^j)$ such that $T_j = \theta^j_{u_j}$.  

We can assume that $\int_X \omega^n=1$. For each $j=1,...,n$, Theorem \ref{thm: existence lambda =0} ensures that there exists a normalizing constant $c_j>0$ and $\varphi_j\in \mathcal{E}(X,\theta^j,P_{\theta^j}[u_j])$ such that $\big(\theta^j_{\varphi_j}\big)^n = c_j\omega^n$. 

Thus we have
$$c_j=\int_X \big(\theta^j_{\varphi_j}\big)^n=   \int_X \big(\theta^j_{P_{\theta^j}[u_j]}\big)^n=\int_X \big(\theta^j_{u_j}\big)^n= \int_X \langle T_j^n \rangle .$$ 
Remark \ref{rem: increasing implies capacity} and  Theorem \ref{thm: E_memb_char} then gives
$$
\int_X \theta^1_{\varphi_1} \wedge ... \wedge \theta^n_{\varphi_n}  = \int_X  \theta^1_{P_{\theta^1}[u_1]} \wedge ... \wedge \theta^n_{P_{\theta^n}[u_n]}= \int_X  \theta^1_{u_1} \wedge ... \wedge \theta^n_{u_n} = \int_X\langle T_1 \wedge \ldots \wedge T_n \rangle .
$$
An application of the mixed Monge--Amp\`ere inequalities (\cite[Proposition 1.11]{BEGZ10}) gives that $\theta^1_{\varphi_1} \wedge \ldots \wedge \theta^n_{\varphi_n} \geq c_1^{1/n} \ldots c_n^{1/n} \omega^n$. The result follows from integrating this estimate. 
\end{proof}

\section{Relative boundedness of solutions}

Recall that we are working with $\phi \in \psh(X,\theta)$ such that $P_{\theta}[\phi]=\phi$, and $\int_X \theta_\phi^n >0$. Let $f \in L^p(\omega^n)$ with $f \geq 0$. In the previous section we have shown that the equation
\begin{equation*}
\theta_u^n = f \omega^n, \ \ u \in \mathcal E(X,\theta,\phi)
\end{equation*}
has a unique solution. In this section we will show that this solution has the same singularity type as $\phi$. This generalizes \cite[Theorem B]{BEGZ10}, that treats the particular case of solutions with minimal singularity type in a big class. Analogous results will be obtained for equations of the type \eqref{eq: MA_exp_version} as well.

Our arguments follow the one in \cite{GL21a} which relies on quasi-psh envelopes. The original argument in \cite{DDL2,DDL4} was given using a Kolodziej type estimate, inspired from \cite{Kol98}.

In our study we make use of the following lemma multiple times:
\begin{lemma}\label{ineq: forms_MA}
Let $v \in \mathcal E(X,\theta,\phi)$, $v \leq \phi$, and $\gamma: \mathbb{R}^+ \cup \{\infty\} \mapsto \mathbb{R}^+\cup\{\infty\}$ denote a concave continuous increasing function with $\gamma'\leq 1$. Then $\chi:= -\gamma(\phi-v)+\phi \in \textup{PSH}(X,\theta,\phi)$ and $\chi$ satisfies
\begin{equation}\label{eq: measure_ineq}
\theta_\chi^n \geq (\gamma'(\phi-v))^n \theta_v^n.
\end{equation}
\end{lemma}

To be precise, in the above statement the function $\chi$ is equal to $-\infty$ on the pluripolar set $\{\phi = -\infty\}$ and  $\chi= -\gamma(\phi-v)+\phi$ otherwise.

\begin{proof} Notice that there exists $\gamma_k: \mathbb{R}^+ \cup \{\infty\} \mapsto \mathbb{R}^+\cup\{\infty\}$ smooth on $[0,\infty)$, such that $\gamma_k \nearrow \gamma$ pointwise and $\gamma'_k \leq 1$. As a result, it will be enough to prove the theorem for $\gamma$ smooth.

First we want to show that $\chi= -\gamma(\phi-v)+\phi \in \textup{PSH}(X,\theta,\phi)$. Observe that $\chi\leq \phi$ since $-\gamma\leq 0$. In order to prove that $\chi$ is $\theta$-psh, we will show that it is the decreasing limit of $\theta$-psh functions. 

Since the problem is local, we can work in a local chart $V\subset X$ where we write $\theta = dd^c g$ in $V$. After slightly shrinking $V$, the mollifications $v_j: =  \rho_{1/j} * (v + g) - g$ and $\phi_j: =  \rho_{1/j} * (\phi + g) - g$ are smooth approximants of $v$ and $\phi$ on $V$. Moreover, they are $\theta$-psh on $V$, $v_j \searrow v$, $\phi_j \searrow \phi$, and $v_j\leq \phi_j$. Let $\chi_{j} := -\gamma(\phi_j-v_j)+\phi_j$. Then for any $j$, $\chi_{j}$ is a smooth $\theta$-psh function in $V$ since
 \begin{eqnarray}\label{eq_psh}
\nonumber  \theta +dd^c\chi_j& = & \gamma'(\phi_j-v_j) \theta_{v_j} + (1-\gamma'(\phi_j-v_j)) \theta_{\phi_j} - \gamma''(\phi_j-v_j) d(\phi_j-v_j) \wedge d^c (\phi_j-v_j)\\
&\geq & \gamma'(\phi_j-v_j) \theta_{v_j},
 \end{eqnarray}
where  we used that $\gamma$ is concave and that $\gamma'\leq 1$. 

We claim that $\chi_j$ is decreasing in $j$. Indeed, for $s\in \mathbb R$ fixed, the function $t\mapsto -\gamma(t-s) +t$ is increasing in $[s,+\infty)$ because $1-\gamma'(t-s)\geq 0$. Using also that $\gamma$ is increasing we find that for $j\leq k$ we have
\[
\chi_j = -\gamma(\phi_j-v_j)+\phi_j\geq -\gamma(\phi_j-v_k)+\phi_j \geq -\gamma(\phi_k-v_k) + \phi_k=\chi_k. 
\]
It follows that $\chi_j$ decreases to some $\theta$-psh function in $V$ which has to be $\chi$. Indeed, the pointwise convergence $\chi_j \to \chi$ is seen to hold on the set $\{\phi>-\infty\}$. Since we have $\chi_j \leq \phi_j \to -\infty$ on $\{\phi>-\infty\}$, pointwise convergence holds everywhere on $X$. It follows that $\chi$ is $\theta$-psh on $X$.

Next we now show that
 \begin{flalign}\label{eq: to_argue}
 \theta +dd^c \chi \geq \gamma'(\phi-v) \theta_v.
 \end{flalign}
 For $t>0$, let $\phi^t := \max(\phi,V_\theta - t)$, $v^t := \max(v,V_\theta - t)$ and $\chi^t:= -\gamma({\phi^t-v_t})+\phi^t$. Note that all of these potentials are locally bounded on $\textup{Amp}(\{\theta\})$. As before, after slightly shrinking $V$, the mollifications $v^t_j$ and $\phi^t_j$ are smooth approximants of $v^t$ and $\phi^t$ on $V$. 
The same computations give that $\chi^t_{j}:= -\gamma({\phi^t_{j}-v^t_j})+\phi^t_{j}$ satisfies \eqref{eq_psh}.
 By Proposition \ref{prop: xing_conv}, letting $j\rightarrow\infty$ we obtain
\begin{eqnarray*}
\theta +dd^c \chi^t  \geq   \gamma'(\phi^t-v^t) (\theta+dd^c v^t) \quad {\rm in} \; V.
\end{eqnarray*}
Let $U_t = \{v > V_\theta - t\}=\{\phi > V_\theta - t\} \cap \{v > V_\theta - t\}$. Using plurifine locality, we get that
	 \begin{eqnarray*}
	 {\bf 1}_{U_t}(\theta +dd^c \chi)={\bf 1}_{U_t}(\theta +dd^c \chi_t) \geq {\bf 1}_{U_t}   \gamma'(\phi-v) (\theta+dd^c v) \quad {\rm in} \; V. 
\end{eqnarray*}
Letting $t \to \infty$, we arrive at \eqref{eq: to_argue}.

Unfortunately \eqref{eq: to_argue} does not directly imply \eqref{eq: measure_ineq}  since $ \gamma'(\phi-v) \theta_v$ is not a closed form. However the above approximation process allows to conclude on the chart $V$ considered above.\\
Since $\chi^t_{j}$ and $v^t_{j} $  are smooth, from \eqref{eq: to_argue} we get that 
\begin{eqnarray*}
(\theta +dd^c \chi^t_{j} )^n \geq   (\gamma'(\phi^t_{j}-v^t_{j}) )^n(\theta+dd^c v^t_j)^n  \quad {\rm in} \; V.
\end{eqnarray*}
We now conclude similarly. By Proposition \ref{prop: xing_conv}, letting $j\rightarrow\infty$ we obtain
\begin{eqnarray*}
(\theta +dd^c \chi^t )^n \geq   (\gamma'(\phi^t-v^t) )^n(\theta+dd^c v^t)^n  \quad {\rm in} \; V.
\end{eqnarray*}
Let $U_t = \{\phi > V_\theta - t\} \cap \{v > V_\theta - t\}$. Using plurifine locality, we get that
	 \begin{eqnarray*}
	 {\bf 1}_{U_t}(\theta +dd^c \chi)^n \geq {\bf 1}_{U_t}   (\gamma'(\phi-v) )^n(\theta+dd^c v)^n \quad {\rm in} \; V. 
\end{eqnarray*}
Letting $t \to \infty$, we arrive at the conclusion.
\end{proof}

\begin{theorem}\label{thm: min sing of solution}
Let $u \in \mathcal E(X,\theta,\phi)$ with $\sup_X u =0$. If $\theta_u^n = f \omega^n$ for some $f \in L^p(\omega^n), \ p > 1,$ then $u$ has the same singularity type as $\phi$. More precisely:
$$
\phi - C\Big(\|f\|_{L^p},p,\omega,\theta,\int_X \theta_\phi^n\Big)  \leq u \leq \phi.
$$
\end{theorem}

\begin{proof}
To simplify the notation, we set $\mu=f\omega^n$.  

\noindent  {\bf A priori estimate. }
We assume at the moment that $u-\phi$ is bounded, hence
 \[
 T_{\max}:= \sup \{t>0 \; : \; \mu(u <\phi-t) >0\}<\infty.  
 \]
Our goal is to establish a uniform bound on $T_{\max}$. 
 Since $f\in L^p(\omega^n)$ and ${\rm PSH}(X,\omega)\subset L^q(\omega^n)$ for any $q>0$, by H\"older inequality, ${\rm PSH}(X,\omega) \subset L^r(\mu)$, for any $r>0$ and  the quantity 
\[
A_r(\mu):= \sup \left\{\int_X (-h)^r d\mu \; : \; h \in {\rm PSH}(X,\omega), \; \sup_X h=0\right\}
\]
is finite and it depends on an upper bound for $\|f\|_p$.  

    By definition,  $u\geq \phi-T_{\max}$ almost everywhere with respect to $\mu=\theta_u^n$,  hence everywhere by the domination principle (Theorem \ref{thm: domination principle}), providing the desired a priori bound. In particular $|u-\phi|\leq T_{\max}$.
    
 We let $\chi:\mathbb{R}^+ \rightarrow \mathbb{R}^+$ denote a convex increasing function
 such that $\chi(0)=0$ and $\chi'(0) = 1$ (hence $\chi'(t)\geq 1$). Let $\gamma : \mathbb{R}^+ \rightarrow \mathbb{R}^+$ denote the inverse function of $\chi$, which is concave and increasing. 
 We set $\psi=-\chi (\phi-u)+ \phi$, $v=P_{\theta}(\psi)$,  and observe that 
$$\varphi:=-\gamma(\phi-v) +\phi\leq -\gamma(\phi-\psi) +\phi=  -\gamma(\chi(\phi-u)) +\phi\leq  u$$ with equality on the contact set $\{v=\psi\}$. Since $u$ has the same singularity type as $\phi$, so does $v$. In particular, by Proposition \ref{prop: comparison generalization} $$\int_X \theta_v^n=\int_X \theta_\phi^n= \int_X \theta_u^n.$$
Observe that $\gamma'\leq 1$ since $\chi'(\gamma(t))\gamma'(t)=1$ and that $\phi-u= \gamma (\phi-v)$  on $\{v=\psi\}$. Lemmas \ref{ineq: forms_MA} and \ref{lem: concentration max} thus give
 \[
 \id_{\{v=\psi\}} (\chi'(\phi-u))^{-n}\theta_v^n = \id_{\{v=\psi\}} (\gamma'(\phi-v))^n\theta_v^n \leq\id_{\{v=\psi\}} \theta_\varphi^n \leq \id_{\{v=\psi\}} \theta_u^n \leq \theta_u^n,
 \] 
 hence $\id_{\{v=\psi\}} \theta_v^n \leq (\chi'(\phi-u))^n \mu$. By Theorem \ref{thm: envelope contact}, we can infer $\theta_v^n \leq (\chi'(\phi-u))^n \mu$.
  
 \smallskip
 
\noindent  {\bf Step 1: Controlling the energy of $v$}.
 The convexity of $\chi$ and the normalization $\chi(0)=0$ yields $\chi(t) \leq t \chi'(t)$.
 Since $v=-\chi (\phi-u)+\phi$ on $\{v=\psi\}$, the above inequality, the convexity of $\chi$ and H\"older inequality (with $p= n+2$ and $q=(n+2)/(n+1)$) yields
 \begin{flalign*}
 \int_X (\phi-v) \theta_v^n
 &=   \int_X \chi (\phi-u)  \theta_v^n \leq  \int_X \chi (\phi-u)(\chi'(\phi-u))^{n} d\mu \\
 &\leq \int_X (\phi-u)(\chi'(\phi-u))^{n+1} d\mu\\
 &\leq\left(\int_X (\phi-u)^{n+2} d\mu \right)^{\frac{1}{n+2}} \left(\int_X (\chi'(\phi-u))^{n+2} d\mu \right)^{\frac{n+1}{n+2}}\\
 &\leq A_{n+2}(\mu)^{\frac{1}{n+2}}. \left(\int_X (\chi'(\phi-u))^{n+2} d\mu \right)^{\frac{n+1}{n+2}}. 
 \end{flalign*}

  \smallskip

\noindent  {\bf Step 2: Controlling the norms $||u||_{L^m}$}. To simplify the notation we set $m=n+3$. 
We are going to choose below the weight $\chi$ in such a way that
$\int_X ( \chi'  (\phi-u))^{n+2} d\mu \leq 2$.
This provides a uniform lower bound on $\sup_X v$ as we now explain. 
Indeed,
$$
0 \leq  (-\sup_X v) \mu(X) = -\sup_X (v-\phi)  \int_X \theta_v^n   
\leq \int_X (\phi-v)\theta_v^n \leq 2A_{n+2}(\mu)^{\frac{1}{n+2}} =C_1,
$$
where the first identity follows from Lemma \ref{lem: model sup over X}.
This yields $\sup_X v \geq -C_1 \mu(X)^{-1}$.
We infer that  $v$ belongs to a compact set of $\theta$-psh functions, hence its norm $||v||_{L^m(\mu)}$ is under control. Since $v\leq \psi$ it follows that $\phi-v \geq \chi(\phi-u) \geq 0$ and hence 
\begin{flalign*}
  \int_X(\chi (\phi-u))^m d\mu  &\leq \int_X (\phi-v)^m d\mu \\
  & \leq \int_X |v-\sup_X v +\sup_X v|^m d\mu \\
  &\leq 2^{m-1} \int_X(|v-\sup_X v|^m + |\sup_X v|^m) d\mu \\
  &\leq 2^{m-1} A_{m}(\mu) + 2^{m-1} |\sup_X v|^m \mu(X)\leq C_2. 
\end{flalign*} 
Chebyshev's inequality thus yields
\begin{equation} \label{eq:clef}
 {\mu }(u<\phi-t) \leq  \frac{1}{(\chi(t))^m}\int_X(\chi (\phi-u))^m d\mu\leq  \frac{C_2}{(\chi(t))^m}. 
 \end{equation}

   \smallskip
 
\noindent  {\bf Step 3: Choice of $\chi$}.
 If $g: \mathbb{R}^+ \rightarrow \mathbb{R}^+$  is strictly increasing with $g(0)=1$, Lebesgue's formula gives
 \begin{equation}\label{Leb_int}
\int_X g (\phi-u) d\mu = \mu(X) +\int_0^{T_{\max}}  g'(t) \mu(u<\phi-t) dt.
 \end{equation}
Fix $0<T_0<T_{\max}$. 
Setting $g(t)=(\chi' (t))^{n+2}$  we define $\chi$ by imposing $\chi(0)=0$, $\chi'(0)=1$, and 
$$
g'(t)=
\begin{cases}
	\dfrac{1}{(1+t)^2 {\mu }(u<\phi-t)}, \; \text{if} \; t\leq T_0\\
	\; 
\\	\frac{1}{(1+t)^2} \; \; \; \; \text{ if}\;  t >  T_0
\end{cases}.
$$
This choice guarantees that $\chi: \mathbb{R}^+ \rightarrow \mathbb{R}^+$ is convex increasing with $\chi' \geq 1$, and by \eqref{Leb_int} that
$$
\int_X ( \chi'  (\phi-u))^{n+2} d\mu \leq  \mu(X) + \int_0^{\infty} \frac{dt}{(1+t)^2} =2.
$$

\smallskip

 \noindent  {\bf Step 4: Conclusion}.  
Observe that $g(t) \geq g(0)=1$, hence $\chi'(t) =(g(t))^{\frac{1}{n+2}} \geq 1$. This yields
\begin{equation} \label{eq:minh(1)}
\chi(t)=\int_0^t \chi'(s) ds \geq t.
\end{equation}
In particular $\chi(t)\geq t$. Together with \eqref{eq:clef}, our choice of $\chi$ yields, for all $t\in [0,T_0]$, 
$$
\frac{1}{(1+t)^2g'(t)}={\mu }(u<\phi-t) \leq \frac{C_2}{(\chi(t))^m}.
$$
 This reads
$$
(\chi(t))^m \leq  C_2(1+t)^2 g'(t)=(n+2) C_2 (1+t)^2 \chi''(t)  (\chi'(t))^{n+1}, \quad \forall t\in [0,T_0]. 
$$
Multiplying  by $\chi'\geq 1$,  integrating between $0$ and $t$, 
we get that for all $t \in [0,T_0]$
\begin{eqnarray*}
\frac{(\chi(t))^{m+1}}{m+1}
&\leq&   (n+2)  C_2   \int_0^t (1+s)^2 \chi''(s)  (\chi'(s))^{n+2}\ ds\\
&\leq &   (n+2)  C_2   \int_0^t [(1+s)^2 \chi''(s)  (\chi'(s))^{n+2} + 2(1+s) ((\chi'(s))^{n+3}-1)]\ ds \\
&= &  \frac{(n+2) C_2 (1+t)^2 }{n+3} (1+s)^2  \left ((\chi'(s))^{n+3} - 1 \right) \big|_{s=0}^{s=t}\\
&\leq  & \frac{(n+2) C_2 (1+t)^2 }{n+3} \left ((\chi'(t))^{n+3} - 1 \right)  \\
&\leq&    C_3    (1+t)^2   (\chi'(t))^{n+3}.
\end{eqnarray*}
Recall that we choose $m=n+3$ so that 
$
\alpha:=m+1> \beta:= n+3>2.
$
The previous inequality then reads
$$
 (1+t)^{-\frac{2}{\beta}} \leq C_4  {\chi'(t)}{\chi(t)^{-\frac{\alpha}{\beta}}}. 
$$
Since $\alpha>\beta>2$ and $\chi(1) \geq 1$ (by \eqref{eq:minh(1)}), 
integrating the above inequality between $1$ and $T_0$ we obtain 

$$ \frac{\beta-2}{\beta} (1+t)^{\frac{\beta-2}{\beta}} \big|_1^{T_0} \leq C_4 \frac{\beta-\alpha}{\beta} \chi(t)^{\frac{-\alpha+\beta}{\beta}} \big|_1^{T_0} \leq C_4 \frac{\alpha-\beta}{\beta} \chi(1)^{\frac{-\alpha+\beta}{\beta}} \leq C_5.$$
It then follows from the above that 
$
T_0 \leq C_6,
$
for some uniform constant $C_6>0$.
Since $T_0$ was chosen arbitrarily in $(0,T_{\max})$ the uniform estimate for $T_{\max}$ follows. 

\paragraph{Relative boundedness of $u$.}
To finish the proof we finally prove that $u-\phi$ is bounded.  
We fix $1<q<p$ and a constant $\varepsilon>0$ so small that $e^{-\varepsilon h}f \in L^q(X)$ for all $h\in {\rm PSH}(X,\omega)$.
For each $j$ we solve 
\[
u_j \in \mathcal{E}(X,\theta,\phi), \; \theta_{u_j}^n = \id_{\{u>\phi-j\}}e^{\varepsilon(u_j-\max(u,\phi-j)} \mu. 
\]
Observe that for any fixed $j$, $\max(u,\phi-j)$ is a subsolution of the above equation because 
\[
\theta_{\max(u,\phi-j)}^n \geq {\bf 1}_{\{u>\phi-j\}} \theta_u^n={\bf 1}_{\{u>\phi-j\}} e^{\varepsilon(\max(u,\phi-j)-\max(u,\phi-j))} \mu.
\]
Lemma \ref{lem: comp_sol_subsol} then gives $u_j\geq \max(u,\phi-j)$, hence $\sup_X u_j\geq 0$ and  $u_j-\phi$ is bounded. 
If $j<k$, then 
\[
\theta_{u_k}^n = {\bf 1}_{\{u>\phi-k\}} \id_{\{u>\phi-k\}}e^{\varepsilon(u_k-\max(u,\phi-k)} \mu \geq \id_{\{u>\phi-j\}}e^{\varepsilon(u_k-\max(u,\phi-j)} \mu.
\]
Invoking again Lemma \ref{lem: comp_sol_subsol} we obtain $u_j\geq u_k$.  The measures 
\[
\mu_j:= \id_{\{u>\phi-j\}}e^{\varepsilon(u_j-\max(u,\phi-j)} \mu = f_j \omega^n
\]
have densities $f_j\leq e^{\varepsilon u_1} e^{-\varepsilon \max(u,\phi-j)}f\leq e^{\varepsilon \sup_X u_1}e^{-\varepsilon \max(u,\phi-j)}f$ which are uniformly bounded in $L^q(X)$. By the above a priori estimate,  we have a uniform bound $u_j \geq \phi-C$, hence $v:=\lim_{j\to \infty}u_j$ satisfies $v\geq \phi-C$ and $\theta_{v}^n =e^{\varepsilon(v-u)}\mu$. Since $\theta_{v}^n =e^{\varepsilon(u-u)}\mu $, Lemma \ref{lem: comp_sol_subsol} ensures $u=v$. Thus $u-\phi$ is bounded as desired. This finishes the proof.  
\end{proof}

\begin{coro}
	 If $\lambda>0$ and $u \in \mathcal{E}(X,\theta,\phi)$, $\theta_u^n = e^{\lambda u}f \omega^n$ for some $f \in L^p(\omega^n), \ p > 1,$ then $u$ has the same singularity type as $\phi$.
\end{coro} 

\begin{proof} 
Since $u$ is bounded from above on $X$ and $\lambda>0$ it follows that $e^{\lambda u} f\in L^p(X,\omega^n)$, $p>1$. The result follows from Theorem \ref{thm: min sing of solution}. 
\end{proof}

\section{Naturality of model type singularities and examples}
\label{sec: Examples}

One may still wonder if our choice of model potentials is a natural one in our pursuit of  complex Monge--Amp\`ere equations with prescribed singularity. We address these doubts in the next well posedness result.

\begin{theorem}\label{thm: naturality_of_model} Suppose that $\psi \in \textup{PSH}(X,\theta)$ and the equation 
$$\theta_u^n = f \omega^n$$
has a solution $u \in \textup{PSH}(X,\theta)$ with the same singularity type as $\psi$, for all $f \in L^{\infty}(X), \ f \geq 0$ satisfying $\int_X \theta_\psi^n = \int_X f \omega^n > 0$. Then $\psi$ has model type singularity. 
\end{theorem}

\begin{proof} Suppose that $[\psi]$ is not of model type. Then $P_{\theta}[\psi]$ is strictly less singular than $\psi$. On the other hand $\mathcal E(X,\theta,\psi) \subset \mathcal E(X,\theta,P_{\theta}[\psi])$ as $\int_X \theta_\psi^n = \int_X \theta_{P_{\theta}[\psi]}^n$. 

By Theorem \ref{thm: MA of env sing type}, there exists $g \in L^\infty$ such that $\theta^n_{P_{\theta}[\psi]} = g \omega^n.$ 
By uniqueness (Theorem \ref{thm: uniqueness}), $P_{\theta}[\psi]$ is the only solution of this last equation inside $\mathcal E(X,\theta,P_{\theta}[\psi])$. 

Since $\mathcal E(X,\theta,\psi) \subset \mathcal E(X,\theta,P_{\theta}[\psi])$, but $P_{\theta}[\psi] \notin \mathcal E(X,\theta,\psi)$, we get that $\theta_u^n = g \omega^n$ cannot have any solution that has the same singularity type as $\psi$.
\end{proof}

Next we point out a simple way to construct potential with model singularity types:

\begin{prop}\label{prop: Lp example} Suppose that $\psi \in \textup{PSH}(X,\theta)$  and $\theta_\psi^n = f \omega^n$ for some $f \in L^p(\omega^n), \ p > 1$ with $\int_X f \omega^n >0$. Then $\psi$ has model type singularity.
\end{prop}
\begin{proof}
We first observe that $\psi \in \mathcal{E}(X,\theta,P_{\theta}[\psi])$. Since $\theta_{\psi}^n$ has $L^p$ density with $p>1$, it thus follows from Theorem  \ref{thm: min sing of solution} that  $\psi-P_{\theta}[\psi]$ is bounded on $X$, hence $[\psi]=[P_\theta[\psi]]$, implying that $\psi$ has model singularity type.
\end{proof}

As noticed in \cite{RWN14} (see \cite{RS05} for the local case), all analytic singularity types are of model type.

\begin{prop}\label{prop: analytic example} Suppose $\psi \in \textup{PSH}(X,\theta)$ has analytic singularity type, i.e. $\psi$ can be locally written as $\frac{c}{2} \log \big(\sum_j |f_j|^2\big) + g$, where $f_j$ are holomorphic, $c > 0$ and $g$ is bounded. Then $[\psi]$ is of model type.
\end{prop}

We give an argument that is different from the elementary one given in \cite[Section 4.5]{DDL2}. Following \cite{RWN14}, we use resolution of singularities, and get a slightly more general result (with $g$ being bounded instead of smooth). 

\begin{proof} Let $\mathcal I $ be the coherent sheaf of holomorphic functions $f$ satisfying $|f| \leq A e^{\frac{2\psi}{c}}$ for some $A>0$ and $\pi: Y \to X$ be a resolution of singularities of $\mathcal I$, as in \cite[Remark 5.9]{Dem12}. We obtain that the singularity type  of $\psi \circ \pi$ is modelled by an snc divisor $D = \sum_{j} \alpha_j D_j \subset Y$, where $\alpha_j > 0$. That is to say, the Lelong numbers of $\psi \circ \pi$ along each $D_j$ is $\alpha_j$, i.e., $\psi \circ \pi$ has $\alpha_j \log|z|$ type poles along each $D_j$.

By Lemma \ref{lem: Lelong} we know that the Lelong numbers of $P_{\pi^*\theta}[\psi \circ \pi]$ and $\psi \circ \pi$ are the same. In particular,  we obtain that $[P_{\pi^*\theta}[\psi \circ \pi]] =[\psi \circ \pi]$. Since any $\pi^*\theta$-psh function can be written as $u\circ \pi$, for some $\theta$-psh function $u$, it follows that for any $C>0$, $P_{\pi^*\theta}(\psi \circ \pi +C, 0) = P_{\theta}(\psi +C, 0) \circ \pi$. This means that $P_{\pi^*\theta}[\psi \circ \pi] =P_{\theta}[\psi] \circ \pi $.
 Combining the above and pushing forward to $X$ we obtain that $[P_{\theta}[\psi]] =[\psi]$, as desired.
\end{proof}

 \chapter{The finite energy range of the complex Monge--Amp\`ere operator}

In this chapter we give a characterization of the Borel measures $\mu$ that are equal to the complex Monge--Ampere measure of some $u \in \mathcal E_\chi(X,\theta,\phi)$, with $\chi$ having polynomial growth, and $\phi$ a model potential ($\phi = P_\theta[\phi]$) with $\int_X \theta_\phi^n >0$.

Before we can achieve this we need to develop more potential theory. A weight is a continuous increasing function  $\chi: [0,\infty) \rightarrow [0,\infty)$ such that $\chi(0)=0$ and $\chi(\infty)=\infty$. Denote by $\chi^{-1}$ its inverse function, i.e. such that $\chi(\chi^{-1}(t))= t$ for all $t\geq 0$.
  
 Throughout  this section we assume that the weight $\chi$ satisfies the following condition 
 \begin{equation}
 	\label{eq: growth condition}
 	\forall t \geq 0, \; \forall \lambda \geq 1, \; \chi(\lambda t) \leq \lambda^M \chi(t),
 \end{equation}
 where $M\geq 1$ is a fixed constant. 
 Observe that from \eqref{eq: growth condition} it follows that 
\begin{equation}
 	\label{eq: growth condition1}
\forall t \geq 0, \; \forall \gamma < 1, \; \chi(\gamma t) \geq \gamma^M \chi(t). 
 \end{equation}
We fix $\phi$ a model potential and we let $\mathcal{E}_{\chi}(X,\theta,\phi)$ denote the set of all $u\in \mathcal{E}(X,\theta,\phi)$ such that 
\[
E_{\chi}(u,\phi):=\int_X \chi(|u-\phi|) \theta_u^n <\infty. 
\]
When $\phi=V_{\theta}$, we denote $\mathcal{E}(X,\theta)=\mathcal{E}(X,\theta,V_{\theta})$, $\mathcal{E}_{\chi}(X,\theta)=\mathcal{E}_{\chi}(X,\theta,V_{\theta}$) and $E_{\chi}(u)=E_{\chi}(u,V_{\theta})$. Compared to \cite{GZ07},  we have changed the sign of the weight, but the weighted classes are the same.  The following simple observation shows that the class $\mathcal{E}(X,\theta,\phi)$ is stable under adding a constant. 

\begin{lemma}
    \label{lem: E chi stable adding constant}
    If $u$ belongs to $\mathcal{E}_\chi(X,\theta,\phi)$ then so does $u+C$ for any constant $C$. 
\end{lemma}
\begin{proof}
Since $\chi$ satisfies \eqref{eq: growth condition}, for $t>0$ and $s>0$ we have 
\[
\chi(t+s) = \chi(2(t+s)/2) \leq 2^M \chi((t+s)/2) \leq 2^M \chi(\max(t,s))\leq 2^M  \max(\chi(t), \chi(s) ),
\]
where the last two inequalities follow from the fact that $\chi $ is increasing.
Then 
   \begin{flalign*}
       \int_X \chi(|u+C-\phi|)\theta_u^n &\leq \int_X \chi(|u- \phi|+|C|)\theta_u^n\\
       &\leq 2^M\int_X \max(\chi(|u-\phi|),\chi(|C|))\theta_u^n \\
       &\leq 2^M \max \left(\int_X \chi(|u-\phi|)\theta_u^n, \ \chi(|C|) \int_X \theta_u^n \right)<\infty. 
   \end{flalign*}
\end{proof}
Next we prove a few of technical results that will be often used.

\begin{lemma}
	\label{lem: int_bound}
 There exists a uniform constant $C>0$ such that for all $u\in \psh(X,\theta,\phi)$ normalized with $\sup_X u=0$ we have
 $$\int_X \chi(\phi-u)\theta_\phi^n \leq C.$$
 \end{lemma}
\begin{proof}
     We recall that by Theorem \ref{thm: MA of env sing type}, $\theta_{\phi}^n\leq {\bf 1}_{\{ \phi=0\}} \theta^n \leq A \omega^n$, for some $A>0$.
     Also, thanks to \eqref{eq: growth condition}, if $\phi-u\geq 1$ we have
     $\chi(\phi-u)\leq (\phi-u)^M \chi(1) $; otherwise (since $\chi$ is increasing) we have   $\chi(\phi-u)\leq \chi(1) $. In both cases we can infer that
     $$\int_X \chi(\phi-u)\theta_\phi^n\leq C' \int_X (|\phi|^M+|u|^M + 1)\omega^n.$$
     The conclusion then follows from the fact that $\int_X |h|^M \omega^n$ is uniformly bounded for  $h \in {\rm PSH}(X,\theta)$ with $\sup_X h =0$. 
\end{proof}
\begin{lemma}
	\label{lem: bound_mixed}
 Let $u\in \mathcal{E}(X,\theta,\phi)$ with $\sup_X u=0$. Then, for any $j\in \{1,...,n\}$
 $$\int_X \chi(\phi-u) \theta_u^j\wedge \theta_{\phi}^{n-j} \leq \int_X \chi(\phi-u) \theta_u^n.$$
  \end{lemma}
\begin{proof}
By Lemma \ref{lem: model sup over X} we have that $u \leq \phi \leq 0$.  By the partial comparison principle (Proposition \ref{prop: general CP}) we have that for any $j\in \{1,...,n\}$
	\begin{flalign*}
		\int_X \chi(\phi-u) \theta_u^j\wedge \theta_{\phi}^{n-j} &= \int_0^{\infty} \theta_u^{j} \wedge \theta_{\phi}^{n-j} (u< \phi-\chi^{-1}(t))dt \\
		&\leq  \int_0^{\infty} \theta_u^{n}(u<\phi-\chi^{-1}(t)) dt 
		= \int_X \chi(\phi-u) \theta_u^n. 
	\end{flalign*}
\end{proof}

Lemmas \ref{lem: mixed energy 1},  \ref{lem: mixed energy 2} below are essentially known by \cite{GZ07}, but we repeat the simple proof for the reader's convenience. Similar simplifications can also be found in \cite{Gup22}.

\begin{lemma}
	\label{lem: mixed energy 1}
	If $u,v\in \mathcal{E}(X,\theta,\phi)$ and $u,v\leq 0$, then 
	\[
	\int_X \chi(\phi-u) \theta_v^n \leq 2^{n+M} E_{\chi}(u,\phi)  +  E_{\chi}(v,\phi).   
	\]
\end{lemma}
\begin{proof}
	By Lemma \ref{lem: model sup over X} we have that $u \leq \phi \leq 0$. By the comparison principle, Corollary \ref{cor: comparison principle},   we have 
	\begin{flalign*}
		\int_X \chi(\phi-u) \theta_v^n &= \int_0^{\infty} \theta_v^n(u <\phi-\chi^{-1}(t))dt\\
		& \leq \int_0^{\infty} \theta_v^n(2u<v+\phi- \chi^{-1}(t))dt + \int_0^{\infty} \theta_v^n(v< \phi-\chi^{-1}(t)) dt\\
			& \leq 2^n \int_0^{\infty} \left( \theta+dd^c {\frac{v+\phi}{2}}\right)^n\left(u<\frac{v+\phi- \chi^{-1}(t)}{2}\right)dt +E_{\chi}(v,\phi)\\
		& \leq 2^n \int_0^{\infty} \theta_u^n(2u<2\phi-\chi^{-1}(t))dt + E_{\chi}(v,\phi)\\
		&= 2^n \int_X \chi(2\phi-2u)\theta_u^n + E_{\chi}(v,\phi)\\
		&\leq 2^{n+M} E_{\chi}(u,\phi) + E_{\chi}(v,\phi),
	\end{flalign*}
	where in the second line we have used the trivial inclusion 
	\[
	\{2u\geq v+\phi- \chi^{-1}(t) \} \cap \{v\geq \phi- \chi^{-1}(t) \} \subseteq \{ u\geq \phi-\chi^{-1}(t)\}, 
	\] 
	in the third one we have used
	\[
	\theta_v^n \leq 2^n \left(\theta+ dd^c \frac{v+\phi}{2}\right)^n,
	\]
	and in the last line  we have used \eqref{eq: growth condition}  with $\lambda=2$.
\end{proof}

\begin{lemma}
\label{lem: mixed energy 2}
	For all $u,v \in \mathcal{E}(X,\theta,\phi)$   with $u\leq v\leq 0$ we have 
	\[
	\int_X \chi(\phi-v)\, \theta_v^n\leq  \int_X \chi(\phi-u)\, \theta_v^n \leq 2^{n+M} E_{\chi}(u,\phi). 
	\]
\end{lemma}

\begin{proof}
	By the comparison principle, Corollary \ref{cor: comparison principle},  and the fact that $u\leq v\leq \phi$, we have 
	\begin{flalign*}
		\int_X \chi(\phi-u)\, \theta_v^n &= \int_0^{\infty} \theta_v^n(u <\phi-\chi^{-1}(t)) dt \\
		& \leq  \int_0^{\infty} \theta_v^n (2u<v+\phi -\chi^{-1}(t))  dt\\
		&\leq 2^n  \int_0^{\infty} \theta_\frac{v+\phi}{2} ^n \left(u< \frac{v +\phi-\chi^{-1}(t)}{2}\right) dt\\
		&\leq 2^n\int_0^{\infty} \theta_u^n (2u <2\phi-\chi^{-1}(t))  dt\\
		& = 2^n \int_X \chi(2\phi-2u) \theta_u^n \leq 2^{n+M} E_{\chi}(u,\phi).   
	\end{flalign*}
\end{proof}

\begin{prop}\label{prop: convexity of E chi}
	If $u \in \mathcal{E}_{\chi}(X,\theta,\phi)$ and $u\leq v$ then $v \in \mathcal{E}_{\chi}(X,\theta,\phi)$. Moreover, the class $\mathcal{E}_{\chi}(X,\theta,\phi)$ is convex.
\end{prop}

\begin{proof}
In view of Lemma \ref{lem: E chi stable adding constant} we can assume that $v\leq 0$. Lemma \ref{lem: mixed energy 2} then yields
$$ 	E_{\chi}(v,\phi)=  \int_X \chi(\phi-v)\, \theta_v^n \leq 2^{n+M} E_{\chi}(u,\phi).$$
We next prove that the class $\mathcal{E}_{\chi}(X,\theta,\phi)$ is convex.
Assume $u$ and $v$ are in $\mathcal{E}_{\chi}(X,\theta,\phi)$. It follows from Corollary \ref{cor: P(u,v) in E} that $w:= P_{\theta}(u,v)$ belongs to $\mathcal{E}(X,\theta,\phi)$.  From Theorem \ref{thm: MA of env sing type} we also have
	\[
	\int_X \chi(\phi-w)\theta_w^n \leq \int_X \chi(\phi-u)\theta_u^n + \int_X \chi(\phi-v)\theta_v^n  <\infty,
	\]
	hence $w\in \mathcal{E}_{\chi}(X,\theta,\phi)$. For $t\in [0,1]$, since $w \leq tu+(1-t)v$, it follows from Lemma \ref{lem: mixed energy 2} that $tu+(1-t)v\in \mathcal{E}_{\chi}(X,\theta,\phi)$. 
\end{proof}

\begin{lemma}
\label{lem: stable under L1 conv}
	Assume $(u_j)$ is a sequence in $\mathcal{E}_{\chi}(X,\theta,\phi)$  converging in $L^1$ to $u\in {\rm PSH}(X,\theta,\phi)$. If 
	$
	\sup_{j} E_{\chi}(u_j,\phi) <\infty
	$, then $u\in \mathcal{E}_{\chi}(X,\theta,\phi)$.  
\end{lemma}

\begin{proof}
We can assume  $u_j\leq 0$ for all $j$. 
	Define $v_k:= (\sup_{j\geq k} u_j)^* \in {\rm PSH}(X,\theta,\phi)$. Then $v_k \searrow u$ and $u$ is more singular than $\phi$. By Theorem \ref{thm: BEGZ_monotonicity_full} we then have $ \int_X \theta_{u}^n \leq \int_X \theta_\phi^n$.\\
 Since $0\geq v_k\geq u_k$, we have $\chi(\phi-v_k)\leq \chi(\phi-u_k)$. Lemma \ref{lem: mixed energy 2} then ensures that $E_{\chi}(v_k,\phi) \leq 2^{n+M} E_{\chi}(u_k,\phi)$ and the latter quantity is uniformly bounded by assumption. 
	
	Fix $t>0$ and consider $v_{k,t}:= \max(v_k,\phi-t)$.  Note that $v_{k,t} \searrow \max(u,\phi-t)$ as $k\to \infty$ and that by Lemma \ref{lem: mixed energy 2} we have $E_{\chi}(v_{k,t},\phi)\leq 2^{n+M} E_{\chi}(v_k,\phi)$. This means that $E_{\chi}(v_{k,t},\phi)$ has a uniform upper bound. Moreover, 
	the functions $\chi(\phi-v_{k,t})$ are quasi-continuous and uniformly bounded on $X$. It thus follows from Theorem \ref{thm: lsc of MA measures} that  
	\[
	\liminf_{k\to \infty}\int_X \chi(\phi-v_{k,t})(\theta +dd^c v_{k,t})^n \geq \int_X \chi(\phi-\max(u,\phi-t))(\theta +dd^c \max(u,\phi
	-t))^n.
	\]
	The above estimate then gives a uniform upper bound for $E_{\chi}(\max(u,\phi-t), \phi)$. We can thus infer that there exists a constant $C>0$ such that
 \[
 C\geq \int_X \chi(\phi-\max(u,\phi-t))\theta_{\max(u,\phi-t)}^n \geq \chi(t) \int_{\{u\leq \phi-t\}} \theta_{\max(u,\phi-t)}^n.
 \]
 Letting $t\to \infty$ yields that $\int_{\{u\leq \phi-t\}} \theta_{\max(u,\phi-t)}^n\rightarrow 0$ (since $\chi(t)\rightarrow \infty$). Hence, using the plurifine property of the Monge--Amp\`ere operator and the fact that $u$ is more singular than $\phi$, we get
 $$\int_X \theta_\phi^n = \lim_{t \to \infty} \int_{X} \theta_{\max(u,\phi-t)}^n= \lim_{t \to \infty}\bigg(\int_{\{u\leq \phi-t\}} \theta_{\max(u,\phi-t)}^n+ \int_{\{u> \phi-t\}} \theta_{u}^n\bigg) \leq \int_X \theta_{u}^n \leq \int_X \theta_\phi^n.$$
This means that $u\in \mathcal{E}(X,\theta,\phi)$. Using plurifine locality, we also obtain 
    \[
    \int_{\{u>\phi-t\}} \chi(\phi-u) \theta_u^n\leq C,
    \]
    and letting $t\to \infty$, we see that $\int_X \chi(\phi-u) \theta_u^n \leq C$. Hence $u\in \mathcal{E}_{\chi}(X,\theta,\phi)$. 
\end{proof}

\begin{lemma}
	\label{lem: integrability implies uniform}
Let $\mu$ be a positive Borel measure on $X$. 
Assume that $\mu(\{\phi=-\infty\})=0$ and $\chi(\phi-u)  \in L^1(\mu)$ for all $u\in  \mathcal{E}_{\chi}(X,\theta,\phi)$ and fix a constant $A>0$. Then there exists a constant $C=C(A)>0$ such that, for all $u \in \mathcal{E}_{\chi}(X,\theta,\phi)$ with $\sup_X u =0$ and $E_{\chi}(u,\phi) \leq A$ we have
	\[
	\int_X \chi(\phi-u) d\mu \leq C. 
	\]
\end{lemma}

 Observe that the condition $\mu(\{\phi = -\infty\})=0$ is necessary. Without it,  $\mu$-integrability of $\chi(\phi-u)$ can not be discussed, as the values of this function are not defined on the set $\{\phi = -\infty\}$.
 
\begin{proof}
	Assume by contradiction that there exists a sequence $(u_j) \subset \mathcal{E}_{\chi}(X,\theta,\phi)$ satisfying $\sup_X u_j =0$ and $E_{\chi}(u_j,\phi) \leq A$ such that
	\[
	\int_X \chi(\phi-u_j) \,d\mu \geq 4^{jM},
	\]
 where $M$ is the exponent appearing in \eqref{eq: growth condition}.\\
Define $v_k: = P_{\theta}(\min_{1\leq j\leq k}( 2^{-j} u_j + (1-2^{-j})\phi))\leq \phi$. Observe that $v_k$ is decreasing and setting
	 \[
	 v:= \lim_k v_k,\quad \psi:= \sum_{j\geq 1} 2^{-j}u_j,
	 \] 
	 then, since $u_j\leq \phi$, we can check that  $\psi \leq 2^{-j} u_j + (1-2^{-j}) \phi\leq \phi$. It then follows that $v \geq \psi \in {\rm PSH}(X,\theta,\phi)$.  
	 We fix $k\geq 2$ and for $1\leq l\leq k$ we set
  \[
	D_1 := \left\{v_k= 2^{-1}u_1 +2^{-1} \phi\right\} 
	\]
 and
	\[
	D_l := \left\{v_k= 2^{-l}u_l +(1-2^{-l}) \phi< \min_{1\leq p\leq l-1}(2^{-p}u_p + (1-2^{-p}) \phi) \right\} \quad {\rm for} \; \ell\geq 2 
	\]
	Note that the sets $D_l$ are pairwise disjoint and \[
	\bigcup_{ k \geq l\geq 1} D_l= \{v_k=\min_{1\leq j\leq k}( 2^{-j} u_j + (1-2^{-j})\phi)\}.
	\]
It follows from Theorem \ref{thm: envelope contact} and Lemma \ref{lem: concentration max} that 
	\[
		\int_X \chi(\phi-v_k) \theta_{v_k}^n \leq \sum_{l=1}^k \int_{D_l} \chi(2^{-l}(\phi-u_l)) (2^{-l}\theta_{u_l} + (1-2^{-l})\theta_{\phi})^n. 
	\]
	We note that 
	\begin{equation}\label{eq: MA_convex_comb}
	(2^{-l}\theta_{u_l}+ (1-2^{-l})\theta_{\phi})^n \leq \theta_{\phi}^n + 2^{-l} \sum_{p=1}^n \binom{n}{p} \theta_{u_l}^p \wedge \theta_\phi^{n-p}.  
	\end{equation}
	Since the sets $D_l$ are pairwise disjoint and $\chi$ is increasing, it then follows that 	
	\begin{eqnarray*}
		\int_X \chi(\phi-v_k) \theta_{v_k}^n  &\leq & \sum_{l=1}^k \int_{D_l} \chi(2^{-l}(\phi-u_l)) (2^{-l}\theta_{u_l} + (1-2^{-l})\theta_{\phi})^n\\
  &\leq & \sum_{l=1}^k \int_{D_l} \chi(2^{-l}(\phi-u_l)) \theta_{\phi}^n + 2^n \sum_{l=1}^k 2^{-l} \int_X \chi(\phi-u_l)  \theta_{u_l}^n\\
		&\leq & \int_{X} \chi(\phi-\psi) \theta_{\phi}^n + 2^n \sum_{l=1}^k 2^{-l} \int_X \chi(\phi-u_l)  \theta_{u_l}^n\\
		&\leq &  C_2 + 2^n A.
	\end{eqnarray*}
 The second inequality follows from \eqref{eq: MA_convex_comb} together with Lemma \ref{lem: bound_mixed}. The third inequality follows from the fact $0\leq 2^{-l}(\phi-u_l) \leq \sum_{l\geq 1} 2^{-l}(\phi-u_l)= \phi-\psi$.
 The last line follows from Lemma \ref{lem: int_bound}.\\
It then follows from Lemma \ref{lem: stable under L1 conv} that $v$, which is the decreasing limit of $v_k$, belongs to $\mathcal{E}_{\chi}(X,\theta,\phi)$. On the other hand, since $v\leq 2^{-k} u_k+ (1-2^{-k})\phi$ for any $k\geq 1$, and thanks to \eqref{eq: growth condition1} we have
	\[
	\int_X \chi(\phi-v)d\mu \geq \int_X \chi(2^{-k} (\phi-u_k)) d\mu \geq 2^{-kM} \int_X \chi(\phi-u_k) d\mu \geq 2^{kM} \to \infty,
	\]
	as $k\to \infty$, contradicting the assumption. 
\end{proof}

\begin{lemma}
	\label{lem: linear growth}
	Assume $\mu$ is a positive Borel measure satisfying  $\mu(\{\phi=-\infty\})=0$ and $\chi(\phi-u)  \in L^1(\mu)$ for all $u\in  \mathcal{E}_{\chi}(X,\theta,\phi)$.
	Then there exists a constant $C>0$ such that, for all $u \in \mathcal{E}_{\chi}(X,\theta,\phi)$ with $\sup_X u =0$,  we have
	\[
	\int_X \chi(\phi-u) d\mu \leq C( E_{\chi}(u,\phi)+1). 
	\]
\end{lemma}

\begin{proof}
	 We prove the lemma by contradiction, assuming there exists a sequence $(u_j) \subset \mathcal{E}_{\chi}(X,\theta,\phi)$ with $\sup_X u_j =0$, but 
	\[
	\int_X \chi(\phi-u_j) d\mu \geq j (E_{\chi}(u_j,\phi)+1). 
	\]
Set $\varepsilon_j:=  (E_{\chi}(u_j,\phi)+1)^{-1}<1$ and 
\[
\psi_j := P_{\theta}(-\chi^{-1}(\varepsilon_j \chi(\phi-u_j)) + \phi).
\]
Since $\theta_{\psi_j}^n$ is supported on $\{\psi_j= -\chi^{-1}(\varepsilon_j \chi(\phi-u_j)) + \phi\}$ (Theorem \ref{thm: envelope contact}), we have 
\[
\int_X \chi(\phi-\psi_j) \theta_{\psi_j}^n \leq \varepsilon_j \int_X \chi(\phi-u_j)  \theta_{\psi_j}^n \leq C_1. 
\]
In the last inequality we have used the fact that $u_j \leq \psi_j$ and that $\chi$ is increasing and Lemma \ref{lem: mixed energy 2}. By Lemma \ref{lem: integrability implies uniform} we thus have 
\[
\int_X \chi(\phi-\psi_j) d\mu \leq C_2. 
\]
But since $\phi-\psi_j \geq \chi^{-1}(\varepsilon_j \chi(\phi-u_j))$, we have
\[
\int_X \chi(\phi-\psi_j) d\mu \geq  \int_X \varepsilon_j\chi(\phi-u_j) d\mu \geq j,
\]
which is a contradiction. 
\end{proof}

\begin{lemma}
	\label{lem: DV21 relative} 
	There exists a constant $C>0$ such that, for all $u, v \in \mathcal{E}_{\chi}(X,\theta,\phi)$ with $\sup_X v=0$ and $u\leq 0$, we have 
\begin{equation}\label{eq: DV21 relative}
	\int_X \chi(\phi-v) \theta_u^n \leq C(1 +E_{\chi}(u,\phi)) E_{\chi}(v,\phi)^{M/(M+1)} +C.  
\end{equation}
\end{lemma}
The above result is due to Duc-Thai Do and Duc-Viet Vu \cite{DV21}. Their proof uses generalized non-pluripolar products. We give below a less technical argument that relies only on the comparison principle. 
\begin{proof}
	
 We observe that for all $t>1$, $v+ (t-1)\phi, tu \in \mathcal{E}_{\chi}(X,t\theta,t\phi)$ and $t\phi-(v+ (t-1)\phi)=\phi-v$ (where equality holds outside a puripolar set). Fixing $t>1$, by Lemma \ref{lem: mixed energy 1} (regarding $v+ (t-1)\phi$ and $tu$ as  elements of $\mathcal{E}(X,t\theta,t\phi)$) we have 
\begin{eqnarray*}
	\int_X \chi(\phi-v) (t \theta+ dd^c   tu)^n  &\leq &  2^{n+M}\int_X \chi(\phi-v) (t\theta+dd^c (v + (t-1)  \phi))^n\\
 &&+\int_X \chi(t\phi-tu) (t\theta+dd^c tu)^n\\
	&\leq & 2^{n+M}t^n \int_X \chi(\phi-v)\theta_\phi^n + t^{n-1}2^{2n+M} E_{\chi}(v,\phi)+ t^{n+M} E_{\chi}(u,\phi),  
\end{eqnarray*}
where the last inequality follows from Lemma \ref{lem: bound_mixed} after observing that $$(t\theta+dd^c (v + (t-1)  \phi))^n= ((t-1)\theta_\phi+\theta_v)^n \leq t^n \theta_\phi^n + t^{n-1} 2^n \sum_{j=1}^n \theta_v^j \wedge \theta_\phi^{n-j}.$$
 Dividing by $t^n$ and using Lemma \ref{lem: int_bound} we thus get 
 \[
 \int_X \chi(\phi-v) \theta_u^n  \leq C_1+ 2^{2n+M} t^{-1} E_{\chi}(v,\phi) + t^M E_{\chi}(u,\phi). 
 \]
Choosing $t= (1+ E_{\chi}(v,\phi))^{{1}/{(M+1)}}$, we finish the proof. 
\end{proof}

\begin{prop}\label{prop: exponential}
	Assume $\varphi \in \mathcal{E}(X,\theta,\phi)$ with $\sup_X \varphi =0$   satisfies $\theta_\varphi^n \leq A \theta_u^n$ and $[u]=[\phi]$ for some $A>0$ and $u \in {\rm PSH}(X,\theta)$. Then there exists $\alpha >0$ such that
	\[
	\int_X e^{\alpha (\phi-\varphi)} \theta_\varphi^n  <\infty. 
	\]
\end{prop}

\begin{proof}
We can assume that $\sup_X \varphi=0$ and $\phi-1\leq u\leq \phi$. We argue as in \cite[Lemma 4.3]{DnL17}.  Fix $t>0$ and $s>1$. Observe that away from the pluripolar set $\{\phi = -\infty\}$ we have that
$\{\varphi < \phi -t -s \}\subseteq \{\varphi < \phi -t +s(u-\phi) \}\subseteq \{\varphi < \phi -t \}$. By the assumption and the partial comparison principle (Proposition \ref{prop: general CP})  we have 
	\begin{flalign*}
	 \int_{\{\varphi < \phi -t -s\}} \theta_\varphi^n  & \leq A 	\int_{\{s^{-1}\varphi + (1- s^{-1})\phi <u -s^{-1}t\}}  \theta_u^n  \\
	&\leq A 	\int_{\{s^{-1}\varphi +(1-s^{-1})\phi <u -s^{-1}t \}} (s^{-1}\theta_{\varphi} + (1-s^{-1}) \theta_{\phi})^n \\
	&\leq A\int_{\{\varphi  <\phi -t \}} (s^{-1}\theta_{\varphi} + (1-s^{-1}) \theta_{\phi})^n \\
&\leq  A \int_{\{\varphi  <\phi -t \}}  \theta_{\phi}^n + A \sum_{k=1}^n s^{-k} \binom{n}{k}  \int_{\{\varphi<\phi-t\}}\theta_{\varphi}^k \wedge \theta_{\phi}^{n-k}\\
&\leq A C_0 \int_{\{\varphi <\phi-t\} } \omega^n + 2^n A s^{-1} \int_{\{\varphi<\phi-t\}} \theta_{\varphi}^n\\
&\leq C_2 e^{- at} \int_{\{\varphi <-t\} } e^{a |\varphi|} \omega^n + 2^n As^{-1} \int_{\{\varphi<\phi-t\}} \theta_{\varphi}^n\\
&\leq C_3 e^{- at} + 2^n As^{-1} \int_{\{\varphi<\phi-t\}} \theta_{\varphi}^n.
	\end{flalign*}
 In the fifth inequality we used Theorem \ref{thm: MA of env sing type}. The last inequality holds for a choice of $a>0$ very small, so that the uniform Skoda integrability theorem (see \cite{Sko72}, \cite{GZbook}) ensures 
	\[
	\int_X e^{-a \varphi} \omega^n \leq C_1
	\]
	  is uniformly bounded. 
  
  We fix $s$ so large that $2^nA s^{-1}< (2e)^{-1}$. Then, for   $a=1/s$ we have  $2^nA s^{-1} e^{a s} <1/2$.  Setting
	 \[
	 F(t):= e^{a t}\int_{\{\varphi<\phi-t\}} \theta_\varphi^n, \; t>0,
	 \]
  we then have 
	 \[ F(t+s) \leq C_4 + \frac{ F(t)}{2},\]
from which we obtain, by induction on $k\in \mathbb{N}$,  
  \[
  F(t+ks) \leq C_4 \sum_{j=1}^k 2^{1-j} + 2^{-k} F(t). 
  \]
Any $t\geq 0$ can be written as $t=t_0+ks$ for some $t_0\in[0,s]$ and $k\in \mathbb{N}$. The above estimate thus gives a uniform bound  $ F(t)\leq 2C_4+ \sup_{[0,s]} F\leq C_5$. Now, for $\alpha=a/2$, we get
\begin{flalign*}
    \int_X e^{\alpha (\phi-\varphi)} \theta_{\varphi}^n &= \int_0^{\infty} \alpha e^{\alpha t} \theta_{\varphi}^n(\{\varphi < \phi -t\})dt= \int_0^{\infty} \frac{a}{2}  e^{-{\frac{a}{2}}t}e^{a t} \theta_{\varphi}^n(\{\varphi < \phi -t\})dt,\\
    &= \int_0^{\infty} \frac{a}{2}  e^{-\frac{a}{2} t} F(t)dt \leq\frac{aC_5}{2}   \int_0^{\infty}  e^{-\frac{a}{2} t} dt < \infty.
\end{flalign*}
\end{proof}

\begin{lemma}
	\label{lem: pluripolar set}
If $E\subset X$ is a pluripolar set then there exists $u \in \mathcal{E}_{\chi}(X,\theta,\phi)$ such that $E\subset \{u=-\infty\}$. 
\end{lemma}
\begin{proof}
If  $\eta(t)= t^M, t\geq 0$ then the inclusion $\mathcal{E}_{\eta}(X,\theta,\phi)\subset \mathcal{E}_{\chi}(X,\theta,\phi)$ holds because $\chi(t)\leq \chi(1)t^M$ for all $t\geq 1$. We can thus assume that $\chi=\eta$ is convex. 
We claim that there exists $h\in \mathcal{E}^1(X,\theta,\phi)$ such that $E\subset \{h=-\infty\}$. Indeed, it follows from \cite[Corollary 2.11]{BBGZ13} that $E\subset \{\psi=-\infty\}$ for some $\psi\in \mathcal{E}^1(X,\theta)$. We can assume $\sup_X \psi =-1$, which implies $\psi\leq V_\theta$. Since $ P_{\theta}(V_\theta,\phi)=\phi$, Proposition \ref{prop: envelope mixed} ensures that the function $h:= P_{\theta}(\psi,\phi)$ belongs to $\mathcal{E}(X,\theta,\phi)$. By Theorem \ref{thm: MA of env sing type}  we have 
\[
\int_X (\phi-h) \theta_h^n \leq \int_{\{h=\psi\}} (\phi-\psi) \theta_\psi^n \leq \int_X (V_{\theta}-\psi) \theta_\psi^n<\infty. 
\]
We thus have that $h\in \mathcal{E}^1(X,\theta,\phi)$ and $E\subset \{\psi=-\infty\} \subset \{h=-\infty\}$, proving the claim. 

Now, let $ \gamma:=\chi^{-1}$ be the inverse of $\chi$ (so $\gamma(t)=t^{1/M}$ and $\gamma'(t)\leq 1$ if $t\geq 1$), which is a concave weight.  By adding a constant, we can assume that $h\leq -1$. Consider $u:=-\gamma(\phi-h) + \phi \in {\rm PSH}(X,\theta,\phi)$, as in Lemma \ref{ineq: forms_MA} (see also the discussion following the statement). 
We observe also that  $E\subset \{u=-\infty\}$ and $u\geq h$, hence $u\in \mathcal{E}(X,\theta,\phi)$.  By Lemma \ref{lem: mixed energy 2},  
\[
\int_X \chi(\phi-u) \theta_u^n =\int_X (\phi-h) \theta_u^n \leq  2^{n} E_1(h, \phi) <\infty. 
\] 
We thus have that $u\in \mathcal{E}_{\chi}(X,\theta,\phi)$, finishing the proof. 
\end{proof}

\begin{prop}
	\label{lem: DDL4 relative}
	Assume $\mu$ is a positive Radon measure  satisfying $\int_X \theta_{\phi}^n = \mu(X)>0$. Assume also that $\mu(\{\phi=-\infty\})=0$ and
\begin{equation}
	\label{eq: DDL4 relative}
	\int_X \chi(\phi-\varphi) d\mu \leq a E_{\chi}(\varphi,\phi) +C, \quad \varphi \in \mathcal{E}_{\chi}(X,\theta,\phi), \; \sup_X \varphi=0,
\end{equation}
for some constants $a\in (0,1)$, $C>0$. Then $\mu = (\theta+dd^c u)^n$ for some $u\in \mathcal{E}_{\chi}(X,\theta,\phi)$. 
\end{prop} 

\begin{proof}
	We prove the proposition by an argument going back to Cegrell \cite{Ceg98}. We first claim that $\mu$ vanishes on pluripolar sets. Indeed, fix such a Borel set $E$. It follows from Lemma \ref{lem: pluripolar set} that $E\subset \{h=-\infty\}$ for some $h\in \mathcal{E}_{\chi}(X,\theta,\phi)$. We can assume that $h \leq \phi$.
 
Since $\int_E \chi (\phi-h) d\mu \leq \int_X \chi (\phi-h) d\mu < \infty$, we get that $\mu(E \cap \{\phi \neq -\infty\}) =0$. Since $\mu(\{\phi = -\infty\})=0$, it follows that $\mu(E) =0$.
 
We next claim that we can write 
	\[
	\mu = f\nu, \; \nu = (\theta+dd^c \phi_0)^n, 
	\]
for some $\phi_0\in {\rm PSH}(X,\theta,\phi)$ such that $\phi -1\leq \phi_0\leq \phi$, and some $0\leq f\in  L^1(X,\nu)$. To see this, we first find $\psi\in \mathcal{E}(X,\theta,\phi)$ such that  $(\theta +dd^c \psi)^n = \mu$ and $\sup_X \psi=0$ (in particular $\psi \leq \phi$). The existence of $\psi$ follows from Theorem \ref{thm: existence lambda =0}.

We set $\phi_0:= e^{\psi-\phi}+\phi$ as in Lemma \ref{ineq: forms_MA} (see also the discussion following the statement). Then $\phi_0\in {\rm PSH}(X,\theta, \phi)$, $[\phi_0]=[\phi]$, and
\begin{eqnarray}\label{eq: MA_ineq}
	 (\theta +dd^c \phi_0)^n \geq  e^{n(\psi-\phi)}(\theta+dd^c \psi)^n = e^{n(\psi-\phi)} \mu. 
\end{eqnarray}

From \eqref{eq: MA_ineq} it follows that $\mu$ is absolutely continuous with respect to $(\theta +dd^c \phi_0)^n$, proving the second claim. 
	 
	 Now, for each $j>1$, let $\varphi_j \in \mathcal{E}(X,\theta,\phi)$ be the unique solution of 
\[
(\theta +dd^c \varphi_j)^n = c_j \min(f,j) (\theta+dd^c \phi_0)^n, \; \sup_X \varphi_j=0,
\] 
where $c_j$ is a normalization constant to have equality between the total masses of the left and right-hand side.
For $j$ large enough we have $c_j a <1$, since $c_j \to 1$ and by assumption $a<1$. We can thus assume that $c_ja <\lambda<1$. It follows from Proposition \ref{prop: exponential} that 
\[
\int_X e^{\alpha_j (\phi-\varphi_j)} (\theta+dd^c \varphi_j)^n  <\infty,
\] 
for some $\alpha_j >0$. In particular, since for any $t>0$ $\chi(t)\leq C e^{\alpha_j t}$ for some $C>0$, we infer that $E_{\chi}(\varphi_j, \phi)$ is finite. We claim that this bound is uniform in $j$.  Indeed, since $\theta_{\varphi_j}^n \leq c_j f \theta_{\phi_0}^n$, we have
\[
E_{\chi}(\varphi_j,\phi) \leq \int_X \chi(\phi-\varphi_j) c_j d \mu \leq ac_j E_{\chi}(\varphi_j,\phi) +C
\]
gives $E_{\chi}(\varphi_j,\phi) \leq C(1-\lambda)^{-1}$. Extracting a subsequence we can assume that $\varphi_j \to \varphi$ in $L^1$. It follows from Lemma \ref{lem: stable under L1 conv} that $\varphi \in \mathcal{E}_{\chi}(X,\theta,\phi)$. It then follows from Lemma \ref{lem: stability of subsolutions} that  $(\theta+dd^c \varphi)^n \geq \mu$. Comparing the total mass, we obtain the equality, finishing the proof.     
\end{proof}
We are now ready to prove our main result of this section.

\begin{theorem}\label{thm: E chi GZ07}
	Assume $\mu(\{\phi=-\infty\})=0$ and $\chi(|\phi-u|) \in L^1(\mu)$, for all $u\in \mathcal{E}_{\chi}(X,\theta,\phi)$. Then $\mu= (\theta+dd^c \varphi)^n$ for some $\varphi\in \mathcal{E}_{\chi}(X,\theta,\phi)$. 
\end{theorem} 
\begin{proof}
	Let $C$ be the constant in Lemma \ref{lem: linear growth}.  Let $v\in \mathcal{E}(X,\theta,\phi)$ be the unique solution to  
\[
(\theta+dd^c v)^n = (4C)^{-1} \mu + b \omega^n,\; \sup_X v=-1.   
\]
Here $b>0$ is a constant so that $\int_X (4C)^{-1} \mu + b \omega^n = \int_X \theta_\phi^n$. 
The existence of $v$ follows from Theorem \ref{thm: existence lambda =0}.  By Lemma \ref{lem: linear growth} there exists a constant $C_1>0$ such that, for all $\varphi \in \mathcal{E}_{\chi}(X,\theta,\phi)$ with $\sup_X \varphi=0$, we have 
\[
\int_X \chi(\phi-\varphi) \theta_v^n \leq \frac{1}{4} E_{\chi}(\varphi,\phi) +C_1.
\]
By Proposition \ref{lem: DDL4 relative}, $v\in \mathcal{E}_{\chi}(X,\theta,\phi)$. Since $\mu \leq 4C (\theta+dd^c v)^n$, it follows from Lemma \ref{lem: DV21 relative} that  $\mu$ satisfies 
$$
\int_X \chi(\phi-u) \mu \leq  4C \int_X \chi(\phi-u) \theta_v^n\leq C'(1 +E_{\chi}(v,\phi)) E_{\chi}(u,\phi)^{M/(M+1)} +C'.  
$$
As a result, $\mu$ also satisfies \eqref{eq: DDL4 relative}. We can thus use Lemma \ref{lem: DDL4 relative} to complete the proof.     
\end{proof}

We summarize the findings of this chapter in the theorem below.

\begin{theorem}
	\label{thm: E chi characterization}
	 Fix a Radon measure $\mu$ with $\mu(\{\phi=-\infty\})=0$ and  $\int_X \theta_{\phi}^n =\mu(X)>0$. Then the following are equivalent. 
	\begin{enumerate}
		\item[(i)] There exists a constant $C>0$ such that, for all $u \in \mathcal{E}_{\chi}(X,\theta,\phi)$ with $\sup_X u=0$, we have 
		\[
		\int_X \chi(\phi-u) d\mu \leq C E_{\chi}(u,\phi)^{M/(M+1)} +C.  
		\]
	\item[(ii)] $\chi(|\phi-u|) \in L^1(\mu)$, for all $u\in \mathcal{E}_{\chi}(X,\theta,\phi)$.
	\item[(iii)] $\mu= (\theta+dd^c \varphi)^n$ for some $\varphi \in \mathcal{E}_{\chi}(X,\theta,\phi)$, with $\sup_X \varphi=0$. 
	\end{enumerate} 
\end{theorem}

\begin{proof}
	(i) $ \Longrightarrow$ (ii) is immediate. The implication (ii) $\Longrightarrow$ (iii) is Theorem \ref{thm: E chi GZ07}.  The implication  (iii) $\Longrightarrow$ (i)  is Lemma \ref{lem: DV21 relative}. 
 \end{proof}

\begingroup
\setlength\bibitemsep{0pt}
\setlength\biblabelsep{0pt}
\printbibliography
\endgroup
\end{document}